\documentclass[10pt,a4paper]{amsart}

\usepackage{amsmath,amsfonts,amssymb,amsthm}
\usepackage[a4paper,inner=2.5cm,outer=2.5cm,top=3cm,bottom=3cm,pdftex]{geometry}
\usepackage{array}
\usepackage{graphicx,color}
\usepackage[colorlinks=true]{hyperref}
\usepackage{here}
\usepackage{amsfonts}

\newtheorem{theorem}{Theorem}[section]

\newtheorem{definition}{Definition}[section]

\newtheorem{corollary}[definition]{Corollary}
\newtheorem{lemma}[definition]{Lemma}
\newtheorem{remark}[definition]{Remark}

\newcommand{\psca}[1]{\left\langle#1\right\rangle}
\newcommand{\abs}[1]{\left\vert#1\right\vert}
\newcommand{\set}[1]{\left\{#1\right\}}
\newcommand{\norm}[1]{\left\Vert#1\right\Vert}
\newcommand{\pare}[1]{\left(#1\right)}
\newcommand{\eps}{\varepsilon}

\newcommand{\LEO}{black}

\newcommand{\SAID}{black}

\numberwithin{equation}{section} 

\begin{document}

\title{Hydrostatic limit of the Navier-Stokes-alpha model}
\author{{Glangetas L\'eo, Ngo Van-Sang, Said El Mehdi}}
\address{Universit\'e de Rouen Normandie, CNRS UMR 6085, Math\'ematiques, 76801 Saint-Etienne du Rouvray, France}



 \email{leo.glangetas@univ-rouen.fr,  van-sang.ngo@univ-rouen.fr, el-mehdi.said@etu.univ-rouen.fr}

\date{\today} 
\maketitle



\begin{abstract}
In this paper we study the hydrostatic limit of the Navier-Stokes-alpha model in a very thin striped domain. We derive some Prandtl-type limit equations for this model and we prove the global well-posedness of the limit system for small initial conditions in an appropriate analytic function space. 
\end{abstract}


\section{Introduction}

\subsection{Motivation} The characteristic feature of a turbulent fluid 
according to Kolmogorov's theory is that the energy cascades from large scales to small scales 
until it reaches the dissipation scale and then turns into heat. 
This feature leads to important costs of calculations in numerical simulation 
because the grid resolutions cannot keep up to the dissipation scale, which is extremely small 
when the Reynolds number is large (which corresponds to turbulent flows). 
So the idea is to consider the effects of smaller scales on larger scales 
instead of capturing all scales and as a consequence, one can achieve a balance between computational costs and precision.

One way to model turbulent flows is the so-called Large Eddy Simulation method. It consists in filtering the small scales and directly calculating the large scales of the turbulent cascade (see \cite {Gh99} for instance). 
Another approach is the Reynolds Averaged Navier-Stokes, based on Reynolds decomposition, which provides mean quantities of turbulent flow while fluctuations will be modeled. This second method is used in the industry due to its small computational costs. 
However, both approaches meet a common problem which is the closure of the model systems where there are more unknowns than equations.

In order to overcome this difficulty, the Navier-Stokes-alpha model was introduced, where an energy ``penalty'' inhibits the creation of small excitations below a certain length scale $\alpha$ (also called the viscous Camassa-Holm equations, see \cite{CH93, CFHOTW98, HMR98, K99, MS2001, M98} and the references therein for a survey of Camassa-Holm equations). 
This ``alpha-modification'' leads to a change in the convection term of the Navier-Stokes equations. More precisely, this is the following system
\begin{equation} \label{ANS-alpha}
	\left\{
	\begin{aligned}
		&\frac{\partial \textbf{v}}{\partial t} + (\textbf{u}\cdot\nabla)\textbf{v}+v_1\nabla u_1 + v_2\nabla u_2 =\nu \Delta \textbf{v}-\nabla q 
		\\
		&\nabla \cdot \textbf{u}=0,
	\end{aligned}
	\right.
\end{equation}
where
\begin{align*}
	\textbf{v}=(v_1,v_2)&=(1-\alpha^2 \Delta)\textbf{u} = ((1-\alpha^2 \Delta)u_1,(1-\alpha^2 \Delta)u_2)
\end{align*}
and  $\textbf{u}$ is the velocity of the fluid and $q$ the modified pressure. 
In \cite{FHT2002}, Foias, Holm and Titi prove the global existence and uniqueness 
of the solution of $(\ref{ANS-alpha})$ in a periodic domain 
for $H^1$-initial data and the convergence of the solution of 3D Camassa-Holm equation (NS-$\alpha$) towards a weak solution of 3D NS equations 
when $\alpha$ tends to zero and in \cite{Bu2002}, Busuioc gives a simple proof of the global existence and uniqueness of $(\ref{ANS-alpha})$ for $H^1_{0}$-initial data in bounded domains with Dirichlet boundary conditions. 


\medskip

In this paper, we consider the system \eqref{ANS-alpha} in the thin strip $\mathcal{S}^{\varepsilon}=\left\{(x,y)\in \mathbb{R}^2,0<y<\varepsilon \right\}$, where the width $\eps$ is supposed to be very small. This consideration is relevant for planetary-scale oceanic and atmospheric flows (see \cite{Pe1987}), for which, the vertical scale (a few kilometers for oceans, 10-20 kilometers for the atmosphere) is much smaller than the horizontal scales (thousands of kilometers). In this framework, the fluid behaviors are approximated by the so-called hydrostatic model, in which the conservation momentum in the vertical direction is replaced by a simple hydrostatic equation. In the case of Navier-Stokes equations for a viscous fluid in a thin strip, the hydrostatic limit leads to the following rescaled system in domain $\mathbb{R}\times]0,1[$ 
\begin{equation}\label{HNS}
	\left\{
	\begin{aligned}
		&\partial_{t} u+u\,\partial_{x} u+v\partial_z u = \partial^2_y u-\partial_{x}p
		\\
		&\partial_{y}p=0
		\\
		&\partial_{x} u+\partial_y v =0
		\\ 
		&(u,v)\vert_{y=0}= (u,v)\vert_{y=1}=0\\
		& u|_{t=0}=u_0. 
	\end{aligned}
	\right.
\end{equation}
This model and its three-dimensional counterpart are very important in oceanography and meteorology (see \cite{B94, LTW92, Pe1987}). 
It is known that without any structural assumption on the initial data, real-analyticity is both necessary \cite{R2009} 
and sufficient \cite{KTVZ2011} for the local well-posedness of the system \eqref{HNS}. 
In \cite{GMV2019} the authors proved that for convex initial data, the local well-posedness holds under simple Gevrey regularity. 
In \cite{PZZ2020}, Paicu \emph{et al.}  proved the existence and uniqueness of global analytic solution with small analytic initial data  with respect to variable $x$. We want to emphasize the similarity of the system \eqref{HNS} and the classical Prandtl system. For the classical Prandtl, the well-posedness also requires analyticity frameworks (see \cite{samm}) or monotonicity hypotheses to be able to work in classical Sobolev frameworks (see for instance \cite{AWXY2015, MW2015}).

In this paper, we will study the hydrostatic limit of the Navier-Stokes-alpha equations as the strip width $\eps \to 0$. 
Using the techniques of \cite{PZZ2020}, we prove the global well-posedness of the limit system in appropriate analytic function spaces. 
We will justify this limit by proving the convergence of the Navier-Stokes-alpha system in a forthcoming paper. 
In the next subsection, we will give brief derivation of the hydrostatic limit of the system \eqref{ANS-alpha} as $\eps \to 0$.

\subsection{Hydrostatic limit of the Navier-Stokes-alpha model} 
We consider the system \eqref{ANS-alpha} in a thin strip
\begin{align*}
	\mathcal{S}^{\varepsilon}=\left\{(x,y)\in \mathbb{R}^2,0<y<\varepsilon \right\},
\end{align*}
where the width $\varepsilon > 0$ is supposed to be very small. Equipped with no-slip boundary conditions, we rewrite \eqref{ANS-alpha} as
\begin{equation}\label{alpha_NS_strp}
	\left\{
	\begin{aligned}
		&\partial_t v^\eps_1+ u^\eps_1\,\partial_x v^\eps_1+u^\eps_2 \,\partial_y v^\eps_1 + v^\eps_1\,\partial_xu^\eps_1+v^\eps_2\,\partial_xu^\eps_2 = \nu \Delta v^\eps_1 -\partial_x q^\eps
		\\
		&\partial_t v^\eps_2+ u^\eps_1\,\partial_x v^\eps_2 + u^\eps_2 \,\partial_y v^\eps_2 + v^\eps_1\,\partial_y  u^\eps_1 + v^\eps_2\,\partial_yu^\eps_2 = \nu \Delta v^\eps_2 -\partial_y q^\eps
		\\
		&\partial_x u^\eps_1+\partial_y u^\eps_2 =0
		\\ 
		&u^\eps_1 = u^\eps_2 = 0\quad\textsl{on}\quad \set{y=0} \cup \set{y=\varepsilon}
		\\
		&\partial_z u^\eps_1 = \partial_z u^\eps_2 = 0\quad\textsl{on}\quad \set{y=0} \cup \set{y=\varepsilon}\\
		&(u^\eps_1,u^\eps_2)_{|t=0}=(u_{1,0},u_{2,0})(x,y),
	\end{aligned}
	\right.
\end{equation}
where, for $i = 1,2$, 
\begin{align*}
	v^\eps_i = u^\eps_i-\alpha^2 \Delta u^\eps_i.
\end{align*} 	
%
We consider the following rescaling
\begin{align*}
u^\eps_1(t,x,y)&=u^\eps\left(t,x,\frac{y}{\varepsilon}\right),
\\
u^\eps_2(t,x,y)&=\varepsilon\,v^\eps\left(t,x,\frac{y}{\varepsilon}\right),
\\
q^\eps(t,x,y)&=\tilde{q}^\eps\left(t,x,\frac{y}{\varepsilon}\right),
\end{align*}
we perform the change of variable $z=\frac{y}{\varepsilon}$ and we set
\begin{align*}
	\mathcal{S}=\left\{(x,z)\in \mathbb{R}^2,0<z<1 \right\}.
\end{align*} 
We remark that for any function $\varphi(t,x,y) = \psi\pare{t,x,\frac{y}{\eps}}$, we have
\begin{align*}
	\partial_y^{(k)} \pare{\varphi(t,x,y)} = \eps^{-k} \partial_z^{(k)} \psi\pare{t,x,z}.  
\end{align*}
Then the following calculations are immediate
\begin{equation*}
	\left\{
	\begin{aligned}
		v^\eps_1(t,x,y) &= u^\eps(t,x,z)-\alpha_1^2\partial_z^2 u^\eps(t,x,z)- \varepsilon^2 \alpha_1^2 \partial_x^2u^\eps(t,x,z)\\
		v^\eps_2(t,x,y) &= \varepsilon ( v^\eps(t,x,z) -\alpha_1^2  \partial_z^2 v^\eps(t,x,z))- \varepsilon^3\, \alpha_1^2 \partial_x^2v^\eps(t,x,z), 
	\end{aligned}
	\right.
\end{equation*}
where $\alpha_1=\frac{\alpha}{\varepsilon}$. 
The system \eqref{alpha_NS_strp} becomes 
\begin{equation*}
\left\{
	\begin{aligned}
		&\partial_{t}\omega^\eps 
		- \alpha_1^2 \varepsilon^2 \partial_t\partial_x^2u^\eps 
		+ u^\eps\,\partial_{x}\omega^\eps 
		- \alpha_1^2 \varepsilon^2 u^\eps\,\partial_x^3 u^\eps + v^\eps\partial_z \omega^\eps - \alpha_1^2 \varepsilon^2 v^\eps\,\partial_z \partial_x^2u^\eps + \omega^\eps \partial_x u^\eps - \alpha_1^2 \varepsilon^2\,\partial_x u^\eps \,\partial_x^2u^\eps 
		\\
		&\hspace{2cm} + \varepsilon^2 \gamma^\eps \partial_x v^\eps - \varepsilon^4\,\alpha_1^2 \partial_x^2 v^\eps\partial_x v^\eps = \partial^2_z \omega^\eps + \varepsilon^2 \partial^2_{x} \omega^\eps - \varepsilon^2 \,\alpha_1^2 \partial^2_z \partial_x^2 u^\eps - \alpha_1^2 \varepsilon^4 \partial_x^4 u^\eps - \partial_x \tilde{q}^\eps
		\\
		&\varepsilon^2 \partial_{t} \gamma^\eps - \alpha_1^2 \varepsilon^4 \partial_t\partial_x^2v^\eps + \varepsilon^2\, u^\eps\partial_{x} \gamma^\eps - \alpha_1^2 \varepsilon^4 \,u^\eps\,\partial_x^3v^\eps + \varepsilon^2 v^\eps\,\partial_z \gamma^\eps - \alpha_1^2 \varepsilon^4 v^\eps\,\partial_z \partial_x^2v^\eps + \omega\,\partial_z u^\eps-\alpha_1^2 \varepsilon^3 \partial_x^2u^\eps\,\partial_z u^\eps 
		\\
		&\hspace{2cm} + \varepsilon^2 \gamma^\eps\,\partial_z v^\eps - \alpha_1^2 \varepsilon^4 \partial_x^2v^\eps \partial_z v^\eps = \varepsilon^2 \partial^2_z \gamma^\eps + \varepsilon^4\,\partial^2_{x} \gamma^\eps - \alpha_1^2 \varepsilon^4\,\partial^2_z \partial_x^2 v^\eps - \alpha_1^2 \varepsilon^6\, \partial_x^4 v^\eps - \partial_z \tilde{ q}^\eps
		\\
		&\partial_{x} u^\eps + \partial_z v^\eps = 0, 
	\end{aligned}
\right.
\end{equation*}
where we denote $\omega^\eps = u^\eps - \alpha_1^2 \partial_z^2 u^\eps$ and  $\gamma^\eps = v^\eps - \alpha_1^2 \partial_z^2 v^\eps$ 
to lighten the notations.

Formally taking $\varepsilon \to 0$ we obtain 
\begin{equation}\label{aNSG_Strip_bis}
	\left\{
	\begin{aligned}
		&\partial_{t} (u-\alpha_1^2\partial_z^2 u)+u\,\partial_{x} (u-\alpha_1^2\partial_z^2 u) 
		+ v\partial_z (u-\alpha_1^2\partial_z^2 u)+(u-\alpha_1^2\partial_z^2 u) \partial_x u 
		= \partial^2_z (u-\alpha_1^2\partial_z^2 u)-\partial_{x}\tilde{q}
		\\
		&(u-\alpha_1^2\partial_z^2 u)\partial_z u=-\partial_{z} \tilde{q}
		\\
		&\partial_{x} u+\partial_y v =0
		\\ &(u,v)=0, \quad \partial_z u=0 \quad\textsl{on}\quad z=0
		\\
		&(u,v)=0, \quad \partial_zu=0 \quad \textsl{on} \quad z=1\\
		&u_{|t=0}=u_0.
	\end{aligned}
	\right.
\end{equation}
%
We consider the modified pressure in the strip
\begin{equation*}
	p=\tilde{q}+\frac{1}{2} u^2- \frac{1}{2}\alpha_1^2 (\partial_z u)^2 
\end{equation*}
and we set
\begin{equation*}
	\omega=u-\alpha_1^2\partial_z^2 u. 
\end{equation*}
Then we can rewrite the system \eqref{aNSG_Strip_bis} as 
\begin{equation}\label{ANSG_Strip}
	\left\{
	\begin{aligned}
		&\partial_{t}\omega+u\,\partial_{x}\omega+v\partial_z\omega 
		+ \alpha_1^2(\partial_z u\,\partial_x \partial_z u-\partial^2_z u\partial_x u) 
		= \partial^2_z \omega-\partial_{x}p
		\\
		&\partial_{z}p=0
		\\
		&\partial_{x} u+\partial_y v =0 
		\\ &(u,v)=0, \quad \partial_z u=0 \quad\textsl{on}\quad z=0
		\\
		&(u,v)=0, \quad \partial_zu=0 \quad \textsl{on} \quad z=1\\
		& u|_{t=0}=u_0. 
	\end{aligned}
	\right.
\end{equation}
\begin{remark}\label{rem incompressible}
	\begin{enumerate}
		\item The no-slip boundary conditions $(u,v)|_{y=0}=(u,v)|_{y=1}=0$ and the incompres\-sibility $\partial_{x} u+\partial_y v =0$ imply 
		\begin{equation*}
			v(t,x,y)=\int_0^y \partial_{\tilde{y}} v(t,x,\tilde{y})\,d\tilde{y}=-\int_0^y \partial_x u(t,x,\tilde{y})\,d\tilde{y}.
		\end{equation*}
	
		\item From the incompressibility condition we deduce that
		\begin{align*}
			\partial_x \int_0^1 u(t,x,y)\,dy =-\int_0^1 \partial_y v(t,x,y)\,dy = v(t,x,1)-v(t,x,0) = 0,
		\end{align*}
		which together with the fact that $u(t,x,y) \to 0 \, \text{ as } \, |x| \to +\infty$, ensures that
		\begin{equation} \label{verticalaverage}
			\int_0^1 u(t,x,y)\,dy=0.
		\end{equation} 
		Since $\partial_z p=0$, integrating the first equation of system \eqref{ANSG_Strip} and using the boundary conditions, we get 
		\begin{equation*}
			\partial_x p(t,x)= \alpha_1^2 \partial^3_z u|_{z=0}-\alpha_1^2 \partial^3_z u|_{z=1}-\partial_x \int_0^1 u^2(t,x,y)\,dz
			-\alpha_1^2 \partial_x \int_0^1 (\partial_zu)^2(t,x,y)\,dz. 
		\end{equation*}
	\end{enumerate}
\end{remark}

\bigskip 

We emphasize that, similar to Prandtl equation, the nonlinear term $v \partial_z \omega$ in the system \eqref{ANSG_Strip} 
creates the loss of one derivative in $x$ variable in energy-type estimates. 
So, in order to overcome this difficulty, it is natural to work in analytic function frameworks. 
In the next paragraph, we will introduce some elements of Littlewood-Paley theory 
and functional spaces that we are going to use in this paper.

\subsection{Littlewood-Paley theory and functional spaces}

We consider a even smooth function $\chi$ in $C^{\infty}_0(\mathbb{R})$ such that the support is contained in the ball $B_{\mathbb{R}}(0,\frac{4}3)$ and $\psi$ is equal to 1 on a neighborhood of the ball $B_{\mathbb{R}}(0,\frac{3}4)$. 
We set $\psi(z) = \chi\pare{\frac{z}2} - \chi(z)$ then the support of $\psi$ is contained in the ring $\set{z \in \mathbb{R}: \frac{3}4 \leq \abs{z} \leq \frac{8}3}$, and $\psi$ is identically equal to 1 on the ring $\set{z \in \mathbb{R}: \frac{4}3 \leq \abs{z} \leq \frac{3}2}$. 
Moreover, the functions $\chi$ and $\psi$ verify the following important properties
\begin{align*}
	&\sum_q \psi (2^{-q}z)=1,  \quad  \forall z \in \mathbb{R}^* , 
	\\
	& \chi(z)+\sum_{q \geq 0} \psi (2^{-q}z)=1,  \quad  \forall \eta \in \mathbb{R}
\end{align*} 
and
\begin{equation*}
	\forall\, j,j' \in \mathbb{N}, \; \abs{j-j'}\geq 2, \quad \text{supp}\, \psi(2^{-j} \cdot) \cap \text{supp}\, \psi(2^{-j'} \cdot) = \emptyset.
\end{equation*} 
Let $\mathcal{F}_h$ and $\mathcal{F}^{-1}_h$ be the Fourier transform and the inverse Fourier transform respectively in the horizontal direction. We will also use the notation $\widehat{f} =\mathcal{F}_h f$. We introduce the following definitions
of the homogeneous dyadic cut-off operators.

\begin{definition}\label{def Dh}
 	For all tempered distributions in the horizontal direction and for all $q \in \mathbb{Z}$ we set
 	\begin{align*}
  		\Delta^h_q u(x,y)&= \mathcal{F}^{-1}_h (\psi(2^{-q} |\xi|)\mathcal{F}_h u(\xi,y)),
  		\\
  		S^h_q u(x,y)& = \mathcal{F}^{-1}_h (\chi(2^{-q} |\xi|)\mathcal{F}_h u(\xi,y)) = \sum_{q'\leq q-1} \Delta^h_{q'} u . 
 	\end{align*}
\end{definition} 
In our paper, we will use the same functional spaces as in \cite{PZZ2020}, the definition of which is given in what follows. 
\begin{definition}
Let $s\in \mathbb{R}$ and $\mathcal{S}=\mathbb{R}\times ]0,1[$. For any $u \in S'_h(\mathcal{S})$, i.e, 
$$u \in S'(\mathcal{S}) \quad\mbox{with}\quad \lim_{q \to -\infty} \|S^h_q u \|_{L^{\infty}}=0,$$ 
we set
\begin{equation*}
		\|u\|_{B^s}=\sum_{q \in \mathbb{Z}} 2^{qs} \|\Delta^h_q u \|_{L^2}.		
\end{equation*}
\begin{itemize}
\item[•] For $s \leq \frac{1}{2}$, we define
	\begin{equation*}
		B^s( \mathcal{S})= \left \{ u \in S'_h(\mathcal{S}), \quad \|u\|_{B^s} <\infty \right \}.
	\end{equation*}
\item[•] For $s \in ]k-\frac{1}{2}, k+\frac{1}{2} ]$, with $k \in \mathbb{N}^*$, 
		we define $B^s(\mathcal{S})$ as the subset of distributions $u \in S'_h(\mathcal{S})$ 
		such that $\partial_x^k u \in B^{s-k} (\mathcal{S})$.
\end{itemize}
\end{definition}
\begin{remark}\label{rem dq}
For $u\in B^s$, there exists a summable sequence of positive numbers ${d_q(u)}$ with $\sum_{q} d_q(u)=1$, such that 
	\[ \|\Delta^h_q u \|_{L^2} \lesssim d_q(u) 2^{-qs} \|u\|_{B^s} . \]
\end{remark}

We also need the following time-weighted Chemin-Lerner-type spaces (see \cite{CL1992}). 
\begin{definition}
	Let $p \in [1,+\infty]$ and $T \in ]0, \infty]$
	and let $f\in L^1_{loc}(\mathbb{R}^+)$ be non-negative function.
	We define space $\widetilde{L}^p_{T,f}(B^s(\mathcal{S}))$ as the closure of $C([0,T],B^s(\mathcal{S}))$ with respect to the norm
	\begin{equation*}
		\|u\|_{\widetilde{L}^p_{T,f}(B^s(\mathcal{S})}=\sum_{q \in \mathbb{Z}} 2^{qs} \left(\int_0^T f(t)  \|\Delta^h_q u \|^p_{L^2} dt \right)^{\frac{1}{p}}
	\end{equation*}
	with the usual change if  $p=+\infty$.
	In the case of $f(t)\equiv 1$, we will simply use the notations $\widetilde{L}^p_{T}(B^s(\mathcal{S}))$ and $\|u\|_{\widetilde{L}^p_{T}(B^s(\mathcal{S}))}$. 
\end{definition}
\begin{remark}\label{rem dqf}
For $u\in {\widetilde{L}^p_T(B^s)}$, there exists a summable sequence of positive numbers ${d_q(u,f)}$ with $\sum_{q} d_q(u,f)=1$, such that
\[ \|\Delta^h_q u \|_{L^p_T(L^2)} \lesssim d_q(u,f) 2^{-pqs} \|u\|_{\widetilde{L}^p_T(B^s)}. \]
\end{remark}

\subsection{Main results}

The main difficulty consists in controlling the nonlinear terms, using the smoothing effect given by the above function spaces. 
The idea is to define the following auxiliary functions, using the method introduced 
by Chemin in \cite{C2004} (see also \cite{CGP2011} or \cite{PZZ2020}): 
for any $f \in L^2(\mathcal{S})$, we set 
\begin{equation}\label{auxiliary_functions}
	f_{\phi}(t,x,y)=e^{\phi(t,D_x)}\,f(t,x,y)
	:=\mathcal{F}^{-1}_{\xi \to x}(e^{\phi(t,\xi)}\,\widehat{f}(t,\xi,y)) 
	\quad\text{with}\quad \phi(t,\xi)=(a-\lambda\theta(t))|\xi|,
\end{equation}
where the quantity $\theta(t)$ describes the evolution of the analytic band of u, satisfies
\begin{equation}\label{Analytic_Bande}
	\theta'(t)=\|\partial_z^2 u_{\phi}(t)\|_{B^{\frac{1}{2}}}  \quad \text{and} \quad \theta(0)=0.
\end{equation}
We remark that if we differentiate a function of the type $e^{\phi(t,D_x)}\,f(t,x,y)$  with respect to the time variable, we have
\begin{equation*}
	\partial_t(e^{\phi(t,D_x)}\,f(t,x,y))=-\lambda \theta'(t)|D_x|e^{\phi(t,D_x)}\,f(t,x,y)+e^{\phi(t,D_x)}\partial_t f(t,x,y), 
\end{equation*}
where $-\lambda \theta'(t)|D_x|e^{\phi(t,D_x)}\,f(t,x,y)$ plays the role of a ``smoothing term'', 
provided that $\theta'(t) >0$, and allows to ``absorb'' the loss of one derivative in the nonlinear terms when we perform energy-type estimates.

Our main result is the folowing.
\begin{theorem} \label{mainth}
Let $a>0$ be fixed. 
We assume that  $e^{a|D_x|}u_0 \in {B^{\frac{3}{2}}}$ and $e^{a|D_x|} \partial_z u_0 \in B^{\frac{3}{2}}$. 
There exist positive constants $c$ and $C$ such that, if we suppose that 
\begin{equation} \label{DataB12}
	\|e^{a|D_x|}u_0 \|_{B^\frac{1}{2}}+\,  \|e^{a|D_x|} \partial_z u_0\|_{B^\frac{1}{2}} 
	\leq 	\frac{c\, a}{1+ \|e^{a|D_x|}u_0 \|_{B^{\frac{3}{2}}}+  \|e^{a|D_x|} \partial_z u_0\|_{B^{\frac{3}{2}}}}
\end{equation}
and the compatibility condition $\int_0^1 u_0 dz=0$, then a unique global solution for the system \eqref{ANSG_Strip} exists and satisfies
\begin{multline} \label{Estimate B3/2}
	\|e^{Rt} u_\phi \|_{\widetilde{L}^{\infty}(\mathbb{R}^+, B^{\frac{3}{2}})} 
	+ \|e^{Rt}  \partial_z u_\phi\|_{\widetilde{L}^{\infty}(\mathbb{R}^+, B^{\frac{3}{2}})} 
	+ \|e^{Rt}  \partial_z u_\phi\|_{\widetilde{L}^2(\mathbb{R}^+, B^{\frac{3}{2}})} 
	+ \|e^{Rt}  \partial^2_z u_\phi\|_{\widetilde{L}^2(\mathbb{R}^+, B^{\frac{3}{2}})} 
	\\
	\leq C(\|e^{a|D_x|}u_0 \|_{B^{\frac{3}{2}}}+\,  \|e^{a|D_x|} \partial_z u_0\|_{B^{\frac{3}{2}}}).
\end{multline}
Furthermore, if 
$e^{a|D_x|}u_0 \in  B^{\frac{5}{2}} $, 
$e^{a|D_x|} \partial_z u_0 \in B^{\frac{5}{2}}$, 
$e^{a |D_x|} \partial^2_z u_0 \in B^{\frac{3}{2}}$, 
then
\begin{multline}\label{time_estimate_Bs}
	\|e^{Rt} (\partial_t u)_\phi \|^2_{\widetilde{L}^2(\mathbb{R}^+, B^{\frac{3}{2}})} + \|e^{Rt} (\partial_t \partial_z u)_\phi \|^2_{\widetilde{L}^2(\mathbb{R}^+, B^{\frac{3}{2}})} + \|e^{Rt} \partial_z u_\phi \|^2_{\widetilde{L}^{\infty}(\mathbb{R}^+, B^{\frac{3}{2}})} + \|e^{Rt} \partial^2_z u_\phi \|^2_{\widetilde{L}^{\infty}(\mathbb{R}^+, B^{\frac{3}{2}})}
	\\
	\leq C \Big( \| e^{a |D_x|} \partial_z u_0 \|_{B^{\frac{3}{2}}} + \| e^{a |D_x|} \partial^2_z u_0 \|_{B^{\frac{3}{2}}}+\|e^{a|D_x|}u_0\|_{B^\frac{5}{2}} + \|e^{a|D_x|} \partial_z u_0\|_{B^\frac{5}{2}}\Big). 
\end{multline}
\end{theorem}
{\color{\SAID}
\begin{remark}
The compatibility condition  $\int_0^1 u_0\, dz=0$ ensures that $\int_0^1 u(t,\cdot)\,dz$
is conserved by the system \eqref{ANSG_Strip}.
\end{remark}}

The rest of the paper is arranged as follows : 
The proof of the main Theorem \ref{mainth} is presented in the next Section \ref{Sect dem th} and the Appendices \ref{app lem:nonlinear}, \ref{app nonlin} and \ref{app estim-Iq}
are devoted to the proof of estimates used in Section \ref{Sect dem th}.


\section{Global existence and uniqueness of the hydrostatic limit system}\label{Sect dem th}
	
In this section, we prove the existence of a unique global solution of the system \eqref{ANSG_Strip} 
using the method introduced by Chemin in \cite{Chemin2004}. 
We recall that, to be able to deal with the lost of one derivative in the tangential direction $x$, 
we work in the same functional settings as in \cite{PZZ2020}.

\subsection{Energy-type \emph{a priori} estimates} 
	
The aim of the first part of this section is to prove the energy-type estimate \eqref{Estimate B3/2}. 
Our proof is based on the following important estimates of the nonlinear terms. We recall that $\phi$, $\theta$ are the auxiliary functions defined as in \eqref{auxiliary_functions} and \eqref{Analytic_Bande} and we set
\[T^* =\sup \left\{t>0, \quad \theta(t)< \frac{a}{ \lambda} \right\}. \]
	
\begin{lemma} \label{lem:nonlinear}
	Let $s>0$, $0<T<T^*$,  $R>0$. There exist a generic constant $C\geq 1$
	and some square root summable positive sequences ($\sum_q\tilde{d}_q^\frac12 =1$, $\sum_q\check{d}_q^\frac12 = 1$) such that 
		\begin{gather} 
			\label{estimate_1}
			\int_0^T \abs{\psca{e^{Rt'}\Delta^h_q(u\partial_x u)_\phi,e^{Rt'} \Delta_q^h u_\phi}}\,dt' \leq 
			C 2^{-2qs}\,(\tilde{d}_q+\check{d}_q)\,\|e^{Rt'}u_\phi\|^2_{\widetilde{L}^2_{T,\theta'(t)}(B^{s+\frac{1}{2}})}, \\
			\label{estimate_2}
			\int_0^T \abs{\psca{e^{Rt'} \Delta^h_q (u\partial_x \partial_z u)_\phi,e^{Rt'} \Delta^h_q \partial_zu_{\phi} }}\,dt' 
			\leq C 2^{-2qs}\,(\tilde{d}_q+\check{d}_q)\, \|e^{Rt'}\partial_zu_\phi\|^2_{\widetilde{L}^2_{T,\theta'(t)}(B^{s+\frac{1}{2}})}, \\
			\label{estimate_3}
			\int_0^T \abs{\psca{e^{Rt'}\Delta^h_q(\partial_z u\partial_x  u)_\phi,e^{Rt'} \Delta_q^h  \partial_z u_\phi}}\,dt' 
			\leq C  2^{-2qs} \,(\tilde{d}_q+\check{d}_q)\,\|e^{Rt'}\partial_z u_\phi\|^2_{\widetilde{L}^2_{T,\theta'(t)}(B^{s+\frac{1}{2}})}, \\
			\label{estimate_4}
			\int_0^T \abs{\psca{e^{Rt'}\Delta^h_q(\partial_z u\partial_x \partial_z u)_\phi,e^{Rt'} \Delta_q^h u_\phi}}\,dt' 
			\leq C  2^{-2qs}\,(\tilde{d}_q+\check{d}_q)\, \|e^{Rt'} \partial_z u_\phi\|^2_{\widetilde{L}^2_{T,\theta'(t)}(B^{s+\frac{1}{2}})}, \\
			\label{estimate_5}
			\int_0^T \abs{\psca{e^{Rt'}\Delta^h_q(v\partial_z u)_\phi,e^{Rt'} \Delta_q^h u_\phi}}\,dt' 
			\leq C 2^{-2qs}\,(\tilde{d}_q+\check{d}_q)\, \|e^{Rt'} \partial_z u_\phi\|^2_{\widetilde{L}^2_{T,\theta'(t)}(B^{s+\frac{1}{2}})}
\end{gather}
and
\begin{multline} \label{estimate_6}
			\int_0^T \abs{\psca{e^{Rt'}\Delta^h_q(v\partial^2_z u)_\phi,e^{Rt'} \Delta_q^h \partial_z u_\phi}}\,dt'\\
			\leq C 2^{-2qs} \,(\tilde{d}_q+\check{d}_q)\,\|e^{Rt'}\partial _z u_\phi\|_{\widetilde{L}^2_{T,\theta'(t)}(B^{s+\frac{1}{2}})}\left(\|e^{Rt'}\partial _z u_\phi\|_{\widetilde{L}^2_{T,\theta'(t)}(B^{s+\frac{1}{2}})}+\|\partial_z  u_\phi\|^{\frac{1}{2}}_{L^{\infty}_T({B^{\frac{3}{2}}})}\, \|e^{Rt'} \partial^2_z u_\phi\|_{\widetilde{L}^2_{T}(B^s)} \right).
\end{multline}
\end{lemma}

The proof of this lemma is given in appendix  \ref{app lem:nonlinear}. 
For $(u,v)$ solution of \eqref{ANSG_Strip}, we define $(u_{\phi},v_{\phi})$ 
be as in \eqref{auxiliary_functions} and \eqref{Analytic_Bande}). Direct calculations show that $(u_{\phi},v_{\phi})$ satisfies the following system
\begin{equation}\label{ANSG_phi}
\left\{
	\begin{aligned}
			&\partial_{t}\omega_\phi+ \lambda \theta'(t) |D_x| \omega_{\phi}+(u\,\partial_{x}\omega)_{\phi}
			+(v\partial_z\omega)_{\phi}+\alpha_1^2(\partial_z u\,\partial_x \partial_z u)_{\phi}-\alpha_1^2(\partial^2_z u\partial_x u)_{\phi} = \partial^2_z \omega_{\phi}-\partial_{x}p_{\phi}
			\\
			&\partial_{z}p_{\phi}=0
			\\
			&\partial_{x} u_{\phi}+\partial_y v_{\phi} =0 
			\\ &(u_\phi,v_\phi)=0, \quad \partial_z u_\phi=0 \quad\textsl{on}\quad z=0
			\\
			&(u_\phi,v_\phi)=0, \quad \partial_zu_\phi=0 \quad \textsl{on} \quad z=1\\
			&{u_\phi}|_{t=0}=e^{a|D_x|}u_0,
	\end{aligned}
\right.
\end{equation}
where $|D_x|$ denotes the Fourier multiplier of symbol $|\xi|$. For any $q \in \mathbb{Z}$, 
applying the dyadic operator $\Delta^h_q$ to the first equation of the system \eqref{ANSG_phi} 
and taking $L^2(S)$ inner product of the obtained equation with $\Delta^h_q u_\phi$, we get
\begin{multline} \label{Energy_Equality}
		\frac{1}{2 }\frac{d}{dt} (\|\Delta^h_q u_\phi \|^2_{L^2}+\alpha_1^2 \|\Delta_q^h \partial_z u_\phi\|^2_{L^2})+\lambda \theta'(t) \left \||D_x|^{\frac{1}{2}}\Delta^h_q u_\phi \right\|^2_{L^2}+\alpha_1^2\,\lambda \theta'(t) \left \||D_x|^{\frac{1}{2}}\Delta^h_q\partial_z  u_\phi \right\|^2_{L^2}\\	
		+\|\Delta_q^h \partial_z u_\phi\|^2_{L^2}+\alpha_1^2 \,\|\Delta_q^h \partial^2_z u_\phi\|^2_{L^2}
		= - \psca{\Delta^h_q (u \partial_x \omega)_\phi,\Delta^h_q u_\phi} - \psca{\Delta^h_q (v \partial_z \omega)_\phi,\Delta^h_q u_\phi} \qquad \quad\\ 
		- \alpha_1^2 \psca{\Delta^h_q(\partial_z u\,\partial_x \partial_z u)_{\phi},\Delta^h_q u_\phi} +\alpha_1^2 \psca{\Delta^h_q(\partial^2_z u\partial_x u)_{\phi},\Delta^h_q u_\phi}.
\end{multline}
Multiplying \eqref{Energy_Equality} by $e^{2Rt}$ and integrating with respect to the time variable, we obtain
\begin{multline*}
\|e^{Rt}\Delta^h_q u_\phi
		\|^2_{L^2} + \alpha_1^2 \|e^{Rt} \Delta_q^h \partial_z u_\phi\|^2_{L^2}+2\lambda \int_0^t \theta'(t') \left \|e^{Rt} |D_x|^{\frac{1}{2}}\Delta^h_q u_\phi(t') \right\|^2_{L^2}\,dt' 
		+ 2 \int_0^t \left \|e^{Rt} \Delta^h_q \partial_z u_\phi(t') \right\|^2_{L^2}\,dt'
		\\
		 +2\lambda\,\alpha_1^2\int_0^t \theta'(t') \left \|e^{Rt} |D_x|^{\frac{1}{2}}\Delta^h_q \partial_z u_\phi(t') \right\|^2_{L^2}\,dt' 
		 +2 \alpha_1 ^2\int_0^t \left \|e^{Rt} \Delta^h_q \partial^2_z u_\phi(t') \right\|^2_{L^2}\,dt'
		\\
		=\|\Delta^h_q u_\phi(0)
		\|^2_{L^2}+ \alpha_1^2 \|\Delta_q^h \partial_z u_\phi(0)\|^2_{L^2}+2D_{1,q}+2D_{2,q}+2D_{3,q}+2D_{4,q} , 
\end{multline*}
where
\begin{align*}
		D_{1,q}=& -\int_0^t \psca{e^{Rt'}\Delta^h_q (u \partial_x \omega)_\phi,e^{Rt'}\Delta^h_q u_\phi} dt',
		\\
		D_{2,q}=&-\int_0^t \psca{e^{Rt'}\Delta^h_q (v \partial_z \omega)_\phi,e^{Rt'}\Delta^h_q u_\phi} dt',
		\\
		D_{3,q}=&- \alpha_1^2 \int_0^t \psca{e^{Rt'}\Delta^h_q(\partial_z u\,\partial_x \partial_z u)_{\phi}, e^{Rt'} \,\Delta^h_q u_\phi}\, dt',
		\\
		D_{4,q}=& \alpha_1^2 \int_0^t \psca{e^{Rt'}\Delta^h_q(\partial^2_z u\partial_x u)_{\phi}, e^{Rt'} \Delta^h_q u_\phi}. 
\end{align*}
Using the definition of $\omega$, we can write
\begin{align*}
	D_{1,q} = -\int_0^t \psca{e^{Rt'} \Delta^h_q(u \partial_xu)_\phi,e^{Rt'} \Delta^h_q u_{\phi}} t' 
	+ \alpha_1^2\int_0^t \psca{e^{Rt'}\Delta^h_q (u\partial_x \partial_z^2 u)_\phi,e^{Rt'} \Delta^h_q u_{\phi}} dt' . 
\end{align*}
Since
\begin{align*}
		\Delta^h_q(u\partial_x \partial_z^2 u)_\phi=\partial_z \Delta^h_q(u\partial_x \partial_z u)_\phi-\Delta^h_q(\partial_z u \partial_x \partial_z u)_\phi,
\end{align*}
by integration by parts with respect to z variable, we have
\begin{align*}
		\int_0^t& \psca{e^{Rt'}\Delta^h_q (u\partial_x \partial_z^2 u)_\phi,e^{Rt'} \Delta^h_q u_{\phi}} dt'\\
		&=\int_0^t \psca{e^{Rt'} \partial_z\Delta^h_q (u\partial_x \partial_z u)_\phi,e^{Rt'} \Delta^h_q u_{\phi}} dt'
		-\int_0^t \psca{e^{Rt'}\Delta^h_q (\partial_z u\partial_x \partial_z u)_\phi,e^{Rt'} \Delta^h_q u_{\phi}} dt'\\
		&=-\int_0^t \psca{e^{Rt'} \Delta^h_q (u\partial_x \partial_z u)_\phi,e^{Rt'} \Delta^h_q \partial_zu_{\phi}} dt'
		-\int_0^t \psca{e^{Rt'}\Delta^h_q (\partial_z u\partial_x \partial_z u)_\phi,e^{Rt'} \Delta^h_q u_{\phi}} dt'.
\end{align*}
Thus 
\begin{align*}
		D_{1,q} = A_q+B_q+D_{3,q}, 
\end{align*}
where
\begin{align*}
		A_q&=-\int_0^t \psca{e^{Rt'}\Delta^h_q(u\partial_x u)_\phi,e^{Rt'} \Delta_q^h u_\phi}\,dt', 
		\\
		B_q &=-\alpha_1^2\int_0^t \psca {e^{Rt'} \Delta^h_q (u\partial_x \partial_z u)_\phi,e^{Rt'} \Delta^h_q \partial_zu_{\phi} } dt' . 
\end{align*}
In a similar way, we can write 
\begin{equation*}  
		D_{4,q}=C_q+D_{3,q} , 
\end{equation*}
where
\begin{align*}
		C_q=-\alpha_1^2 \int_0^t \psca {e^{Rt'}\Delta^h_q(\partial_z u\partial_x  u)_\phi,e^{Rt'} \Delta_q^h  \partial_z u_\phi }\,dt',
\end{align*}
and
\begin{align*}
		D_{2,q}=E_q+F_q+ C_q+D_{3,q}
\end{align*}
and where
\begin{align*}
		E_q&=-\int_0^t \psca {e^{Rt'}\Delta^h_q(v\partial_z u)_\phi,e^{Rt'} \Delta_q^h u_\phi }\,dt' ,
		\\
		F_q &=-\alpha_1^2 \int_0^t \psca {e^{Rt'}\Delta^h_q(v\partial^2_z u)_\phi,e^{Rt'} \Delta_q^h \partial_z u_\phi }\,dt' .
\end{align*}
To summarize, we have
\begin{multline}  \label{Dyadic}
		\|e^{Rt}\Delta^h_q u_\phi \|^2_{L^2}+ \alpha_1^2 \|e^{Rt} \Delta_q^h \partial_z u_\phi\|^2_{L^2} + 2\lambda\,\alpha_1^2\int_0^t \theta'(t') \left \|e^{Rt'} |D_x|^{\frac{1}{2}}\Delta^h_q \partial_z u_\phi(t') \right\|^2_{L^2}\,dt'
		\\
		+2\lambda \int_0^t \theta'(t') \left \|e^{Rt'} |D_x|^{\frac{1}{2}}\Delta^h_q u_\phi(t') \right\|^2_{L^2}\,dt' +2\int_0^t \left \|e^{Rt'} \Delta^h_q \partial_z u_\phi(t') \right\|^2_{L^2}\,dt'
		\\
		+2 \alpha_1 ^2\int_0^t \left \|e^{Rt'} \Delta^h_q \partial^2_z u_\phi(t') \right\|^2_{L^2}\,dt' 
		\\
		=\|\Delta^h_q u_\phi(0) \|^2_{L^2}+ \alpha_1^2 \|\Delta_q^h \partial_z u_\phi(0)\|^2_{L^2}+2A_q+2B_q+2C_q+8D_{3,q}+2E_q+2 F_q , 
\end{multline}
where 
\begin{align*}
		A_q&=-\int_0^t \psca {e^{Rt'}\Delta^h_q(u\partial_x u)_\phi,e^{Rt'} \Delta_q^h u_\phi }\,dt' ,
		\\
		B_q &=-\alpha_1^2\int_0^t \psca {e^{Rt'} \Delta^h_q (u\partial_x \partial_z u)_\phi,e^{Rt'} \Delta^h_q \partial_zu_{\phi} }\, dt' ,
		\\
		C_q &=-\alpha_1^2 \int_0^t \psca {e^{Rt'}\Delta^h_q(\partial_z u\partial_x  u)_\phi,e^{Rt'} \Delta_q^h  \partial_z u_\phi }\,dt' ,
		\\
		D_{3,q} & =- \alpha_1^2 \int_0^t \psca{ e^{Rt'}\Delta^h_q(\partial_z u\,\partial_x \partial_z u)_{\phi}, e^{Rt'} \,\Delta^h_q u_\phi }\, dt' ,
		\\
		E_q&=-\int_0^t \psca {e^{Rt'}\Delta^h_q(v\partial_z u)_\phi,e^{Rt'} \Delta_q^h u_\phi }\,dt' , 
		\\
		F_q &=-\alpha_1^2 \int_0^t \psca {e^{Rt'}\Delta^h_q(v\partial^2_z u)_\phi,e^{Rt'} \Delta_q^h \partial_z u_\phi }\,dt.
\end{align*}
We recall that we set \[T^* =\sup \left\{t>0, \; \theta(t)< \frac{a}{ \lambda} \right\}. \]
Lemma \ref{lem:nonlinear} yields that, for any $0 < T < T^*$, we have 
\begin{multline*}
		\|e^{Rt}\Delta^h_q u_\phi \|^2_{L_T^{\infty}(L^2)} + \alpha_1^2 \|e^{Rt} \Delta_q^h \partial_z u_\phi\|^2_{L_T^{\infty}(L^2)} 
		+\lambda\,\alpha_1^2\,2^{q}\,\int_0^T \theta'(t') \left \|e^{Rt'}\Delta^h_q \partial_z u_\phi(t') \right\|^2_{L^2}\,dt'
		\\
		+\alpha_1 ^2\int_0^T \left \|e^{Rt'} \Delta^h_q \partial^2_z u_\phi(t') \right\|^2_{L^2}\,dt' 
		+\int_0^T \left \|e^{Rt'} \Delta^h_q \partial_z u_\phi(t') \right\|^2_{L^2}\,dt' 
		+\lambda\,2^{q} \int_0^T \theta'(t') \left \|e^{Rt'}\Delta^h_q u_\phi(t') \right\|^2_{L^2}\,dt'
		\\
		\leq \|\Delta^h_q u_\phi(0) \|^2_{L^2}+ \alpha_1^2 \|\Delta_q^h \partial_z u_\phi(0)\|^2_{L^2}
		+ C 2^{-2qs}d^2_q\,\|e^{Rt'}u_\phi\|^2_{\widetilde{L}^2_{T,\theta'(t)}(B^{s+\frac{1}{2}})}
		\\
		+C 2^{-2qs} \,d_q^2\,\|e^{Rt'}\partial _z u_\phi\|_{\widetilde{L}^2_{T,\theta'(t)}(B^{s+\frac{1}{2}})} 
		\left(
		\|e^{Rt'}\partial _z u_\phi\|_{\widetilde{L}^2_{T,\theta'(t)}(B^{s+\frac{1}{2}})} 
		+\|\partial_z  u_\phi\|^{\frac{1}{2}}_{L^{\infty}_T({B^{\frac{3}{2}}})}\, 
		\|e^{Rt'} \partial^2_z u_\phi\|_{\widetilde{L}^2_{T}(B^s)} 
		\right).
\end{multline*}
Multiplying the previous inequality by $2^{2qs}$, 
taking square root of resulting inequality 
and then summing with respect to $q \in \mathbb{Z}$, we get
\begin{align*}
		\|e^{Rt} u_\phi &\|_{\widetilde{L}_T^{\infty}(B^s)} + \alpha_1^2 \|e^{Rt}  \partial_z u_\phi\|_{\widetilde{L}_T^{\infty}(B^s)} +\sqrt{\lambda}\,\|e^{Rt'}u_\phi\|_{\widetilde{L}^2_{T,\theta'(t)}(B^{s+\frac{1}{2}})} +\sqrt{\lambda}\,\alpha_1\,\|e^{Rt'}\partial _z u_\phi\|_{\widetilde{L}^2_{T,\theta'(t)}(B^{s+\frac{1}{2}})}
		\\
		+& \|e^{Rt}  \partial_z u_\phi\|_{\widetilde{L}_T^{2}(B^s)} +\alpha_1 \|e^{Rt}  \partial^2_z u_\phi\|_{\widetilde{L}_T^{2}(B^s)}
		\\
		\leq & 6\|u_\phi(0) \|_{B^s}+6\, \alpha_1 \| \partial_z u_\phi(0)\|_{B^s}+ C \,\|e^{Rt'}u_\phi\|_{\widetilde{L}^2_{T,\theta'(t)}(B^{s+\frac{1}{2}})}\,+C \alpha_1\,\|e^{Rt'}\partial _z u_\phi\|_{\widetilde{L}^2_{T,\theta'(t)}(B^{s+\frac{1}{2}})}
		\\
		& +C \|\partial_z  u_\phi\|^{\frac{1}{4}}_{L^{\infty}_T({B^{\frac{3}{2}}})}\, \|e^{Rt'} \partial^2_z u_\phi\|^{\frac{1}{2}}_{\widetilde{L}^2_{T}(B^s)}\, \|e^{Rt'} \partial_z u_\phi\|^{\frac{1}{2}}_{\widetilde{L}^2_{T,\theta'(t)}(B^{s+\frac{1}{2}})}. 
\end{align*}
Therefore, there exists a constant $C_1$ such that
\begin{align*}
		\|e^{Rt} u_\phi &\|_{\widetilde{L}_T^{\infty}(B^s)} + \|e^{Rt}  \partial_z u_\phi\|_{\widetilde{L}_T^{\infty}(B^s)} +\sqrt{\lambda}\,\|e^{Rt'}u_\phi\|_{\widetilde{L}^2_{T,\theta'(t)}(B^{s+\frac{1}{2}})} +\sqrt{\lambda}\,\|e^{Rt'}\partial _z u_\phi\|_{\widetilde{L}^2_{T,\theta'(t)}(B^{s+\frac{1}{2}})}
		\\
		+& \|e^{Rt}  \partial_z u_\phi\|_{\widetilde{L}_T^{2}(B^s)} + \|e^{Rt}  \partial^2_z u_\phi\|_{\widetilde{L}_T^{2}(B^s)}
		\\
		\leq & C_1(\|u_\phi(0) \|_{B^s}+\| \partial_z u_\phi(0)\|_{B^s})+ C_1 \,\|e^{Rt'}u_\phi\|_{\widetilde{L}^2_{T,\theta'(t)}(B^{s+\frac{1}{2}})}\,+C_1 \,\|e^{Rt'}\partial _z u_\phi\|_{\widetilde{L}^2_{T,\theta'(t)}(B^{s+\frac{1}{2}})}
		\\
		& +C_1 \|\partial_z  u_\phi\|^{\frac{1}{4}}_{L^{\infty}_T({B^{\frac{3}{2}}})}\, \|e^{Rt'} \partial^2_z u_\phi\|^{\frac{1}{2}}_{\widetilde{L}^2_{T}(B^s)}\, \|e^{Rt'} \partial_z u_\phi\|^{\frac{1}{2}}_{\widetilde{L}^2_{T,\theta'(t)}(B^{s+\frac{1}{2}})} . 
\end{align*}
We remark that Young's inequality implies 
\begin{align*}
		C_1 \|\partial_z  u_\phi\|^{\frac{1}{4}}_{L^{\infty}_T({B^{\frac{3}{2}}})}\, \|e^{Rt'} \partial^2_z u_\phi\|^{\frac{1}{2}}_{\widetilde{L}^2_{T}(B^s)}\, &\|e^{Rt'} \partial_z u_\phi\|^{\frac{1}{2}}_{\widetilde{L}^2_{T,\theta'(t)}(B^{s+\frac{1}{2}})}
		\\
		& \leq C_2 \|\partial_z  u_\phi\|^{\frac{1}{2}}_{L^{\infty}_T({B^{\frac{3}{2}}})}\,\|e^{Rt'} \partial_z u_\phi\|_{\widetilde{L}^2_{T,\theta'(t)}(B^{s+\frac{1}{2}})}+\frac{1}{2}\|e^{Rt'} \partial^2_z u_\phi\|_{\widetilde{L}^2_{T}(B^s)}. 
\end{align*}
So, putting $C_3 = \max(C_1,C_2)$ and choosing
\begin{align}\label{condition}
		\lambda \geq C_3^2 \, (1+ \|\partial_z  u_\phi\|_{L^{\infty}_T({B^{\frac{3}{2}}})}),
\end{align}
we have for any $s>0$
\begin{multline} \label{Estimate_B^s}
		\|e^{Rt} u_\phi \|_{\widetilde{L}_T^{\infty}(B^s)} + \|e^{Rt}  \partial_z u_\phi\|_{\widetilde{L}_T^{\infty}(B^s)} 
		+ \|e^{Rt}  \partial_z u_\phi\|_{\widetilde{L}_T^{2}(B^s)} + \|e^{Rt}  \partial^2_z u_\phi\|_{\widetilde{L}_T^{2}(B^s)}
		\\
		\leq C_3(\|e^{a|D_x|}u_0 \|_{B^s}+\,  \|e^{a|D_x|} \partial_z u_0\|_{B^s}). \hspace{1cm}
\end{multline}
In particular, under the condition \eqref{condition}, we have
\begin{align*}
		\| \partial_z u_\phi\|_{\widetilde{L}_T^{\infty}(B^{\frac{3}{2}})} 
		\leq C_3 (\|e^{a|D_x|}u_0 \|_{B^{\frac{3}{2}}}+\,  \|e^{a|D_x|} \partial_z u_0\|_{B^{\frac{3}{2}}}) . 
\end{align*}
Then, by  taking 
$\lambda=C_3^2(1+ C_3(\|e^{a|D_x|}u_0
	\|_{B^{\frac{3}{2}}} +  
	\|e^{a|D_x|} \partial_z u_0\|_{B^{\frac{3}{2}}}))$,  
the inequality \eqref{Estimate_B^s} holds for any $s>0$.
In particular for $s=\frac{1}{2}$, we have 
\begin{align}\label{estim ut u0}
		\|e^{Rt} u_\phi \|_{\widetilde{L}_T^{\infty}(B^\frac{1}{2})}
		+ &\|e^{Rt}  \partial_z u_\phi\|_{\widetilde{L}_T^{\infty}(B^\frac{1}{2})}
		+\|e^{Rt}  \partial_z u_\phi\|_{\widetilde{L}_T^{2}(B^\frac{1}{2})}
		+\|e^{Rt}  \partial^2_z u_\phi\|_{\widetilde{L}_T^{2}(B^\frac{1}{2})}  \notag
		\\
		& \leq C_3 (\|e^{a|D_x|}u_0
		\|_{B^\frac{1}{2}}+\,  \|e^{a|D_x|} \partial_z u_0\|_{B^\frac{1}{2}}) . 
\end{align}

\medskip
	
Now we want to estimate the function $\theta(t)$ on $[0,T^*[$. Using Cauchy-Schwarz inequality, we can write 
\begin{align*}
		\theta(t)
		&=\int_0^t e^{-Rt'} \times e^{Rt'}\|\partial^2_z u_\phi(t')\|_{B^{\frac{1}{2}}} dt'
		\\
		& \leq \left (\int_0^ T e^{-2Rt'} dt'\right)^{\frac{1}{2}}\,
		\left (\int_0^ T \|e^{Rt'}\partial^2_z u_\phi(t')\|^2_{B^{\frac{1}{2}}} dt'\right)^{\frac{1}{2}}\,
		\\
		& \leq C_R \,\|e^{Rt}  \partial^2_z u_\phi\|_{\widetilde{L}_T^{2}(B^\frac{1}{2})}.
\end{align*}
Then, we deduce from \eqref{estim ut u0} the existence of a constant $C_4 > 0$ such that
\begin{align*}
		\theta(t) \leq C_4 \Big(\|e^{a|D_x|}u_0
		\|_{B^\frac{1}{2}}+\,  \|e^{a|D_x|} \partial_z u_0\|_{B^\frac{1}{2}} \Big) . 
\end{align*}
Taking $c$ small enough (see the assumption on the initial data \eqref{DataB12}), we obtain, for any $t \in [0, T^*[$ that
\begin{equation*} 
		\theta(t) \leq C_4 \Big(\|e^{a|D_x|}u_0
		\|_{B^\frac{1}{2}}+\,  \|e^{a|D_x|} \partial_z u_0\|_{B^\frac{1}{2}} \Big) \leq \frac{a}{2\lambda}.
\end{equation*}
Using the definition of $T^*$ and a continuity argument, we deduce $T^*=+\infty$ and that the inequality \eqref{Estimate_B^s} holds for any $T>0$.\\ 

{\sl Conclusion }: 
Taking $s=\frac{3}{2}$ in \eqref{Estimate_B^s}, we obtain Estimate \eqref{Estimate B3/2} in Theorem \ref{mainth}.

\subsection{Estimate of the time derivative  of the solution} 
In this paragraph, we give brief ideas of the proof of Estimate \eqref{time_estimate_Bs}. 
Applying $\Delta^h_q$ to \eqref{ANSG_phi} and taking the $L^2$ inner product of the obtained equation with $e^{2Rt} \Delta^h_q (\partial_t u)_\phi$, we have
\begin{align} \label{equ221}
		\|e^{Rt} \Delta^h_q (\partial_t u)_\phi \|^2_{L^2}&+\alpha_1^2\|e^{Rt} \Delta^h_q (\partial_t \partial_z u)_\phi \|^2_{L^2}\\
		=&e^{2Rt} \psca {\Delta^h_q \partial_z^2 u_\phi, \Delta^h_q (\partial_t u)_\phi }-\,\alpha_1^2 e^{2Rt} \psca {\Delta^h_q \partial_z^4 u_\phi, \Delta^h_q (\partial_t u)_\phi }\notag
		\\
		&-e^{2Rt} \psca {\Delta^h_q(u\,\partial_{x}\omega)_{\phi},\Delta^h_q (\partial_t u)_\phi } -\,e^{2Rt} \psca {\Delta^h_q(v\partial_z\omega)_{\phi},\Delta^h_q (\partial_t u)_\phi } \notag
		\\
		&-e^{2Rt}\alpha_1^2 \psca {\Delta^h_q(\partial_z u\,\partial_x \partial_z u)_{\phi},\Delta^h_q (\partial_t u)_\phi }
		+\,e^{2Rt}\alpha_1^2 \psca {\Delta^h_q(\partial^2_z u\partial_x u)_{\phi},\Delta^h_q (\partial_t u)_\phi} . \notag
\end{align}
Since $ (\partial_t u)_\phi=\partial_t u_\phi+\lambda \, \theta'(t) |D_x| u_\phi $, then
\begin{align} \label{equ222}
		e^{2Rt} \psca { \Delta^h_q \partial_z^2 u_\phi, \Delta^h_q (\partial_t u)_\phi } 
		&= e^{2Rt} \psca {\Delta^h_q \partial_z^2 u_\phi, \Delta^h_q \partial_t u_\phi }
		+\theta'(t) |D_x|e^{2Rt} \psca {\Delta^h_q \partial_z^2 u_\phi, \Delta^h_q u_\phi}
		\\
		&= -\frac{1}{2} \frac{d}{dt} \|e^{Rt} \Delta^h_q \partial_z u_\phi \|^2_{L^2}- \lambda \, \theta'(t)\|e^{Rt} |D_x|^{\frac{1}{2}} \Delta^h_q \partial_z u_\phi \|^2_{L^2} \notag
		\\
		& \leq -\frac{1}{2} \frac{d}{dt} \|e^{Rt} \Delta^h_q \partial_z u_\phi \|^2_{L^2}.\notag
\end{align}
Similar calculations give
\begin{align} \label{equ223}
		-\, e^{2Rt} \psca {\Delta^h_q \partial_z^4 u_\phi, \Delta^h_q (\partial_t u)_\phi } \, \leq -   \frac{1}{2}  \frac{d}{dt} \|e^{Rt} \Delta^h_q \partial^2_z u_\phi \|^2_{L^2}.
\end{align}
Plugging the estimates \eqref{equ222} and \eqref{equ223} into \eqref{equ221}, we get
\begin{align*}
		\|e^{Rt} \Delta^h_q (\partial_t u)_\phi \|^2_{L^2}&+\alpha_1^2\|e^{Rt} \Delta^h_q (\partial_t \partial_z u)_\phi \|^2_{L^2}+\frac{1}{2} \frac{d}{dt} \|e^{Rt} \Delta^h_q \partial_z u_\phi \|^2_{L^2}+\frac{1}{2} \alpha_1^2 \frac{d}{dt} \|e^{Rt} \Delta^h_q \partial^2_z u_\phi \|^2_{L^2}\\
		\leq 
		&-e^{2Rt} \psca {\Delta^h_q(u\,\partial_{x} u)_{\phi},\Delta^h_q (\partial_t u)_\phi>+\alpha_1^2 e^{2Rt}<\Delta^h_q(u\,\partial_{x} \partial^2_z u)_{\phi},\Delta^h_q (\partial_t u)_\phi }
		\\&-\,e^{2Rt} \psca {\Delta^h_q(v\partial_z u)_{\phi},\Delta^h_q (\partial_t u)_\phi>- \alpha_1^2 e^{2Rt}<\Delta^h_q(v\partial^2_z u)_{\phi},\Delta^h_q (\partial_t \partial_z u)_\phi }\\
		&-e^{2Rt}\alpha_1^2 \psca{ \Delta^h_q(\partial_z u\,\partial_x \partial_z u)_{\phi},\Delta^h_q (\partial_t u)_\phi} +2\,e^{2Rt}\alpha_1^2 \psca {\Delta^h_q(\partial^2_z u\partial_x u)_{\phi},\Delta^h_q (\partial_t u)_\phi }
\end{align*}
and so 
\begin{align*}
		\|e^{Rt} \Delta^h_q (\partial_t u)_\phi \|^2_{L^2}&+\alpha_1^2\|e^{Rt} \Delta^h_q (\partial_t \partial_z u)_\phi \|^2_{L^2}+\frac{1}{2} \frac{d}{dt} \|e^{Rt} \Delta^h_q \partial_z u_\phi \|^2_{L^2}+\frac{1}{2} \alpha_1^2 \frac{d}{dt} \|e^{Rt} \Delta^h_q \partial^2_z u_\phi \|^2_{L^2}\\
		\leq 
		& \; C \Big (\| e^{Rt}\Delta^h_q(u\,\partial_{x} u)_{\phi}\|^2_{L^2} +\, \|e^{Rt} \Delta^h_q(u\partial_x \partial^2_z u)_{\phi} \|^2_{L^2} +\|e^{Rt} \Delta^h_q(v\partial_z u)_{\phi} \|^2_{L^2}
		\\
		& +\|e^{Rt} \Delta^h_q(v\partial^2_z u)_{\phi} \|^2_{L^2}
		+ \| e^{Rt}\Delta^h_q(\partial_z u\,\partial_x \partial_z u)_{\phi} \|^2_{L^2} +\, \|e^{Rt} \Delta^h_q(\partial^2_z u\partial_x u)_{\phi} \|^2_{L^2} \Big ) . 
\end{align*}
Integrating with respect to time yields
\begin{align}\label{time_estimate}
		&\|e^{Rt} \Delta^h_q (\partial_t u)_\phi \|^2_{L^2_T({L^2})}+\|e^{Rt} \Delta^h_q (\partial_t \partial_z u)_\phi \|^2_{L^2_T({L^2})}+ \|e^{Rt} \Delta^h_q \partial_z u_\phi \|^2_{L^\infty_T({L^2})}+ \|e^{Rt} \Delta^h_q \partial^2_z u_\phi \|^2_{L^\infty_T({L^2})}\\  
		& \hspace{2cm} \leq C \Big ( \| e^{a |D_x|}\Delta^h_q \partial_z u_0 \|^2_{L^2} +\|e^{a |D_x|} \Delta^h_q \partial^2_z u_0 \|^2_{L^2} +\| e^{Rt}\Delta^h_q(u\,\partial_{x} u)_{\phi}\|_{L^2_T(L^2)} \notag 
		\\
		&\hspace{2.5cm}  +\|e^{Rt} \Delta^h_q(u\partial_x \partial^2_z u)_{\phi} \|_{L^2_T(L^2)} +\|e^{Rt} \Delta^h_q(v\partial_z u)_{\phi} \|_{L^2_T(L^2)}+\|e^{Rt} \Delta^h_q(v\partial^2_z u)_{\phi} \|_{L^2_T(L^2)} \notag
		\\
		&\hspace{2.5cm}  + \| e^{Rt}\Delta^h_q(\partial_z u\,\partial_x \partial_z u)_{\phi} \|_{L^2_T(L^2)} +\, \|e^{Rt} \Delta^h_q(\partial^2_z u\partial_x u)_{\phi} \|_{L^2_T(L^2)} \Big ). \notag
\end{align}
To be able to obtain Estimate \eqref{time_estimate_Bs}, we need the following lemma, the proof of which will be given in the appendix \ref{app nonlin}.
\begin{lemma} \label{law product}
We have the following estimates for the nonlinear terms : 
\begin{gather} 
	\label{Es1}
	\| e^{Rt}(u\,\partial_{x} u)_{\phi}\|_{\widetilde{L}_T^2(B^{\frac{3}{2}})} 
	\lesssim
	\|\partial_z u_\phi\|_{\widetilde{L}_T^\infty(B^\frac{1}{2})}  
	\|e^{Rt} \partial_z u_\phi\|_{\widetilde{L}_T^2(B^\frac{5}{2})} , \\
	\label{Es2}
	\|e^{Rt}(u\,\partial_x\partial^2_{z} u)_{\phi}\|_{\widetilde{L}_T^2(B^{\frac{3}{2}})}
	\lesssim
	\|\partial_z u_\phi\|_{\widetilde{L}_T^\infty(B^\frac{1}{2})}  
	\|e^{Rt} \partial^2_z u_\phi\|_{\widetilde{L}_T^2(B^\frac{5}{2})} 
	  + \|\partial_z u_\phi\|_{\widetilde{L}_T^\infty(B^\frac{5}{2})} 
	  \|e^{Rt} \partial^2_z u_\phi\|_{\widetilde{L}_T^2(B^\frac{1}{2})} , \\
	\label{Es3}
	\|e^{Rt}(v\,\partial_{z} u)_{\phi}\|_{\widetilde{L}_T^2(B^{\frac{3}{2}})}
	\lesssim
	\|\partial_z u_\phi\|_{\widetilde{L}_T^\infty(B^\frac{1}{2})}   
    \|e^{Rt} \partial_z u_\phi\|_{\widetilde{L}_T^2(B^\frac{5}{2})} , \\
	\label{Es4}
	\|e^{Rt}(v\,\partial^2_{z} u)_{\phi}\|_{\widetilde{L}_T^2(B^{\frac{3}{2}})}
	\lesssim
	\|\partial_z u_\phi\|_{\widetilde{L}_T^\infty(B^\frac{1}{2})}  
	\|e^{Rt} \partial^2_z u_\phi\|_{\widetilde{L}_T^2(B^\frac{5}{2})}
	  + \|\partial_z u_\phi\|_{\widetilde{L}_T^\infty(B^\frac{5}{2})} 
	  \|e^{Rt} \partial^2_z u_\phi\|_{\widetilde{L}_T^2(B^\frac{1}{2})} , \\
	\label{Es5}
	\|e^{Rt} (\partial_z u\,\partial_x \partial_{z} u)_{\phi}\|_{\widetilde{L}_T^2(B^{\frac{3}{2}})}
	\lesssim
	\|\partial_z u_\phi\|_{\widetilde{L}_T^\infty(B^\frac{1}{2})}  
	\|e^{Rt} \partial^2_z u_\phi\|_{\widetilde{L}_T^2(B^\frac{5}{2})}
	  + \|\partial_z u_\phi\|_{\widetilde{L}_T^\infty(B^\frac{5}{2})} 
	  \|e^{Rt} \partial^2_z u_\phi\|_{\widetilde{L}_T^2(B^\frac{1}{2})} , \\
	\label{Es6}
	\|e^{Rt}\Delta^h_q(\partial_x u\,\partial^2_{z} u)_{\phi}\|_{\widetilde{L}_T^2(B^{\frac{3}{2}})}
	\lesssim
	\|\partial_z u_\phi\|_{\widetilde{L}_T^\infty(B^\frac{1}{2})}  
	\|e^{Rt} \partial^2_z u_\phi\|_{\widetilde{L}_T^2(B^\frac{5}{2})}
	  + \|\partial_z u_\phi\|_{\widetilde{L}_T^\infty(B^\frac{5}{2})} 
	  \|e^{Rt} \partial^2_z u_\phi\|_{\widetilde{L}_T^2(B^\frac{1}{2})}.
\end{gather}
\end{lemma}

{\sl Conclusion}  :
Multiplying  \eqref{time_estimate} by $2^{\frac{3q}{2}}$, summing up with respect to $q \in \mathbb{Z}$, 
and then using Lemma \ref{law product}, assumption on initial data \eqref{DataB12} and Estimate \eqref{Estimate_B^s}, 
we  obtain Estimate \eqref{time_estimate_Bs} in Theorem \ref{mainth}.

\subsection{Uniqueness of the solution} 
We suppose that $u_1, u_2$ are two solutions of \eqref{ANSG_Strip} on $[0,T]$ with the same initial data. 
Let $U=u_1-u_2$, $V=v_1-v_2$, $P=p_1-p_1$
then, $(U,V,P)$ satisfies the following equations
\begin{equation}\label{uniquness_sol}
		\left\{
		\begin{aligned}
			&\partial_{t}(U-\alpha_1^2 \partial_z^2 U)+u_1\,\partial_{x}(U-\alpha_1^2 \partial_z^2 U)+U\partial_x \omega_2+v_1\partial_z(U-\alpha_1^2 \partial_z^2 U)+V\,\partial_z \omega_2+B \\
			& \hspace{2cm} = \partial^2_z (U-\alpha_1^2 \partial_z^2 U)-\partial_{x}P
			\\
			&\partial_{z}P=0
			\\
			&\partial_{x} U+\partial_z V=0
			\\ &(U,V)=0, \quad \partial_z U=0 \quad\textsl{on}\quad z=0
			\\
			&(U,V)=0, \quad \partial_z  U=0 \quad \textsl{on} \quad z=1\\
			& U|_{t=0}=0,
		\end{aligned}
		\right.
\end{equation}
where we set $\omega_2=u_2-\alpha_1^2 \partial^2_z u_2$ and where
\begin{equation*}
		B = \alpha_1^2 \partial_z u_1 \partial_x\partial_z U
		+ \alpha_1^2 \partial_z U \partial_x\partial_z u_2
		- \alpha_1^2 \partial^2_z u_1 \partial_x U 
		- \alpha_1^2 \partial_z^2 U \, \partial_x u_2.
\end{equation*}
We now consider the following auxiliary functions
\begin{equation*}
		\Phi(t,\xi)=(a-\mu\Theta(t))|\xi| \quad \text{with} \quad \Theta'(t)=\|\partial_z^2 u_{1_{\phi}}(t)\|_{B^{\frac{1}{2}}} +\|\partial_z^2 u_{2_{\phi}}(t)\|_{B^{\frac{1}{2}}}
\end{equation*} 
and following \eqref{auxiliary_functions} we define 
\begin{align*}
		f_{\Phi}(t,x,y)=e^{\Phi(t,D_x)}\,f(t,x,y) := \mathcal{F}^{-1}_{\xi \to x}(e^{\Phi(t,\xi)}\,\widehat{f}(t,\xi,y)),
\end{align*}
where $\mu \geq \lambda$ will be determined later. 
From Theorem \ref{mainth}, we deduce that
\begin{align*}
\| u_{1_{\phi}} \|_{L^{\infty}(B^{\frac{3}{2}})} +\| u_{2_{\phi}} \|_{L^{\infty}(B^{\frac{3}{2}})}
+\|\partial_zu_{1_{\phi}} \|_{L^{\infty}(B^{\frac{3}{2}})}  
+\|\partial_zu_{2_{\phi}} \|_{L^{\infty}(B^{\frac{3}{2}})}   \leq M , 
\end{align*}
where $M \geq 1$ is a constant.
Following the proof of \eqref{Dyadic}, we get from \eqref{uniquness_sol} the following estimate
\begin{multline}  \label{UNICITY}
		\|e^{Rt}\Delta^h_q U_\Phi \|^2_{L^2}+ \alpha_1^2 \|e^{Rt} \Delta_q^h \partial_z U_\Phi\|^2_{L^2} + 2\eta\,\alpha_1^2\int_0^t \Theta'(t') \left \|e^{Rt'} |D_x|^{\frac{1}{2}}\Delta^h_q \partial_z U_\Phi(t') \right\|^2_{L^2}\,dt'
		\\
		+2\mu \int_0^t \Theta'(t') \left \|e^{Rt'} |D_x|^{\frac{1}{2}}\Delta^h_q U_\Phi(t') \right\|^2_{L^2}\,dt' +2\int_0^t \left \|e^{Rt'} \Delta^h_q \partial_z U_\Phi(t') \right\|^2_{L^2}\,dt'
		\\
		+2 \alpha_1 ^2\int_0^t \left \|e^{Rt'} \Delta^h_q \partial^2_z U_\Phi(t') \right\|^2_{L^2}\,dt' 
		\\
		=2\,I_{1,q}+2\,I_{2,q}+2\,I_{3,q}+2\,I_{4,q}+2\,I_{5,q}+2\,I_{6,q}+2\,I_{7,q}+2\,I_{8,q} , 
\end{multline}
where
\begin{align*}
I_{1,q}= &-\alpha_1^2 \int_0^t \psca{e^{Rt'}\Delta^h_q(\partial_z u_1 \partial_x\partial_z U)_{\Phi}, e^{Rt'} \Delta^h_q U_\Phi} dt' ,
\\
	I_{2,q}=& -\int_0^t \psca{e^{Rt'}\Delta^h_q (u_1\,\partial_{x}(U-\alpha_1^2 \partial_z^2 U))_\Phi,e^{Rt'}\Delta^h_q U_\phi} dt',
	\\
	I_{3,q}= &\, \alpha_1^2 \int_0^t \psca{e^{Rt'}\Delta^h_q(\partial^2_z u_1 \partial_x U)_{\Phi}, e^{Rt'} \Delta^h_q U_\Phi} dt',
	\\
	I_{4,q}= &- \int_0^t \psca{e^{Rt'}\Delta^h_q(V\,\partial_z \omega_2)_{\Phi}, e^{Rt'} \Delta^h_q U_\Phi} dt',
	\\
	I_{5,q}=&-\int_0^t \psca{e^{Rt'}\Delta^h_q (v_1\partial_z(U-\alpha_1^2 \partial_z^2 U))_\Phi,e^{Rt'}\Delta^h_q U_\Phi} dt',
	\\
	I_{6,q}= &\, \alpha_1^2 \int_0^t \psca{e^{Rt'}\Delta^h_q(\partial_z^2 U \, \partial_x u_2)_{\Phi}, e^{Rt'} \Delta^h_q U_\Phi} dt',
	\\
	I_{7,q}= &- \alpha_1^2\int_0^t \psca{e^{Rt'}\Delta^h_q(\partial_z U \partial_x\partial_z u_2)_{\Phi}, e^{Rt'} \Delta^h_q U_\Phi} dt' , 
	\\
	I_{8,q}=&-  \int_0^t \psca{e^{Rt'}\Delta^h_q(U\partial_x \omega_2)_{\Phi}, e^{Rt'} \,\Delta^h_q U_\Phi}\, dt'  . 
\end{align*}
The terms $I_{j,q}$, $j \in \set{1,\ldots,8}$ can be controlled as in the following lemma. The proof of this lemma is very close to the proof of Lemma
\ref{lem:nonlinear}. We will give a brief version of this proof in the appendix \ref{app estim-Iq}. 
\begin{lemma} \label{lem:estim-Iq}
Let $s \in ]0, 1[$, $T > 0$,  $R>0$. There exists a generic constant $C\geq 1$ such that, for any  $0 < t < T$, we have
\begin{gather} 
	\label{I1q}
		|I_{1,q}| \leq C 2^{-2qs} d_q^2 \, \|e^{Rt'} \partial_z U_\Phi\|^2_{\widetilde{L}^2_{T,\Theta'(t)}(B^{s+\frac{1}{2}})} , \\
	\label{I2q}
 		|I_{2,q}| \leq C 2^{-2qs}d^2_q\,\|e^{Rt'} \partial_z U_\Phi\|^2_{\widetilde{L}^2_{T,\Theta'(t)}(B^{s+\frac{1}{2}})}  ,   \\
    \label{I3q}
		|I_{3,q}| \leq C  2^{-2qs} d_q^2\, \|e^{Rt'} \partial_z U_\Phi\|^2_{\widetilde{L}^2_{T,\Theta'(t)}(B^{s+\frac{1}{2}})} , \\
    \label{I4q}
		|I_{4,q}| \leq C 2^{-2qs} d_q^2 \, \|e^{Rt'} \partial_z U_\Phi\|^2_{\widetilde{L}^2_{T,\Theta'(t)}(B^{s+\frac{1}{2}})} , \\
    \label{I5q}
		|I_{5,q}| \leq C 2^{-2qs}d^2_q\,\| u_{1_{\Phi}} \|_{L^{\infty}_T(B^{\frac{3}{2}})}^{\frac{1}{2}} \,
		\|e^{Rt'} \partial^2_z U_\Phi\|_{\widetilde{L}^2_{T}(B^{\frac{1}{2}})} \, 
		\|e^{Rt'} \partial_z U_\Phi\|^2_{\widetilde{L}^2_{T,\Theta'(t)}(B^{s+\frac{1}{2}})} , \\
    \label{I6q}
 		 |I_{6,q}| \leq C 2^{-2qs}d^2_q\,\| u_{2_{\Phi}} \|_{L^{\infty}_T(B^{\frac{3}{2}})}^{\frac{1}{2}} \,
 		 \|e^{Rt'} \partial^2_z U_\Phi\|_{\widetilde{L}^2_{T}(B^{\frac{1}{2}})} \, 
 		 \|e^{Rt'} \partial_z U_\Phi\|^2_{\widetilde{L}^2_{T,\Theta'(t)}(B^{s+\frac{1}{2}})} ,
\end{gather}
\begin{align}	\label{I7q}
|I_{7,q}| \leq 
  C  2^{-2qs}d^2_q & \|e^{Rt'} \partial_z U_\Phi\|_{\widetilde{L}^2_{T,\Theta'(t)}(B^{s+\frac{1}{2}})}, \\
  &\times 
  \left(
    \|e^{Rt'} \partial_z U_\Phi\|_{\widetilde{L}^2_{T,\Theta'(t)}(B^{s+\frac{1}{2}})}
    +\|\partial_zu_{2_{\Phi}} \|_{L^{\infty}_T(B^{\frac{3}{2}})}^{\frac{1}{2}} \,
    \|e^{Rt'} \partial_z U_\Phi\|_{\widetilde{L}^2_{T}(B^{s+\frac{1}{2}})} 
  \right) , \notag
\end{align}
\begin{align}\label{I8q}
\abs{I_{8,q}} \leq 
  C  2^{-2qs}d^2_q & \|e^{Rt'} \partial_z U_\Phi\|_{\widetilde{L}^2_{T,\Theta'(t)}(B^{s+\frac{1}{2}})}
\\
  &\times \left(
      \|e^{Rt'} \partial_z U_\Phi\|_{\widetilde{L}^2_{T,\Theta'(t)}(B^{s+\frac{1}{2}})}
      +\|\partial_zu_{2_{\Phi}} \|_{L^{\infty}_T(B^{\frac{3}{2}})}^{\frac{1}{2}} \, 
       \|e^{Rt'} \partial_z U_\Phi\|_{\widetilde{L}^2_{T}(B^{s+\frac{1}{2}})} 
  \right) .  \notag
\end{align}			
\end{lemma}
Now, using Lemma \ref{lem:estim-Iq} and choosing $\mu=C^2\,M^2$, we deduce from \eqref{UNICITY} that, for any $s\in]0,1[$
\begin{multline*}
		\|e^{Rt} U_\Phi \|_{\widetilde{L}_T^{\infty}(B^s)}
		+ \|e^{Rt}  \partial_z U_\Phi\|_{\widetilde{L}_T^{\infty}(B^s)}
		+\|e^{Rt}  \partial_z U_\Phi\|_{\widetilde{L}_T^{2}(B^s)}
		+\|e^{Rt}  \partial^2_z U_\Phi\|_{\widetilde{L}_T^{2}(B^s)}
		\\
		\leq C(\|e^{a|D_x|}U_0
		\|_{B^s}+\,  \|e^{a|D_x|} \partial_z U_0\|_{B^s}) = 0,
\end{multline*}
which implies the uniqueness of the solution.

\subsection{Construction of approximate solutions}
Before introducing the construction methods, we remark that if $(u,v,p)$ is a solution of \eqref{ANSG_Strip} 
such that all the following expressions make sense, then we have : 
\begin{lemma}\label{lem=energy}
	\begin{equation}\label{eq dtu2}
		\frac{1}{2} \frac{d}{dt} \left(\|u\|^2_{L^2}+\alpha_1^2 \|\partial_z u\|^2_{L^2} \right) + \|\partial_z u\|^2_{L^2}+\alpha_1^2 \|\partial^2_z u\|^2_{L^2}=0.
	\end{equation}
\end{lemma}
We also remark that Estimate \eqref{eq dtu2} is also true for the approximate solutions.
\begin{proof}
We set $w = u - \alpha_1^2 \partial_z^2 u$. 
Taking $L^2$ scalar product of \eqref{ANSG_Strip} with $u$, we have
\begin{align*}  
  \psca{ \partial_t \omega, u } 
+ \psca{u\,\partial_{x}\omega, u} + \psca{v\partial_z\omega,u}
+ \alpha_1^2 \psca{\partial_z u\,\partial_x \partial_z u,u} - \alpha_1^2 \psca {\partial^2_z u\partial_x u,u}
=\psca{\partial^2_z \omega,u}-\psca{\partial_{x}p,u}.
\end{align*}
We first consider the linear terms. Performing integration by parts in $z$-variable and remarking that 
\begin{align*}
\psca {\partial_x p,u }&=-\psca {p, \partial_x u} 
= \psca{p, \partial_z v} 
=-\psca {\partial_z p, v} 
=0.
\end{align*}
We obtain
\begin{equation*} 
\frac{1}{2} \frac{d}{dt} \left(\|u\|^2_{L^2}+\alpha_1^2 \|\partial_z u\|^2_{L^2} \right)
  +\|\partial_z u\|^2_{L^2}+\alpha_1^2 \|\partial^2_z u\|^2_{L^2} + \text{NL} = 0 , 
\end{equation*}
where the non-linear term NL is
\begin{equation*}
	\text{NL} = \psca{u\,\partial_{x}\omega, u} + \psca{v\partial_z\omega,u} + \alpha_1^2 \psca{\partial_z u\,\partial_x \partial_z u,u} - \alpha_1^2 \psca {\partial^2_z u\partial_x u,u}.
\end{equation*}
Performing integration by parts in $x$ or $z$ variable and using the incompressibility condition, we can write 
\begin{align*}
 	\psca {u\,\partial_{x}\omega, u } = 
 	-\frac{1}{2}  \int_{\mathbb{R} \times ]0,1[} (\partial_x u) \, u^2 + 2\alpha_1^2\int_{\mathbb{R} \times ]0,1[} (\partial_z^2 u) (\partial_{x}u) \, u , 
\end{align*}
\begin{align*}
 	\psca{v\,\partial_{z}\omega, u }&= \frac{1}{2}  \int_{\mathbb{R} \times ]0,1[} v \partial_z(u^2) - \alpha_1^2\int_{\mathbb{R} \times ]0,1[} \partial_{z} (\partial_z^2 u) (uv) \\
	& =- \frac{1}{2}  \int_{\mathbb{R} \times ]0,1[} (\partial_z v) \, u^2 - \frac{1}{2}\alpha_1^2\int_{\mathbb{R} \times ]0,1[} (\partial_{z}u)^2 \,(\partial_z v) - \alpha_1^2\int_{\mathbb{R} \times ]0,1[}(\partial_z^2 u) (\partial_{x}u)\,u 
\end{align*}
and
\begin{align*}
 	\alpha_1^2 \psca {\partial_z u\,\partial_x \partial_z u,u} = -\frac{1}{2} \alpha_1^2 \int_{\mathbb{R} \times ]0,1[} (\partial_z u)^2  \partial_x u.
\end{align*}
Putting these identities into NL, we get 
\begin{align*}
	\text{NL} = 0
\end{align*}
and this concludes the proof of the lemma. 
\end{proof}

\medskip

We now introduce the following Hilbert spaces
\begin{gather*}
	H^2_0=\left\{ f \in H^2(]0,1[) \,|\, f(0)=f(1)= f'(0)=f'(1)=0\right\},\\
	H^1_0 =\left\{ f \in H^1(]0,1[) \,|\, f(0)=f(1)=0 \right\}
\end{gather*}
equipped with their respective norms
\begin{gather*}
	\norm{f}_{H^2_0} = \left( \alpha_1^2 \int_0^1 (f''(z))^2 dz + \int_0^1 (f'(z))^2 dz   \right)^{\frac12},
	\\
	\norm{f}_{H^1_0} = \left( \alpha_1^2 \int_0^1 (f'(z))^2 dz + \int_0^1 (f(z))^2 dz \right)^{\frac12} 
\end{gather*}
and let be a common Hilbert basis $\{\tilde{e}_k\}_{k\ge 1}$ such that 
\begin{equation*}
	\forall v \in H^2_0, \quad \langle \tilde{e}_k, v  \rangle_{H^2_0} = \lambda_k \, \langle \tilde{e}_k, v  \rangle_{H^1_0}.
\end{equation*}
For any $(u,v,p)$ sufficiently smooth on $[0,T] \times \mathcal{S}$, $T > 0$, such that
\begin{align*}
	&\partial_{x} p = \alpha_1^2 \partial^3_z u|_{z=0} 
	- \alpha_1^2 \partial^3_z u|_{z=1}
	- \partial_x \int_0^1 (u)^2(t,x,y)\,dz
	-\alpha_1^2 \partial_x \int_0^1 (\partial_z u)^2(t,x,y)\,dz , 
	\\
	&v(t,x,z)  = - \int_0^z \partial_x u(t,x,\tilde{z}) d\tilde{z} . 
\end{align*}
We set
\begin{align}
	\label{eq:defR} R(u) = - (u\,\partial_x u) + \alpha_1^2 (u\,\partial_x \partial_z^2 u) - (v\,\partial_z u) + \alpha_1^2 (v\,\partial_z \partial_z^2 u) - \alpha_1^2\, (\partial_z  u\,\partial_x \partial_z u) + \alpha_1^2 (\partial^2_z u\partial_x u) - \partial_x p.
\end{align} 
%
For any $n \in \mathbb{N}$, we will look for an approximate solution $u^n$ of \eqref{ANSG_Strip} of the form
\begin{equation*}
	u^n (t,x,z) = \sum _{i=1}^n \tilde{u}_{in}(t,x) \, \tilde{e}_i(z) , 
\end{equation*}
where $\set{\tilde{u}_{in}}$ is solution of the following system of $n$ equations 
\begin{equation*}
	\left\{
	\begin{aligned}
		& \langle \partial_{t} (u^n -\alpha_1^2\partial_z^2 u^n) ,  \tilde{e}_k \rangle_{L^2_z} 
		 = \langle  \partial^2_z  (u^n-\alpha_1^2\partial_z^2 u^n)  , \tilde{e}_k \rangle_{L^2_z}
		  + \langle R(u^n) , \tilde{e}_k \rangle_{L^2_z}, \quad k=1,\ldots,n  
		\\
		& u_n{_{|t=0}}  =  \sum_{k=1}^n \langle u_0, \tilde{e}_k \rangle_{H^1_0} \, \tilde{e}_k , 
	\end{aligned}  
	\right.
\end{equation*}
which also means that, for $1\leq k \leq n$,
\begin{align*}
	\partial_{t} \tilde{u}_{kn}(t,x) & = -\lambda_k \, \tilde{u}_{kn}(t,x) + \langle R(u^n)(t,x,\cdot) , \tilde{e}_k \rangle_{L^2_z}.
\end{align*} 
Next, we define the following frequency cut-off operators $J_n$ in $x$-variable. For $f\in L^2(\mathbb{R})$, we set 
\begin{equation*}
	J_n f (x)=  \mathcal{F}^{-1}_h \left(\displaystyle{1\!\!1_{[-n,n]}} \mathcal{F}_h f \right)(x).
\end{equation*}
The operator $J_n$ is continuous from $L^2$ to $L^2$, satisfies $J_n^2 f = J_n f$ and for any positive integers $n$ and $p$, we have
\begin{align*}
	&|\partial_x^p J_n f|_{L^2(\mathbb{R})} \leq n^{p} \, |f|_{L^2(\mathbb{R})}, 
	\\
 	&|\partial_x^p J_n f|_{L^\infty(\mathbb{R})} \leq c_{np} |f|_{L^2(\mathbb{R})}.
\end{align*}  
Now, we consider the following approximate system 
\begin{equation}\label{Approx_system_Strip}
	\left\{
	\begin{aligned}
		& \partial_{t} \tilde{u}_{kn}(t,x) = -\lambda_k \, \tilde{u}_{kn}(t,x) + \langle J_n R(J_n u^n)(t,x,\cdot) , \tilde{e}_k \rangle_{L^2_z}, \quad k \in \set{1,\ldots,n}
		\\
		& \tilde{u}_{kn}{_{|t=0}}  =  \langle J_n u_0, \tilde{e}_k \rangle_{H^1_0} .
	\end{aligned}
\right.
\end{equation}
Since $\tilde{u}_{kn} \mapsto  \langle J_n R(J_n u^n)(t,x,\cdot) , \tilde{e}_k \rangle_{L^2_z}$ is locally Lipschitz in $L^2$, \eqref{Approx_system_Strip} is a system of ordinary differential equations in $L^2$. Then, Picard’s theorem implies the existence of a unique maximal solution $\set{\tilde{u}_{kn}(t,x)}$ of \eqref{Approx_system_Strip} on $ [0, T_n[$. Since $J_n^2 = J_n$, we deduce that $\set{J_n \tilde{u}_{kn}}$ is also a solution of \eqref{Approx_system_Strip} on $ [0, T_n[$. Thus, the uniqueness of the solution implies that $\tilde{u}_{kn} = J_n \tilde{u}_{kn}$, for any $k \in \set{1,\ldots,n}$. 
We remark that $\set{\tilde{e}_k}$ is not necessarily orthogonal in $L^2_z(]0,1[)$ 
but using the Gram-Schmidt process, we can always construct a orthogonal family $\set{e_k}$ of $L^2_z(]0,1)[$ such that
$$
\text{Vect}\left\{e_1, \ldots, e_n\right\} = \text{Vect} \left\{\tilde{e}_1, \ldots, \tilde{e}_n\right\}.
$$ 
We denote $\mathbb{P}_n$ the orthogonal projection in $L^2_z(]0,1)[$ onto $\text{Vect}\left\{e_1, \ldots, e_n\right\}$
\begin{equation*}
\mathbb{P}_n \, f = \sum _{k=1}^n \langle f, e_k \rangle_{L^2_z} \, e_k.
\end{equation*}
Then, we remark that $u^n$ is also the solution of the system 
\begin{equation}\label{Approx_system_Strip_un}
	\left\{
	\begin{aligned}
		& \partial_{t} (u^n -\alpha_1^2\partial_z^2 u^n) 
		 =  \partial^2_z  (u^n-\alpha_1^2\partial_z^2 u^n)  
		  + \mathbb{P}_n  R(u^n) 
		\\
		& u_n{_{|t=0}}  = \mathbb{P}_n  u_0.
	\end{aligned}
	\right.
\end{equation}
Same calculations as in the proof of Lemma \ref{lem=energy} and the fact that $\mathbb{P}_n$ the orthogonal projection in $L^2_z(]0,1)[$ onto $\text{Vect}\left\{e_1, \ldots, e_n\right\}$ imply that
\begin{equation*}
	\frac{1}{2} \frac{d}{dt} \left(\|u^n\|^2_{L^2} + \alpha_1^2 \|\partial_z u^n\|^2_{L^2} \right) + \|\partial_z u^n\|^2_{L^2}+\alpha_1^2 \|\partial^2_z u^n\|^2_{L^2}=0, 
\end{equation*}
for any $t\in [0, T_n[$, which means $T_n=+\infty$.

\begin{remark}
	Estimates \eqref{Estimate B3/2} and \eqref{time_estimate_Bs} also apply to the approximate solutions $u^n$.
\end{remark}

\subsection{Passage to the limit} Now, we want to take the limit of the sequence of approximate solutions \eqref{Approx_system_Strip_un}. 
We already see that $\set{u^n}$ is bounded in $L^{\infty}(\mathbb{R}_+, L^2_{loc})$ and due to \eqref{time_estimate_Bs}, we remark that $\set{\partial_t u^n}$ is bounded in $L^{2}(\mathbb{R}_+, H^{-1}_{loc})$. 
Applying Aubin-Lions lemma, we deduce the existence of a subsequence, always noted by $\set{u^n}$ for the sake of simplicity, such that $u^n \longrightarrow u$ in $L^{\infty}_{loc}(\mathbb{R}_+, H^{-1}_{loc})$. By interpolation, we obtain
\begin{equation}
	\label{eq:unlim} u^n  \longrightarrow  u \qquad  in \qquad  L^{\infty}_{loc}(\mathbb{R}_+, H^{-\delta}_{loc}),\quad 
	  \text{for all} \quad 0 <\delta < 1.
\end{equation}
We recall that $u^n$ is solution of the following system 
\begin{equation*}
	\left\{
	\begin{aligned}
		& \langle \partial_{t} (u^n -\alpha_1^2\partial_z^2 u^n) ,  e_k \rangle_{L_z^2} = \langle  \partial^2_z  (u^n-\alpha_1^2\partial_z^2 u^n)  , e_k \rangle_{L_z^2} + \langle R(u^n) , e_k \rangle_{L_z^2}, \quad k=1,\ldots,n
		\\
		& u_n{_{|t=0}}  =  \sum_{k=1}^n \langle u_0, e_k \rangle_{L_z^2} \, e_k ,
	\end{aligned}
	\right.
\end{equation*}
where $R$ is defined  in \eqref{eq:defR}. So the main point of this paragraph is to prove the following lemma on the convergence of the nonlinear term $R(u^n)$. 
\begin{lemma} \label{Cv_NL}
We have the following convergence as $n\to\infty$ 
	\begin{equation*}
		R(u^n)\longrightarrow R(u) \quad \text{in} \quad L^2_{loc}(H^{-4}_{loc}).
	\end{equation*}
\end{lemma}
\begin{proof} [Proof of Lemma \ref{Cv_NL}] We will only prove the convergence of the term $u^n\partial_x u^n$. The other terms of $R(u^n)$ can be treated in a similar way. We first write
\begin{equation*}
	u^n\partial_x u^n-u\partial_x u = (u^n-u)\partial_x u^n + u \partial_x(u^n-u).
\end{equation*}
From Estimates \eqref{Estimate B3/2} and \eqref{eq:unlim}, we have $\set{\partial_x u^n}$ 
is uniformly bounded in $L^2_{loc}(H^{\frac{1}{2}}_{loc})$. 
Choosing $\delta < \frac{1}2$ in \eqref{eq:unlim} and using the product law in Sobolev spaces on $\mathbb{R}^2$, we get
\begin{align*}
	\|(u^n-u)\partial_x u^n\|_{ L^2_{loc}(H^{-\frac{1}2-\delta}_{loc})} \leq \|u^n-u\|_{ L^\infty_{loc}(H^{-\delta}_{loc})} \; \|\partial_x u^n\|_{ L^2_{loc}(H^{\frac{1}{2}}_{loc})} \longrightarrow 0.
\end{align*}
Now, using again Estimate \eqref{Estimate B3/2}, we have that $u$ is bounded in $L^2_{loc}(H^\frac12_{loc})$. 
Then, \eqref{eq:unlim} and the product law in Sobolev spaces on $\mathbb{R}^2$ yield 
\begin{align*}
	\|u\partial_x (u^n-u)\|_{ L^2_{loc}(H^{-\frac12 - \delta}_{loc})} 
	\leq \|u\|_{ L^2_{loc}(H^\frac12_{loc})} \; \|\partial_x (u^n-u)\|_{ L^\infty_{loc}(H^{-\delta}_{loc})} \longrightarrow 0.
\end{align*}
\end{proof}

Now, coming back to \eqref{Approx_system_Strip_un}, Lemma \ref{Cv_NL} proves that the limit $u$ is solution of \eqref{ANSG_Strip} in the sense of distributions. This concludes the proof or Theorem \ref{mainth}.


\section{Appendix - proof of the Lemma \ref{lem:nonlinear}}\label{app lem:nonlinear}

{\color{\LEO}

We first introduce some notations and classical mathematical tools. 
Then we will prove the estimations of Lemma \ref{lem:nonlinear}.
We will only detail the proof of estimates \eqref{estimate_1} and \eqref{estimate_5}.
The same procedure can be followed to prove the other estimates.

We recall the following Bernstein-type lemma, which states that derivatives act almost as multipliers on distributions 
whose Fourier transforms are supported in a ball or an annulus. 
We refer the reader to \cite{BCD2011} for a proof of this lemma.
\begin{lemma}
	\label{lem bernstein}
	Let $k\in\mathbb{N}$, $d \in \mathbb{N}^*$ and $r_1, r_2 \in \mathbb{R}$ satisfy $0 < r_1 < r_2$. 
	There exists a constant $C > 0$ such that, for any $a, b \in \mathbb{R}$, $1 \leq a \leq b \leq +\infty$, 
	for any $\lambda > 0$ and for any $u \in L^a(\mathbb{R}^d)$, we have
	\begin{equation*} 
		\text{supp}\,\pare{\widehat{u}} \subset \set{\xi \in \mathbb{R}^d \;\vert\; \abs{\xi} \leq r_1\lambda} \quad 
		\Longrightarrow \quad 
		\sup_{\abs{\alpha} = k} \norm{\partial^\alpha u}_{L^b} \leq C^k\lambda^{k+ d \pare{\frac{1}a-\frac{1}b}} \norm{u}_{L^a} 
	\end{equation*}
	and
	\begin{equation*} 
		\text{supp}\,\pare{\widehat{u}} \subset \set{\xi \in \mathbb{R}^d \;\vert\; r_1\lambda \leq \abs{\xi} \leq r_2\lambda} \quad 
		\Longrightarrow \quad 
		C^{-k} \lambda^k\norm{u}_{L^a} \leq \sup_{\abs{\alpha} = k} \norm{\partial^\alpha u}_{L^a} \leq C^k \lambda^k\norm{u}_{L^a}.
	\end{equation*}
\end{lemma}

In order to prove the estimates of Lemma \ref{lem:nonlinear}, we need the Bony decomposition 
of a product of two functions $a$ and $b$ in the horizontal direction (see \cite{BCD2011})
\begin{equation}\label{Bony decomp}
	ab = T_a^h \, b + T_b^h \, a + R^h(a,b) , 
\end{equation}
where
\begin{align*}
T^h_a b &= \sum_{q \in \mathbb{Z}} S_{q-1}^h a\Delta^h_q b
\quad\text{and}\quad
R^h(a,b) 
  = \sum_{|q-q'|\leq 1} \Delta^h_q a \Delta_{q'}^h b
  = \sum_{q\in\mathbb{Z}} \widetilde{\Delta}^h_{q} a \, \Delta_{q}^h b
\end{align*}
and
\begin{equation*}
\widetilde{\Delta}^h_{q} f 
  =  \sum_{|q-q'|\leq 1} \Delta^h_{q'} f. 
\end{equation*}
From the support properties to the Fourier transform of the terms $\Delta^h_q f$, 
we can verify 
\begin{align}\label{DhSh=0}
&\Delta^h_q(S^h_{q'-1} a\, \Delta^h_{q'} b)=0 \quad if \quad |q'-q| \geq 5 \quad\text{and}\quad
 \Delta^h_q(S^h_{q'+2} a\, \Delta^h_{q'} b)=0 \quad if \quad q' \leq q- 4.
\end{align}
Then, for a function $f$ (we suppose that all the expressions below make sense), we write  
\begin{align}	
	\label{int abf} \int_0^T |\langle e^{Rt'} \Delta^h_q(a\, b)_\phi, e^{Rt'} \Delta^h_q f \rangle |dt' 
	\leq \mathcal{A}_{q}+\mathcal{B}_{q}+\mathcal{R}_{q} , 
\end{align}
where
\begin{align*}
	& \mathcal{A}_{q}=\int_0^T |\langle e^{Rt'} \Delta^h_q(T^h_a b)_\phi, e^{Rt'} \Delta^h_q f \rangle|dt' , 
	\\
	& \mathcal{B}_{q}=\int_0^T |\langle e^{Rt'} \Delta^h_q(T^h_b a)_\phi, e^{Rt'} \Delta^h_q f \rangle |dt' , 
	\\
	& \mathcal{R}_{q}=\int_0^T |\langle e^{Rt'} \Delta^h_q(R^h(a,b))_\phi, e^{Rt'} \Delta^h_q f \rangle |dt' . 
\end{align*}
Then we have from the definition of the operator $T^h_a$ and \eqref{DhSh=0}
\begin{align*}
	& \mathcal{A}_{q}    \leq    \sum_{|q'-q|\leq 4 } \int_0^T e^{2Rt'}
       \left|\langle \Delta^h_q \left( S_{q'-1}^h a\,\Delta^h_{q'} b\right)_\phi, \Delta^h_q f_\phi \rangle \right|dt', 
		\\
	& \mathcal{B}_{q}    \leq    \sum_{|q'-q|\leq 4 } \int_0^T e^{2Rt'}
       \left|\langle \Delta^h_q \left( S_{q'-1}^h b\,\Delta^h_{q'} a\right)_\phi, \Delta^h_q f_\phi \rangle \right|dt', 
		\\
	& \mathcal{R}_{q} \leq     \sum_{q \geq q-3} \int_0^T e^{2Rt'}
       \left|\langle \Delta^h_q \left( \Delta^h_{q'} a \,\Delta^h_{q'} b \right)_\phi, \Delta^h_q f_\phi \rangle \right|dt' . 
\end{align*}
Throughout this section, we will keep these notations (which can however be modified when needed). 
Following the remark \ref{rem dq}, we define 
\  
For any $f\in L^2(R)$, we define \[f^+=\mathcal{F}^{-1}_h(|\widehat{f}|) . \]
We remark that, for any $q \in \mathbb{Z}$, we have
\begin{align}\label{f^+}
\Delta_q^h f^+=(\Delta_q^h f)^+,    \qquad 
S_q^h f^+=(S_q^h f)^+ \qquad \text{and} \qquad
    \|f^+\|_{L^2}=\|f\|_{L^2}.
\end{align}

We will use the following lemma (see \cite{CGP2011}):
\begin{lemma}
For smooth functions we have
\begin{align*}
&\left|\mathcal{F}_h \left( (S_{q'-1}^h a \,\, \Delta_{q'}^h b)_\phi \right) (\xi,y) \right| 
   \leq  |\mathcal{F}_h(S_{q'-1}^ha_\phi^+ \,\,\Delta_{q'}^h b_\phi^+)(\xi,y)|,
\\
   &\left|\mathcal{F}_h \left( (\Delta_{q'}^h a \,\, \Delta_{q'}^h b)_\phi \right) (\xi,y) \right| 
   \leq  |\mathcal{F}_h(\Delta_{q'}^h a_\phi^+ \,\,\Delta_{q'}^h b_\phi^+)(\xi,y)|.
\end{align*}
\end{lemma}
\begin{proof}
Setting $\tilde{a}=S_{q'-1}^h a$ and $\tilde{b}=\Delta_{q'}^h b$, 
we have 
from the inequality $\phi(t,\xi) \leq  \phi(t,\xi-\eta) + \phi(t,\eta)$ 
\begin{align*}
\left|\mathcal{F}_h \left((\tilde{a} \,\, \tilde{b})_\phi \right) (\xi,y) \right| 
&= \left| \,e^{\phi(t,\xi)} 
   \int_{\mathbb{R}} \mathcal{F}_h(\tilde{a})(\xi-\eta,y) \, \mathcal{F}_h(\tilde{b})(\eta,y) \,d \eta  \right| 
\\
&\leq  
   \int_{\mathbb{R}} 
     \left|e^{\phi(t,\xi-\eta)} \mathcal{F}_h(\tilde{a})(\xi-\eta,y)\right|  \, 
     \left|e^{\phi(t,\eta)} \mathcal{F}_h(\tilde{b})(\eta,y)  \right|   \,d \eta 
\\
&\leq  
   \int_{\mathbb{R}} 
       \mathcal{F}_h(\tilde{a}^+_\phi)(\xi-\eta,y)  \, 
       \mathcal{F}_h(\tilde{b}^+_\phi)(\eta,y)    \,d \eta,
\end{align*}
and this concludes the proof of the first inequality,
and the same procedure can be followed to deduce the second inequality.
\end{proof}
Therefore from this lemma, we derive the following corollary 
from Plancherel formula, Fubini's theorem and the H\"older's inequality.
\begin{corollary}\label{cor Dq Sq}
For smooth functions we have
\begin{align*}
&\|\Delta_q^h  f_\phi\|_{L^2} \lesssim \|f_\phi\|_{L^2},
\end{align*}
\begin{align*}
\left|\langle \Delta_{q}^h(S_{q'-1}^h a \, \Delta_{q'}^hb)_\phi, \Delta_q^h f_\phi  \rangle\right|
  & \lesssim \|S_{q'-1}^h a_\phi^+ \|_{L^\infty} \, \| \Delta_{q'}^h  b_\phi \|_{L^2} \, \|\Delta_q^h  f_\phi\|_{L^2}, 
  \\
  & \lesssim \|S_{q'-1}^h a_\phi^+ \|_{L^2_z(L^\infty_x)} \, \| \Delta_{q'}^h  b_\phi \|_{L^\infty_z(L^2_x)} \, \|\Delta_q^h  f_\phi\|_{L^2},
  \\
   & \lesssim \|S_{q'-1}^h a_\phi^+ \|_{L^2_z(L^\infty_x)} \, \| \Delta_{q'}^h  b_\phi \|_{L^2} \, \|\Delta_q^h  f_\phi\|_{L^\infty_z(L^2_x)} 
\end{align*}
and
\begin{align*}
\left|\langle \Delta_{q}^h (\Delta_{q'}^h a \, \Delta_{q'}^hb)_\phi, \Delta_q^h f_\phi  \rangle\right|
& \lesssim \|\Delta_{q'}^h a_\phi^+ \|_{L^\infty} \, \| \Delta_{q'}^h  b_\phi \|_{L^2} \, \|\Delta_q^h  f_\phi\|_{L^2} , 
\\
& \lesssim \|\Delta_{q'}^h a_\phi^+ \|_{L^2_z(L^\infty_x)} \, \| \Delta_{q'}^h  b_\phi \|_{L^\infty_z(L^2_x)} \, \|\Delta_q^h  f_\phi\|_{L^2} , 
\\
& \lesssim \|\Delta_{q'}^h a_\phi^+ \|_{L^\infty_z(L^2_x)} \, \| \Delta_{q'}^h  b_\phi \|_{L^2} \, \|\Delta_q^h  f_\phi\|_{L^2_z(L^\infty_x)}.  
\end{align*}
\end{corollary}
We will also use the following Poincar\'e inequality : 
For a function $u\in H^2$ such that $u$ and $u'$ vanish for $z=0,1$ we have
\begin{equation}\label{poincare}
 \|u \|_{L^\infty} \lesssim  \|\partial_z u \|_{L^2} 
 \quad\text{and}\quad
 \|u \|_{L^2} \lesssim  \|\partial_z u \|_{L^2} \lesssim \|\partial_z^2 u \|_{L^2}.
\end{equation}

\begin{proof}[Proof of Estimate \eqref{estimate_1}]
We apply the Bony's decomposition with $a=u$, $b=\partial_x u$ and $f=u_\varphi$.
Then following the notations of \eqref{int abf} we can write
\begin{align*}
\int_0^T |<e^{Rt'} \Delta^h_q(u \partial_x u)_\phi, e^{Rt'} \Delta^h_q u_\phi>|dt' \leq \mathcal{A}_{q}+ \mathcal{B}_{q}+ \mathcal{R}_{q}.
\end{align*}
\noindent
$\bullet$ Estimate of $\mathcal{A}_{q} = \int_0^T |<e^{Rt'} \Delta^h_q(T^h_u \partial_x u)_\phi, e^{Rt'} \Delta^h_q u_\phi>|dt'$.
We have from \eqref{DhSh=0}
\begin{align*}
\mathcal{A}_{q}&=\int_0^T \Big |\Big<e^{Rt'} \Delta^h_q\Big(\sum_{q' \in \mathbb{Z}} S_{q'-1}^h u\,\Delta^h_{q'} \partial_x u\Big)_\phi, e^{Rt'} \Delta^h_q u_\phi \Big> \Big|dt'
\\
& \leq \sum_{|q'-q|\leq 4 } \int_0^T e^{2Rt'}\Big|\Big< \Delta^h_q\left( S_{q'-1}^h u\,\Delta^h_{q'} \partial_x u\right)_\phi, \Delta^h_q u_\phi \Big> \Big|dt'.
\end{align*}
From the previous inequality and corollary \ref{cor Dq Sq} we get 
\begin{align*}
	\mathcal{A}_q 
	\lesssim 
	\sum_{|q'-q|\leq 4 } \int_0^T e^{2Rt'} \|S_{q'-1}^h u^+_\phi\|_{L^{\infty}}\,\|\Delta^h_{q'} \partial_x u_\phi\|_{L^2}  \|\Delta^h_q u_\phi \|_{L^2}    \,dt'. 
\end{align*}
Since the support of $\Delta_q^h u_\phi$ is in a ring and using the first inequality of Lemma \ref{lem bernstein} with $\lambda=2^q$, we get
\begin{align*}
\|\Delta^h_q u^+_\phi\|_{L^{\infty}} 
  &\lesssim 2^{\frac{q}{2}} \|\Delta^h_q u^+_\phi\|_{L^{\infty}_z({L^2_x})}
  = 2^{\frac{q}{2}} \|\Delta^h_q u_\phi\|_{L^{\infty}_z({L^2_x})}  . 
\end{align*}
Therefore from Poincar\'e inequality \eqref{poincare} and remark \ref{rem dq} we get 
\begin{align*}
\|\Delta^h_q u^+_\phi\|_{L^{\infty}} 
 &\lesssim 2^{\frac{q}{2}} \|\Delta^h_q u_\phi\|_{L^2_x(L^{\infty}_z)}   
  \lesssim 2^{\frac{q}{2}}  \|\Delta^h_q \partial^2_z u_\phi\|_{L^2} 
\\
 &\lesssim d_q(\partial_z^2 u_\phi) \|\partial^2_z u_\phi\|_{B^\frac{1}{2}}. 
\end{align*}
Then from the definition \ref{def Dh} of $S_{q'-1}$ and using $\sum_q d_q(\partial_z^2 u_\phi) = 1$, 
 \begin{equation}\label{estim Sq u}
 \|S_{q'-1}^h u^+_\phi\|_{L^{\infty}} \lesssim \|\partial^2_z u_\phi\|_{B^\frac{1}{2}}. 
\end{equation}
We deduce from the previous inequalities, lemma \ref{lem bernstein} and the cauchy-schwarz inequality that
\begin{align*}
\mathcal{A}_q
& \lesssim \sum_{|q'-q|\leq 4 } \int_0^T e^{2Rt'} \|\partial^2_z u_\phi\|_{B^\frac{1}{2}}\,
  \left(2^{q'} \|\Delta^h_{q'}  u_\phi\|_{L^2} \right)  \|\Delta^h_q u_\phi \|_{L^2}  \,dt'
\\
& \lesssim \sum_{|q'-q|\leq 4 }2^{q'} 
  \left(\int_0^T e^{2Rt'} \|\partial^2_z u_\phi\|_{B^\frac{1}{2}}\,\|\Delta^h_{q'}  u_\phi\|^2_{L^2} \,dt'\right)^{\frac{1}{2}}
  \times \left(\int_0^T e^{2Rt'} \|\partial^2_z u_\phi\|_{B^\frac{1}{2}}\, \|\Delta^h_q u_\phi \|^2_{L^2} \,dt'\right)^{\frac{1}{2}}.
\end{align*}
Since  $\theta'(t)=\|\partial^2_z u_\phi\|_{B^\frac{1}{2}}$, we have from remark \ref{rem dqf}
\begin{equation*}
\left(\int_0^T e^{2Rt'} \|\partial^2_z u_\phi\|_{B^\frac{1}{2}}\,\|\Delta^h_{q'}  u_\phi\|^2_{L^2} \,dt'\right)^{\frac{1}{2}}
\lesssim 2^{-q'(s+1/2)} \, d_{q'}(u_\phi,\theta') \, \|e^{Rt'} u_\phi\|_{\widetilde{L}^2_{T,\theta'(t)}  \, (B^{s+\frac{1}{2}})}
\end{equation*}
and we finaly obtain
\begin{align}\label{estim Aq}
\mathcal{A}_q \lesssim 2^{-2qs} \tilde{d_q} \, \|e^{Rt'} u_\phi\|^2_{\widetilde{L}^2_{T,\theta'(t)}(B^{s+\frac{1}{2}})} , 
\end{align}
where $\tilde{d_q}$ is the square root summable sequence (since $\sum_q d_q(u_\phi)= \sum_{q'}d_{q'}(u_\phi,\theta')=1$) defined by 
\begin{align}\label{dq tilde}
\tilde{ d_q}=d_q(u_\phi) \left(\sum_{|q'-q|\leq 4 } d_{q'}(u_\phi,\theta') 2^{(q-q')(s-1/2)} \right).
\end{align}
\noindent
$\bullet$ Estimate of $\mathcal{B}_q = 
\int_0^T |<e^{Rt'} \Delta^h_q(T^h_{\partial_x u} u)_\phi, e^{Rt'} \Delta^h_q u_\phi>|dt'$. 
From the notations \eqref{int abf}, 
\begin{align*}
\mathcal{B}_q& \lesssim \sum_{|q'-q|\leq 4 } \int_0^T e^{2Rt'} 
  \|S_{q'-1}^h \partial_x u^+_\phi\|_{L^{\infty}}\,\|\Delta^h_{q'}u_\phi\|_{L^2} 
   \|\Delta^h_q u_\phi \|_{L^2} \,dt'.
\end{align*}
Then, following the proof of Estimate \eqref{estim Sq u} and the definition \ref{def Dh} of $S^h_{q}$,  we have
\begin{align*}
\|S_{q'-1}^h \partial_x u^+_\phi\|_{L^{\infty}} 
 \lesssim 2^{q'}\| \partial^2_z u_\phi \|_{B^{\frac{1}{2}}}.
\end{align*}
Therefore using the Cauchy-Schwarz inequality and remark \ref{rem dq} we derive
\begin{align*}
\mathcal{B}_q & \lesssim  \sum_{|q'-q|\leq 4 } \int_0^T e^{2Rt'}  
  \left(2^{q'} \|\partial^2_z u_\phi\|_{B^\frac{1}{2}} \right) \,
   \|\Delta^h_{q'}u_\phi\|_{L^2}  \|\Delta^h_q u_\phi \|_{L^2}  \,dt'
 \\
 & \lesssim \sum_{|q'-q|\leq 4 }2^{q'} \left(\int_0^T e^{2Rt'} \|\partial^2_z u_\phi\|_{B^\frac{1}{2}}\,
 \|\Delta^h_{q'}  u_\phi\|^2_{L^2}\right)^{\frac{1}{2}}
 \times \left(\int_0^T e^{2Rt'} \|\partial^2_z u_\phi\|_{B^\frac{1}{2}}\, 
 \|\Delta^h_q u_\phi \|^2_{L^2}\right)^{\frac{1}{2}}
  \\
 & \lesssim  2^{-2qs} d_q(u_\phi) \left(\sum_{|q'-q|\leq 4 } d_{q'}(u_\phi) 2^{(q-q')(s-1/2)} \right)
  \|e^{Rt'} u_\phi\|^2_{\widetilde{L}^2_{T,\theta'(t)}(B^{s+\frac{1}{2}})} . 
\end{align*}
Then from the definition \eqref{dq tilde} of $\tilde{d_q}$
\begin{align}\label{estim Bq}
\mathcal{B}_q \lesssim 2^{-2qs} \tilde{d_q} \, \|e^{Rt'} u_\phi\|^2_{\widetilde{L}^2_{T,\theta'(t)}(B^{s+\frac{1}{2}})}.
\end{align}
\noindent
$\bullet$ Estimate of $\mathcal{R}_q$.
We have from the definition \eqref{int abf}
\begin{align*}
\mathcal{R}_{q} \leq \sum_{q'\geq q-3} \int_0^T e^{2Rt'}
       \left|\langle \Delta^h_q \left( \Delta^h_{q'} u \,\Delta^h_{q'} \partial_x u \right)_\phi, \Delta^h_q u_\phi \rangle \right|dt'. 
\end{align*}
Following the proof of \eqref{estim Sq u} and using corollary \ref{cor Dq Sq} and lemma \ref{lem bernstein} we get
\begin{align*}
\mathcal{R}_q 
& \lesssim \sum_{q'\geq q-3 } \int_0^T e^{2Rt'} \|\widetilde{\Delta}^h_{q'}u^+_\phi\|_{L^{\infty}_z(L^2_x)}\,
  \|\Delta^h_{q'}\partial_x u_\phi\|_{L^2}  \|\Delta^h_q u_\phi \|_{L^2_z(L^{\infty}_x)} \,dt'
\\
& \lesssim \sum_{q'\geq q-3 } \int_0^T e^{2Rt'}
    \left(2^{-\frac{q'}{2}}d_{q'}(\partial_z^2 u_\phi)\|\partial^2_z u_\phi\|_{B^{\frac{1}{2}}}\right)\,
    \left(2^{q'}\|\Delta^h_{q'} u_\phi\|_{L^2}\right)\,
    \left(2^{\frac{q}{2}}\|\Delta^h_q u_\phi \|_{L^2}\right)\,dt' . 
\end{align*}
Them using $d_{q'}(\partial_z^2 u_\phi) \leq 1$, the Cauchy-Schwarz inequality and remark \ref{rem dq} we get 
\begin{align*}
\mathcal{R}_q 
& \lesssim 2^{\frac{q}{2}}\sum_{q'\geq q-3 } 2^{\frac{q'}{2}}\left(\int_0^T e^{2Rt'} \|\partial^2_z u_\phi\|_{B^\frac{1}{2}}\,\|\Delta^h_{q'}  u_\phi\|^2_{L^2}\right)^{\frac{1}{2}}
 \times \left(\int_0^T e^{2Rt'} \|\partial^2_z u_\phi\|_{B^\frac{1}{2}}\, \|\Delta^h_q u_\phi \|^2_{L^2}\right)^{\frac{1}{2}} 
\\
& \lesssim 2^{\frac{q}{2}}\sum_{q'\geq q-3 }2^{\frac{q'}{2}} 
  \left(2^{-q'(s+1/2)} d_{q'}(u_\phi,\theta') \|e^{Rt'} u_\phi\|_{\widetilde{L}^2_{T,\theta'(t)}(B^{s+\frac{1}{2}})}\right)\, 
  \left(2^{-q(s+1/2)} d_{q}(u_\phi,\theta') \|e^{Rt'} u_\phi\|_{\widetilde{L}^2_{T,\theta'(t)}(B^{s+\frac{1}{2}})}\right) . 
\end{align*}
That is
\begin{align}\label{estim Rq}
\mathcal{R}_q 
& \lesssim 2^{-2qs} \check{d_q}\,\|e^{Rt'} u_\phi\|^2_{\widetilde{L}^2_{T,\theta'(t)}(B^{s+\frac{1}{2}})} , 
\end{align}
where $\check{d_q}$ is the square root summable sequence defined by
\begin{equation*}
\check{d_q} = d_{q}(u_\phi,\theta')\sum_{q'\geq q-3 } 2^{(q-q')s} d_{q'}(u_\phi,\theta') .
\end{equation*}
\noindent
$\bullet$ 
By Summing the estimates \eqref{estim Aq}, \eqref{estim Bq} and \eqref{estim Rq} we deduce \eqref{estimate_1}.
\end{proof}

}

\begin{proof}[Proof of Estimate $(\ref{estimate_2})$]
We apply the Bony's decomposition with $a=u$, $b=\partial_x \partial_z u$ and $f=\partial_z u_\varphi$.
Then following the notations of \eqref{int abf} we can write
\begin{align*}
\int_0^T |<e^{Rt'} \Delta^h_q(u \partial_x \partial_z u)_\phi, e^{Rt'} \Delta^h_q \partial_z u_\phi>|dt' \leq \mathcal{A}_{q}+\mathcal{B}_{q}+\mathcal{R}_{q}.
\end{align*}
\noindent
$\bullet$ Estimate of $\mathcal{A}_{q} = 
\int_0^T |<e^{Rt'} \Delta^h_q(T^h_u \partial_x \partial_z u)_\phi, e^{Rt'} \Delta^h_q \partial_z u_\phi>|dt'$. 
\begin{align*}
 \mathcal{A}_{q}& \lesssim \sum_{|q'-q|\leq 4 } \int_0^T e^{2Rt'} \|S_{q'-1}^h u^+_\phi\|_{L^{\infty}}\,\|\Delta^h_{q'} \partial_x \partial_z u_\phi\|_{L^2}  \|\Delta^h_q \partial_z  u_\phi \|_{L^2} 
 \\
 & \lesssim \sum_{|q'-q|\leq 4 } \int_0^T e^{2Rt'} \|\partial^2_z u_\phi\|_{B^\frac{1}{2}}\,2^{q'} \|\Delta^h_{q'}  \partial_z u_\phi\|_{L^2}  \|\Delta^h_q \partial_z u_\phi \|_{L^2}  
 \\
 & \lesssim \sum_{|q'-q|\leq 4 }2^{q'} \left(\int_0^T e^{2Rt'} \|\partial^2_z u_\phi\|_{B^\frac{1}{2}}\,\|\Delta^h_{q'} \partial_z u_\phi\|^2_{L^2}\right)^{\frac{1}{2}}
 \times \left(\int_0^T e^{2Rt'} \|\partial^2_z u_\phi\|_{B^\frac{1}{2}}\, \|\Delta^h_q \partial_z u_\phi \|^2_{L^2}\right)^{\frac{1}{2}}
 \\
 & \lesssim  2^{-2qs} \tilde{d_q} \, \|e^{Rt'} \partial_z u_\phi\|^2_{\widetilde{L}^2_{T,\theta'(t)}(B^{s+\frac{1}{2}})} , 
\end{align*}
where 
\begin{align}\label{dq tilde2}
\tilde{d_q}=d_q(\partial_z u_\phi, \theta') \left(\sum_{|q'-q|\leq 4 } d_{q'}(\partial_z u_\phi, \theta') 2^{(q-q')(s-\frac{1}{2})} \right) . 
\end{align}
\noindent
$\bullet$ Estimate of $\mathcal{B}_{q} = 
\int_0^T |<e^{Rt'} \Delta^h_q(T^h_{\partial_x\partial_z u} u)_\phi, e^{Rt'} \Delta^h_q \partial_z u_\phi>|dt'$. 
\begin{align*}
\mathcal{B}_{q} & \lesssim \sum_{|q'-q|\leq 4 } \int_0^T e^{2Rt'} \|S_{q'-1}^h \partial_x\partial_z u^+_\phi\|_{L^{\infty}}\,\|\Delta^h_{q'}u_\phi\|_{L^2}  \|\Delta^h_q\partial_z u_\phi \|_{L^2} \,dt 
 \\
 & \lesssim \sum_{|q'-q|\leq 4 } \int_0^T e^{2Rt'} 2^{q'} \|\partial^2_z u_\phi\|_{B^\frac{1}{2}}\,\|\Delta^h_{q'} \partial_z u_\phi\|_{L^2}  \|\Delta^h_q \partial_z u_\phi \|_{L^2} 
\\
 & \lesssim \sum_{|q'-q|\leq 4 }2^{q'} \left(\int_0^T e^{2Rt'} \|\partial^2_z u_\phi\|_{B^\frac{1}{2}}\,\|\Delta^h_{q'} \partial_z u_\phi\|^2_{L^2}\right)^{\frac{1}{2}}
 \times \left(\int_0^T e^{2Rt'} \|\partial^2_z u_\phi\|_{B^\frac{1}{2}}\, \|\Delta^h_q \partial_z u_\phi \|^2_{L^2}\right)^{\frac{1}{2}}
\\
 & \lesssim  2^{-2qs} \tilde{d_q} \, \|e^{Rt'} \partial_z u_\phi\|^2_{\widetilde{L}^2_{T,\theta'(t)}(B^{s+\frac{1}{2}})} , 
\end{align*}
with the same previous definition \eqref{dq tilde2} of $\tilde{d_q}$. 

\noindent
$\bullet$  Estimate of $\mathcal{R}_{q} = 
\int_0^T |<e^{Rt'} \Delta^h_q(R^h(u,\partial_x\partial_z u))_\phi, e^{Rt'} \Delta^h_q \partial_z u_\phi>|dt'$. 
\begin{align*}
\mathcal{R}_{q} 
& \lesssim \sum_{q'\geq q-3 } 
  \int_0^T e^{2Rt'} 
    \|\widetilde{\Delta}^h_{q'}u^+_\phi\|_{L^{\infty}_z(L^2_x)}\,
    \|\Delta^h_{q'}\partial_x \partial_z u_\phi\|_{L^2}\, 
    \|\Delta^h_q \partial_z u_\phi \|_{L^2_z(L^{\infty}_x)} \,dt'
\\
& \lesssim \sum_{q'\geq q-3 } \int_0^T e^{2Rt'} 2^{-\frac{q'}{2}} \|\partial^2_z u_\phi\|_{B^{\frac{1}{2}}}\,2^{q'}\|\Delta^h_{q'} \partial_z u_\phi\|_{L^2}\, 2^{\frac{q}{2}} \|\Delta^h_q \partial_z  u_\phi \|_{L^2}\,dt'
\\
& \lesssim 2^{\frac{q}{2}}\sum_{q'\geq q-3 } 
  2^{\frac{q'}{2}}\left(\int_0^T e^{2Rt'} 
  \|\partial^2_z u_\phi\|_{B^\frac{1}{2}}\,\|\Delta^h_{q'} 
  \partial_z u_\phi\|^2_{L^2}\right)^{\frac{1}{2}}
 \times 
 \left(\int_0^T e^{2Rt'} \|\partial^2_z u_\phi\|_{B^\frac{1}{2}}\, 
 \|\Delta^h_q \partial_z u_\phi \|^2_{L^2}\right)^{\frac{1}{2}}
\\
& \lesssim  2^{-2qs} \check{d_q} \, 
  \|e^{Rt'} \partial_z u_\phi\|^2_{\widetilde{L}^2_{T,\theta'(t)}(B^{s+\frac{1}{2}})} , 
\end{align*}
where
\begin{equation}\label{dq check 2}
\check{d_q} = d_{q}(\partial_z u_\phi, \theta')
   \sum_{q'\geq q-3 } 2^{(q-q')s} d_{q'}(\partial_z u_\phi, \theta') . 
\end{equation}

\noindent
$\bullet$ By Summing  the above estimates for $\mathcal{A}_q$, $\mathcal{B}_q$ and $\mathcal{R}_q$, 
we deduce $(\ref{estimate_2}$).
\end{proof}

\begin{proof}[Proof of Estimate $(\ref{estimate_3})$]
We apply the Bony's decomposition with $a=\partial_z u$, $b=\partial_x  u$ and $f=\partial_z u_\varphi$.
Then following the notations of \eqref{int abf} we can write
\begin{align*}
\int_0^T |<e^{Rt'} \Delta^h_q( \partial_z u \, \partial_x  u )_\phi, e^{Rt'} \Delta^h_q \partial_z u_\phi>|dt' \leq \mathcal{A}_{q}+\mathcal{B}_{q}+\mathcal{R}_{q}.
\end{align*}
\noindent
$\bullet$ Estimate of $\mathcal{A}_{q} = 
\int_0^T |<e^{Rt'}\Delta^h_q(T^h_{\partial_z u}\partial_x  u)_\phi,e^{Rt'}\Delta_q^h \partial_z u_\phi>|\,dt'$. 
\begin{align*}
\mathcal{A}_{q}
 & \lesssim \sum_{|q'-q|\leq 4 } 
 \int_0^T e^{2Rt'} 
 \|S_{q'-1}^h \partial_zu^+_\phi\|_{L^2_z(L_x^{\infty})}\,
 \|\Delta^h_{q'} \partial_x u_\phi\|_{L_z^{\infty}(L_x^2)} 
 \|\Delta^h_q \partial_z u_\phi \|_{L^2} . 
\end{align*}
Using  Lemma \ref{lem bernstein}, Poincar\'e inequality  we have
\begin{align*}
\|\Delta^h_q \partial_zu^+_\phi\|_{L^2_z(L^{\infty}_x)} &\lesssim 2^{\frac{q}{2}} \|\Delta^h_q \partial_zu_\phi\|_{L^2}
\\
& \lesssim d_q \|\partial^2_z u_\phi\|_{B^\frac{1}{2}} . 
\end{align*}
Then
\begin{equation*}
 \|S_{q'-1}^h \partial_zu^+_\phi\|_{L^2_z(L_x^{\infty})}\lesssim  \|\partial^2_z u_\phi\|_{B^\frac{1}{2}} . 
\end{equation*}
Using Poincar\'e inequality and  Lemma \ref{lem bernstein}, we have
\begin{align*}
\|\Delta^h_{q'} \partial_x u_\phi\|_{L_z^{\infty}(L_x^2)}\lesssim & \|\Delta^h_{q'} \partial_x \partial_z u_\phi\|_{L^2}
\\
\lesssim & 2^{q'} \|\Delta^h_{q'} \partial_z u_\phi\|_{L^2}. 
\end{align*}
Then 
\begin{align*}
\mathcal{A}_{q}
\lesssim& \sum_{|q'-q|\leq 4 } 2^{q'} \int_0^T e^{2Rt'}\| \partial^2_z u_\phi \|_{B^{\frac{1}{2}}}\,\|\Delta^h_{q'} \partial_z u_\phi\|_{L^2} \|\Delta^h_q \partial_zu_\phi \|_{L^2} \,dt'
\\
 \lesssim& \sum_{|q'-q|\leq 4 }2^{q'} \left(\int_0^T e^{2Rt'} \|\partial^2_z u_\phi\|_{B^\frac{1}{2}}\,\|\Delta^h_{q'}  \partial_z u_\phi\|^2_{L^2}\right)^{\frac{1}{2}}
 \times \left(\int_0^T e^{2Rt'} \|\partial^2_z u_\phi\|_{B^\frac{1}{2}}\, \|\Delta^h_q \partial_z u_\phi \|^2_{L^2}\right)^{\frac{1}{2}} 
\\
\lesssim& 2^{-2qs} \tilde{d_q}  \, \|e^{Rt'} \partial_z u_\phi\|^2_{\widetilde{L}^2_{T,\theta'(t)}(B^{s+\frac{1}{2}})} , 
\end{align*}
where $\tilde{d_q}$ is given by \eqref{dq tilde2}.

\noindent
$\bullet$ Estimate of $\mathcal{B}_{q} = \int_0^T  |<e^{Rt'}\Delta^h_q(T^h_{\partial_x  u}\partial_z u)_\phi,e^{Rt'}\Delta_q^h \partial_z u_\phi>|\,dt'$. 
\begin{align*}
\mathcal{B}_{q} 
\lesssim& \sum_{|q'-q|\leq 4 } \int_0^T e^{2Rt'} \|S_{q'-1}^h \partial_x u^+_\phi\|_{L^{\infty}}\,\|\Delta^h_{q'} \partial_z u_\phi\|_{L^2} \|\Delta^h_q \partial_z u_\phi \|_{L^2} \,dt
\\
\lesssim& \sum_{|q'-q|\leq 4 }  \int_0^T e^{2Rt'} 2^{q'}\| \partial^2_z u_\phi \|_{B^{\frac{1}{2}}}\,\|\Delta^h_{q'} \partial_z u_\phi\|_{L^2} \|\Delta^h_q \partial_zu_\phi \|_{L^2} \,dt'
\\
\lesssim& \sum_{|q'-q|\leq 4 }2^{q'} \left(\int_0^T e^{2Rt'} \|\partial^2_z u_\phi\|_{B^\frac{1}{2}}\,\|\Delta^h_{q'}  \partial_z u_\phi\|^2_{L^2}\right)^{\frac{1}{2}}
 \times \left(\int_0^T e^{2Rt'} \|\partial^2_z u_\phi\|_{B^\frac{1}{2}}\, \|\Delta^h_q \partial_z u_\phi \|^2_{L^2}\right)^{\frac{1}{2}} 
\\
\lesssim& 2^{-2qs} \tilde{d_q} \, \|e^{Rt'} \partial_z u_\phi\|^2_{\widetilde{L}^2_{T,\theta'(t)}(B^{s+\frac{1}{2}})} . 
\end{align*}
\noindent
$\bullet$ Estimate of $\mathcal{R}_{q} = \int_0^T |<e^{Rt'}\Delta^h_q(R^h(\partial_z u,\partial_x  u))_\phi,e^{Rt'}\Delta_q^h \partial_z u_\phi>|\,dt'$. 
\begin{align*}
\mathcal{R}_{q} & \lesssim \sum_{q'\geq q-3 } \int_0^T e^{2Rt'} \|\widetilde{\Delta}^h_{q'} \partial_z u^+_\phi\|_{L^2}\,\|\Delta^h_{q'}\partial_x  u_\phi\|_{L_z^{\infty}(L_x^2)}  \|\Delta^h_q \partial_z u_\phi \|_{{L_z^{2}(L_x^{\infty})} } \,dt'
\\
& \lesssim \sum_{q'\geq q-3 } \int_0^T 2^{-\frac{q'}{2}}\, d_{q'}(\partial_z u_\phi) e^{2Rt'} \,\| \partial^2_z u_\phi\|_{B^{\frac{1}{2}}}\,2^{q'}\|\Delta^h_{q'} \partial_z u_\phi\|_{L^2}\,2^{\frac{q}{2}}\|\Delta^h_{q} \partial_z u_\phi\|_{L^2} \,dt'
\\
& \lesssim 2^{\frac{q}{2}} \sum_{q'\geq q-3 } \int_0^T 2^{\frac{q'}{2}}\,  \|\partial^2_z u_\phi\|_{B^{\frac{1}{2}}}\,\|\Delta^h_{q'} \partial_z u_\phi\|_{L^2}\,\|\Delta^h_{q} \partial_z u_\phi\|_{L^2} \,dt'
\\
& \lesssim 2^{\frac{q}{2}}\sum_{q'\geq q-3 } 2^{\frac{q'}{2}}\left(\int_0^T e^{2Rt'} \|\partial^2_z u_\phi\|_{B^\frac{1}{2}}\,\|\Delta^h_{q'} \partial_z u_\phi\|^2_{L^2}\right)^{\frac{1}{2}}
 \times \left(\int_0^T e^{2Rt'} \|\partial^2_z u_\phi\|_{B^\frac{1}{2}}\, \|\Delta^h_q\partial_z u_\phi \|^2_{L^2}\right)^{\frac{1}{2}}
 \\
& \lesssim 2^{-2qs} \check {d_q} \, \|e^{Rt'} \partial_z u_\phi\|^2_{\widetilde{L}^2_{T,\theta'(t)}(B^{s+\frac{1}{2}})} , 
\end{align*}
where $\check {d_q} $ is given by \eqref{dq check 2}.

\noindent
$\bullet$  Summing up the above estimates for $\mathcal{A}_q$, $\mathcal{B}_q$ and $\mathcal{R}_q$, we obtain ($\ref{estimate_3}$).
\end{proof}

\begin{proof}[Proof of Estimate $(\ref{estimate_4})$]
We apply the Bony's decomposition with $a= \partial_z u$, $b=\partial_x \partial_z u$ and $f=u_\varphi$.
Then following the notations of \eqref{int abf} we can write
\begin{align*}
\int_0^T |<e^{Rt'} \Delta^h_q( \partial_z  u \, \partial_x \partial_z u)_\phi, e^{Rt'} \Delta^h_q  u_\phi>|dt' 
  \leq \mathcal{A}_{q}+\mathcal{B}_{q}+\mathcal{R}_{q}.
\end{align*}
\noindent
$\bullet$ Estimate of $\mathcal{A}_{q} = \int_0^T |<e^{Rt'}\Delta^h_q(T^h_{\partial_z u}\partial_x \partial_z u)_\phi,e^{Rt'}\Delta_q^h u_\phi>|\,dt'$.
\begin{align*}
\mathcal{A}_{q}
\lesssim \sum_{|q'-q|\leq 4 } \int_0^T 
  e^{2Rt'} \|S_{q'-1}^h \partial_zu^+_\phi\|_{L^2_z(L_x^{\infty})} \, 
  \|\Delta^h_{q'} \partial_x \partial_z u_\phi\|_{L^2} \,  
  \|\Delta^h_q u_\phi \|_{L_z^{\infty}(L_x^2)}  . 
\end{align*}
Using   Lemma \ref{lem bernstein} and Poincar\'e inequality, we have
\begin{align*}
\|\Delta^h_q \partial_zu^+_\phi\|_{L^2_z(L^{\infty}_x)} &\lesssim 2^{\frac{q}{2}} \|\Delta^h_q \partial_zu^+_\phi\|_{L^2}
\\
& \lesssim d_q(\partial^2_z u_\phi)  \|\partial^2_z u_\phi\|_{B^\frac{1}{2}} . 
\end{align*}
Then
\begin{equation*}
 \|S_{q'-1}^h \partial_zu^+_\phi\|_{L^2_z(L_x^{\infty})}\lesssim  \|\partial^2_z u_\phi\|_{B^\frac{1}{2}} . 
\end{equation*}
By Poincar\'e inequality  we have
\begin{equation*}
 \|\Delta^h_q u_\phi \|_{L_z^{\infty}(L_x^2)} \lesssim  \|\Delta^h_q \partial_z u_\phi\|_{L^2}  . 
\end{equation*}
Then
\begin{align*}
\int_0^T |<e^{Rt'}&\Delta^h_q(T^h_{\partial_z u}\partial_x \partial_z u)_\phi,e^{Rt'}\Delta_q^h u_\phi>|\,dt 
\\
 & \lesssim \sum_{|q'-q|\leq 4 } \int_0^T e^{2Rt'} 
    \|\partial^2_z u_\phi\|_{B^\frac{1}{2}}\,
    \left(2^{q'} \|\Delta^h_{q'}\partial_z u_\phi\|_{L^2}\right)\,
    \|\Delta^h_q \partial_z u_\phi\ \|_{L^2}  
 \\
 & \lesssim \sum_{|q'-q|\leq 4 }2^{q'} \left(\int_0^T e^{2Rt'} \|\partial^2_z u_\phi\|_{B^\frac{1}{2}}\,\|\Delta^h_{q'}  \partial_z u_\phi\|^2_{L^2}\right)^{\frac{1}{2}}
 \times \left(\int_0^T e^{2Rt'} \|\partial^2_z u_\phi\|_{B^\frac{1}{2}}\, \|\Delta^h_q \partial_z u_\phi \|^2_{L^2}\right)^{\frac{1}{2}} 
\\
& \lesssim 2^{-2qs} \tilde{d_q} \, \|e^{Rt'} \partial_z u_\phi\|^2_{\widetilde{L}^2_{T,\theta'(t)}(B^{s+\frac{1}{2}})} , 
\end{align*}
where $\tilde{d_q}$ is defined by \eqref{dq tilde2}.

\noindent
$\bullet$ Estimate of 
$\mathcal{B}_{q} = \int_0^T |<e^{Rt'}\Delta^h_q(T^h_{\partial_x \partial_z u}\partial_z u)_\phi,e^{Rt'}\Delta_q^h u_\phi>|\,dt'$. 
\begin{align*}
\mathcal{B}_{q}
\lesssim& \sum_{|q'-q|\leq 4 } \int_0^T 
  e^{2Rt'} \|S_{q'-1}^h \partial_x\partial_z u^+_\phi\|_{L^2_z(L_x^{\infty})} \,
  \|\Delta^h_{q'} \partial_z u_\phi\|_{L^2} \, 
  \|\Delta^h_q u_\phi \|_{L_z^{\infty}(L_x^2} )\,dt' . 
\end{align*}
We have
\begin{align*}
\|S_{q'-1}^h \partial_x\partial_z u^+_\phi\|_{L^2_z(L_x^{\infty})}
& \lesssim \sum_{l\leq q'-2} 2^{\frac{3l}{2}}\|\Delta^h_l \partial_z u_\phi^+ \|_{L^2}
\\
&=\sum_{l\leq q'-2} 2^{\frac{3l}{2}}\|\Delta^h_l \partial_z u_\phi \|_{L^2}
\\
&\lesssim\sum_{l\leq q'-2} 2^{l} \,2^{\frac{l}{2}}\|\Delta^h_l \partial^2_z u_\phi \|_{L^2}
\\
& \lesssim \sum_{l\leq q'-2} 2^{l} \,d_{l}(\partial^2_zu_\phi)\| \partial^2_z u_\phi \|_{B^{\frac{1}{2}}}
\\
& \lesssim 2^{q'}\| \partial^2_z u_\phi \|_{B^{\frac{1}{2}}} , 
\end{align*}
and by the Poincar\'e inequality, we get  
\begin{equation*}
 \|\Delta^h_q u_\phi \|_{L_z^{\infty}(L_x^2)} \lesssim  \|\Delta^h_q \partial_z u_\phi\|_{L^2}  . 
\end{equation*}
From the previous estimates,  
\begin{align*}
\mathcal{B}_{q}
\lesssim& \sum_{|q'-q|\leq 4 }  \int_0^T e^{2Rt'} 
  \|S_{q'-1}^h \partial_x\partial_z u^+_\phi\|_{L^2_z(L_x^{\infty})}\,
  \|\Delta^h_{q'} \partial_z u_\phi\|_{L^2} 
  \|\Delta^h_q u_\phi \|_{L_z^{\infty}(L_x^2} )\,dt'
\\
\lesssim& \sum_{|q'-q|\leq 4 }  \int_0^T e^{2Rt'} 2^{q'} 
    \|\partial^2_z u_\phi \|_{B^{\frac{1}{2}}}\,
    \|\Delta^h_{q'} \partial_z u_\phi\|_{L^2} 
    \|\Delta^h_q \partial_zu_\phi \|_{L^2} \,dt'
\\
\lesssim& \sum_{|q'-q|\leq 4 }2^{q'} 
  \left(\int_0^T e^{2Rt'} \|\partial^2_z u_\phi\|_{B^\frac{1}{2}}\,\|\Delta^h_{q'}  \partial_z u_\phi\|^2_{L^2}\right)^{\frac{1}{2}}
  \times 
  \left(\int_0^T e^{2Rt'} \|\partial^2_z u_\phi\|_{B^\frac{1}{2}}\, \|\Delta^h_q \partial_z u_\phi \|^2_{L^2}\right)^{\frac{1}{2}} 
\\
\lesssim& 2^{-2qs}\tilde{d_q} \, \|e^{Rt'} \partial_z u_\phi\|^2_{\widetilde{L}^2_{T,\theta'(t)}(B^{s+\frac{1}{2}})} . 
\end{align*}
Then 
\begin{align*}
\mathcal{B}_{q} 
\lesssim 2^{-2qs}\tilde{d_q} \, \|e^{Rt'} \partial_z u_\phi\|^2_{\widetilde{L}^2_{T,\theta'(t)}(B^{s+\frac{1}{2}})} . 
\end{align*}
\noindent
$\bullet$ Estimate of $\mathcal{R}_{q} = \int_0^T |<e^{Rt'}\Delta^h_q(R^h(\partial_z u,\partial_x \partial_z u))_\phi,e^{Rt'}\Delta_q^h u_\phi>|\,dt'$. 
\begin{align*}
\mathcal{R}_{q}
& \lesssim \sum_{q'\geq q-3 } \int_0^T e^{2Rt'} \|\widetilde{\Delta}^h_{q'} \partial_z u^+_\phi\|_{L^2}\,\|\Delta^h_{q'}\partial_x \partial_z u_\phi\|_{L^2}  \|\Delta^h_q u_\phi \|_{L^{\infty}} \,dt'
\\
& \lesssim \sum_{q'\geq q-3 } \int_0^T 2^{-\frac{q'}{2}}\, d_{q'}(\partial_z u_\phi) e^{2Rt'} \,\| \partial^2_z u_\phi\|_{B^{\frac{1}{2}}}\,2^{q'}\|\Delta^h_{q'} \partial_z u_\phi\|_{L^2}\,2^{\frac{q}{2}}\|\Delta^h_{q} \partial_z u_\phi\|_{L^2} \,dt'
\\
& \lesssim 2^{\frac{q}{2}} \sum_{q'\geq q-3 } \int_0^T 2^{\frac{q'}{2}}\,  \|\partial^2_z u_\phi\|_{B^{\frac{1}{2}}}\,\|\Delta^h_{q'} \partial_z u_\phi\|_{L^2}\,\|\Delta^h_{q} \partial_z u_\phi\|_{L^2} \,dt'
\\
& \lesssim 2^{\frac{q}{2}}\sum_{q'\geq q-3 } 2^{\frac{q'}{2}}\left(\int_0^T e^{2Rt'} \|\partial^2_z u_\phi\|_{B^\frac{1}{2}}\,\|\Delta^h_{q'} \partial_z u_\phi\|^2_{L^2}\right)^{\frac{1}{2}}
 \times \left(\int_0^T e^{2Rt'} \|\partial^2_z u_\phi\|_{B^\frac{1}{2}}\, \|\Delta^h_q\partial_z u_\phi \|^2_{L^2}\right)^{\frac{1}{2}}
 \\
& \lesssim 2^{-2qs}\check{d_q} \, \|e^{Rt'} \partial_z u_\phi\|^2_{\widetilde{L}^2_{T,\theta'(t)}(B^{s+\frac{1}{2}})} , 
\end{align*}
where $\check{d_q}$ is given by \eqref{dq check 2}.

\noindent
$\bullet$  By summing the three estimates for $\mathcal{A}_q$, $\mathcal{B}_q$ and $\mathcal{R}_q$, we obtain ($\ref{estimate_4}$).
\end{proof}

{\color{\LEO} 

\begin{proof}[Proof of Estimate $(\ref{estimate_5})$]
We apply the Bony's decomposition with $a = v$, $b = \partial_z u$ and $f=u_\varphi$.
Then following the notations of \eqref{int abf} we can write
\begin{align*}
\int_0^T |<e^{Rt'} \Delta^h_q( v \, \partial_z u)_\phi, e^{Rt'} \Delta^h_q  u_\phi>|dt' 
  \leq \mathcal{A}_{q}+\mathcal{B}_{q}+\mathcal{R}_{q}.
\end{align*}
\noindent
$\bullet$ Estimate of $\mathcal{A}_{q} = \int_0^T |<e^{Rt'}\Delta^h_q(T^h_{v}\partial_z  u)_\phi,e^{Rt'}\Delta_q^h  u_\phi>|\,dt'$.
From corollary \ref{cor Dq Sq}, 
\begin{align*}
\mathcal{A}_{q}
 \lesssim \sum_{|q'-q|\leq 4 } \int_0^T e^{2Rt'} 
 \|S_{q'-1}^h v^+_\phi\|_{L^{\infty}}\,\|\Delta^h_{q'} \partial_z u_\phi\|_{L^2} \|\Delta^h_q  u_\phi \|_{L^2}.
\end{align*}
From the incompressibility condition (see remark \ref{rem incompressible}) and the definition of $f_\phi$ we get 
\[\partial_x u_\phi+\partial_z v_\phi=0.\] Since $v_\phi|_{z=0}=0$,
then
\begin{align*}
v_\phi (t,x,z)=-\int_0^z \partial_x u_\phi(t,x ,z')dz'
\quad\text{and}\quad
\Delta_q^h v_\phi (t,x,z)=-\int_0^z \partial_x \Delta_q^h  u_\phi(t,x ,z')dz' . 
\end{align*}
Therefore
from lemma \ref{lem bernstein} we get 
\begin{align}\label{estim v}
\|\Delta^h_{q} v_\phi\|_{L_z^{\infty}(L_x^2)}& \leq \int_0^1  \|\partial_x \Delta_{q}^h  u_\phi(t,. ,z')\|_{L^{2}_x}dz'\notag
\\
&\lesssim \int_0^1 2^{q}  \|\Delta_{q}^h  u_\phi(t,. ,z')\|_{L^{2}_x}dz'  \notag
\\
&\lesssim 2^{q}  \|\Delta_{q}^h  u_\phi\|_{L^{2}} . 
\end{align}
Again by Lemma \ref{lem bernstein} and the properties of \eqref{f^+} of $f^+$, and the previous inquality we have 
\begin{align} \label{Bernstien v+}
\|\Delta_q^h v^+_\phi \|_{L^{\infty}} 
  &\lesssim 2^{\frac{q}{2}}\|\Delta_q^h v_\phi \|_{L_z^{\infty}(L^2_x)}
   \lesssim 2^{\frac{3q}{2}}\|\Delta_q^h u_\phi \|_{L^2} . 
\end{align}
Then from the Poincar\'e inequality and the remark \ref{rem dq}, 
\begin{align*}
\|\Delta_q^h v^+_\phi \|_{L^{\infty}}
 \lesssim 2^{\frac{3q}{2}}\|\Delta_q^h \partial^2_z u_\phi\|_{L^{2}} \lesssim 2^{q} \,d_{q}(\partial^2_zu_\phi)\|\partial^2_zu_\phi \|_{B^{\frac{1}{2}}} . 
\end{align*}
Then
\begin{align*}
\|S_{q'-1}^h v^+_\phi\|_{L^{\infty}} 
& \lesssim \sum_{l\leq q'-2} 2^{l} \,d_{l}(\partial_z^2 u_\phi)\| \partial^2_zu_\phi \|_{B^{\frac{1}{2}}}
\lesssim 2^{q'}\| \partial^2_z u_\phi \|_{B^{\frac{1}{2}}}. 
\end{align*}
From the previous estimates and lemma \ref{lem bernstein}, 
\begin{align*}
\mathcal{A}_{q}
 \lesssim  & \sum_{|q'-q|\leq 4 } \int_0^T e^{2Rt'} 
  \left(2^{q'} \|\partial^2_z u_\phi \|_{B^{\frac{1}{2}}}\right) \,
  \|\Delta^h_{q'} \partial_z u_\phi\|_{L^2} \|\Delta^h_q  u_\phi \|_{L^2}
 \\
\lesssim& \sum_{|q'-q|\leq 4 }2^{q'} \left(\int_0^T e^{2Rt'} \|\partial^2_z u_\phi\|_{B^\frac{1}{2}}\,\|\Delta^h_{q'}  \partial_z u_\phi\|^2_{L^2}\right)^{\frac{1}{2}}
 \times \left(\int_0^T e^{2Rt'} \|\partial^2_z u_\phi\|_{B^\frac{1}{2}}\, \|\Delta^h_q \partial_z u_\phi \|^2_{L^2}\right)^{\frac{1}{2}} 
\\
\lesssim& 2^{-2qs}\tilde{d_q} \, \|e^{Rt'} \partial_z u_\phi\|^2_{\widetilde{L}^2_{T,\theta'(t)}(B^{s+\frac{1}{2}})} , 
\end{align*}
where $\tilde{d}_q$ is given by \eqref{dq tilde2}.

\noindent
$\bullet$ Estimate of $\mathcal{B}_{q} = \int_0^T |<e^{Rt'}\Delta^h_q(T^h_{\partial_z  u} v)_\phi,e^{Rt'}\Delta_q^h u_\phi>|\,dt'$.
\begin{align*}
\mathcal{B}_{q}
 & \lesssim \sum_{|q'-q|\leq 4 } \int_0^T e^{2Rt'} \|S_{q'-1}^h \partial_z  u^+_\phi\|_{L^2_z(L_x^{\infty})}\,\|\Delta^h_{q'} v_\phi\|_{L_z^{\infty}(L_x^2)} \|\Delta^h_q  u_\phi \|_{L^2}\, dt' 
\\
 & \lesssim \sum_{|q'-q|\leq 4 } \int_0^T e^{2Rt'} 
   \|\partial^2_z u_\phi \|_{B^{\frac{1}{2}}}\,
    \left(2^{q'}  \|\Delta^h_{q'} u_\phi\|_{L^2} \right) 
    \|\Delta^h_q  u_\phi \|_{L^2}
\\
&\lesssim \sum_{|q'-q|\leq 4 }2^{q'} \left(\int_0^T e^{2Rt'} \|\partial^2_z u_\phi\|_{B^\frac{1}{2}}\,\|\Delta^h_{q'}\partial_z  u_\phi\|^2_{L^2}\right)^{\frac{1}{2}}
 \times \left(\int_0^T e^{2Rt'} \|\partial^2_z u_\phi\|_{B^\frac{1}{2}}\, \|\Delta^h_q \partial_z u_\phi \|^2_{L^2}\right)^{\frac{1}{2}} 
\\
&\lesssim 2^{-2qs}\tilde{d_q} \, \|e^{Rt'}\partial_z u_\phi\|^2_{\widetilde{L}^2_{T,\theta'(t)}(B^{s+\frac{1}{2}})}. 
\end{align*}
\noindent
$\bullet$ Estimate of $\mathcal{R}_{q} = \int_0^T |<e^{Rt'}\Delta^h_q(R^h(v,\partial_z  u))_\phi,e^{Rt'}\Delta_q^h u_\phi>|\,dt'$.
\begin{align*}
\mathcal{R}_{q}
 & \lesssim  \sum_{q'\geq q- 3 }  \int_0^T e^{2Rt'}\, 
 \|\Delta^h_{q'} v_\phi\|_{L_z^{\infty}(L_x^2)}
 \|\widetilde{\Delta}^h_{q'} \partial_z  u_\phi\|_{L^2}\,
 \|\Delta^h_q  u_\phi \|_{L^2_z(L^{\infty}_x)}\, dt' . 
\end{align*}
From Estimate \eqref{estim v} we derive 
%
%
\begin{align*}
\|\Delta^h_{q'} v_\phi\|_{L_z^{\infty}(L_x^2)}& \leq \int_0^1  \|\partial_x \Delta_{q'}^h  u_\phi(t,. ,z')\|_{L^{2}_x}dz'
\\
&\lesssim 2^{q'} \int_0^1  \|\Delta_{q'}^h  u_\phi(t,. ,z')\|_{L^{2}_x}dz'
\\
&\lesssim 2^{q'}  \|\Delta_{q'}^h  u_\phi\|_{L^{2}} . 
\end{align*}
Then, we have
\begin{align*}  
\mathcal{R}_{q}
& \lesssim  \sum_{q'\geq q- 3 }  \int_0^T e^{2Rt'}\, \|\Delta^h_{q'} v_\phi\|_{L_z^{\infty}(L_x^2)}
    \|\widetilde{\Delta}^h_{q'} \partial_z  u_\phi\|_{L^2}\,
    2^{\frac{q}{2}} \|\Delta^h_q  u_\phi \|_{L^2}\, dt' 
\\
& \lesssim 2^{\frac{q}{2}} \sum_{q'\geq q- 3 }  
  \int_0^T e^{2Rt'}\,2^{q'}  \|\Delta_{q'}^h  u_\phi\|_{L^{2}} \,
  2^{-\frac{q'}{2}}d_{q'}(\partial^2_z u_\phi)  \,
  \| \partial^2_z u_\phi\|_{B^{\frac{1}{2}}}\,\|\Delta^h_q  u_\phi \|_{L^2}\, dt' 
\\
& \lesssim 2^{\frac{q}{2}} \sum_{q'\geq q- 3 }  
  \int_0^T e^{2Rt'}\,2^{\frac{q'}{2}}  \|\Delta_{q'}^h  u_\phi\|_{L^{2}}  \,
  \| \partial^2_z u_\phi\|_{B^{\frac{1}{2}}}\,\|\Delta^h_q  u_\phi \|_{L^2}\, dt' 
\\
& \lesssim 2^{\frac{q}{2}}\sum_{q'\geq q-3 } 2^{\frac{q'}{2}}\left(\int_0^T e^{2Rt'} 
 \|\partial^2_z u_\phi\|_{B^\frac{1}{2}}\,\|\Delta^h_{q'} \partial_z  u_\phi\|^2_{L^2}\right)^{\frac{1}{2}}
 \times 
 \left(\int_0^T e^{2Rt'} \|\partial^2_z u_\phi\|_{B^\frac{1}{2}}\, 
 \|\Delta^h_q \partial_z u_\phi \|^2_{L^2}\right)^{\frac{1}{2}}
 \\
& \lesssim 2^{-2qs}\check{d_q} \, \|e^{Rt'} \partial_z u_\phi\|^2_{\widetilde{L}^2_{T,\theta'(t)}(B^{s+\frac{1}{2}})} , 
\end{align*}
where $\check{d_q}$ is given by \eqref{dq check 2}.

\noindent
$\bullet$  By summing up the above three estimates for $\mathcal{A}_q$, $\mathcal{B}_q$ and $\mathcal{R}_q$, we obtain ($\ref{estimate_5}$).
\end{proof}

}

\begin{proof}[Proof of Estimate $(\ref{estimate_6})$]

We apply the Bony's decomposition with $a = v$, $b = \partial_z^2 u$ and $f = \partial_z u_\phi$.
Then following the notations of \eqref{int abf} we can write
\begin{align*}
\int_0^T |<e^{Rt'} \Delta^h_q( v \, \partial_z^2 u)_\phi, e^{Rt'} \Delta^h_q  \partial_z u_\phi>|dt' 
  \leq \mathcal{A}_{q}+\mathcal{B}_{q}+\mathcal{R}_{q}.
\end{align*}
\noindent
$\bullet$ Estimate of $\mathcal{A}_{q} = \int_0^T |<e^{Rt'}\Delta^h_q(T^h_{v}\partial^2_z  u)_\phi,e^{Rt'}\Delta_q^h \partial_z u_\phi>|\,dt'$.  
By  Poincar\'e inequality we have
\begin{align*}
 \|\Delta_q^h  u_\phi\|_{L^{2}}
 \lesssim \|\Delta_q^h \partial_z  u_\phi\|^{\frac{1}{2}}_{L^2}\,\|\Delta_q^h \partial_z ^2u_\phi\|^{\frac{1}{2}}_{L^2} . 
\end{align*}
Then from \eqref{Bernstien v+} and the previous estimate,  and remark \ref{rem dq} we get
\begin{align*}
  \|\Delta_q^h v^+_\phi \|_{L^{\infty}} &\lesssim 2^{\frac{3q}{2}} 
  \|\Delta_q^h \partial_z  u_\phi\|^{\frac{1}{2}}_{L^2}\,\|\Delta_q^h \partial_z ^2u_\phi\|^{\frac{1}{2}}_{L^2} \\
& \lesssim  2^{\frac{q}{2}} d_q(\partial_z u_\phi)^\frac{1}{2} 
  \|\partial_z  u_\phi\|^{\frac{1}{2}}_{B^{\frac{3}{2}}}\,d_q(\partial_z ^2u_\phi)^\frac{1}{2} 
  \|\partial_z ^2u_\phi\|^{\frac{1}{2}}_{B^{\frac{1}{2}}} . 
\end{align*}
Then we have
\begin{align*}
\|S_{q'-1}^h v^+_\phi\|_{L^{\infty}} 
& \lesssim  2^{\frac{q'}{2}}
  \|\partial_z  u_\phi\|^{\frac{1}{2}}_{B^{\frac{3}{2}}}\,
  \| \partial_z ^2u_\phi\|^{\frac{1}{2}}_{B^{\frac{1}{2}}} . 
\end{align*}
Therefore
\begin{align*}
\mathcal{A}_{q} 
& \lesssim \sum_{|q'-q|\leq 4 } \int_0^T e^{2Rt'} 
   \|S_{q'-1}^h v^+_\phi\|_{L^{\infty}}\,
   \|\Delta^h_{q'} \partial^2_z u_\phi\|_{L^2} 
   \|\Delta^h_q \partial_z u_\phi \|_{L^2}  
 \\
& \lesssim \sum_{|q'-q|\leq 4 } 2^{\frac{q'}{2}} \int_0^T e^{Rt'}\, 
   \|\partial_z  u_\phi\|^{\frac{1}{2}}_{B^{\frac{3}{2}}}\|\Delta^h_{q'} \partial^2_z u_\phi\|_{L^2} \,
   e^{Rt'} \| \partial_z ^2u_\phi\|^{\frac{1}{2}}_{B^{\frac{1}{2}}} 
   \|\Delta^h_q \partial_z u_\phi \|_{L^2}
 \\ 
&\lesssim \sum_{|q'-q|\leq 4 } 2^{\frac{q'}{2}} 
  \left(\int_0^T e^{2Rt'}\|\partial_z  u_\phi\|_{B^{\frac{3}{2}}}\|\Delta^h_{q'} \partial^2_z u_\phi\|^2_{L^2}\right)^{\frac{1}{2}}
 \times 
 \left(\int_0^T e^{2Rt'} \,\| \partial^2_z u_\phi\|_{B^{\frac{1}{2}}} \|\Delta^h_q \partial_z u_\phi \|^2_{L^2}\right)^{\frac{1}{2}} 
  \\
&\lesssim \sum_{|q'-q|\leq 4 } 2^{\frac{q'}{2}}
  \|\partial_z  u_\phi\|^{\frac{1}{2}}_{L^{\infty}_T({B^{\frac{3}{2}}})} 
 \left(\int_0^T e^{2Rt'}\|\Delta^h_{q'} \partial^2_z u_\phi\|^2_{L^2}\right)^{\frac{1}{2}}
 \times 
 \left(\int_0^T e^{2Rt'} \,\| \partial^2_z u_\phi\|_{B^{\frac{1}{2}}} \|\Delta^h_q \partial_z u_\phi \|^2_{L^2}\right)^{\frac{1}{2}} 
  \\
&\lesssim 2^{-2qs}\hat{d_q} \,
  \|\partial_z  u_\phi\|^{\frac{1}{2}}_{L^{\infty}_T({B^{\frac{3}{2}}})}\, 
  \|e^{Rt'} \partial^2_z u_\phi\|_{\widetilde{L}^2_{T}(B^s)}\, 
  \|e^{Rt'} \partial_z u_\phi\|_{\widetilde{L}^2_{T,\theta'(t)}(B^{s+\frac{1}{2}})} , 
\end{align*}
where 
\begin{align*}
\hat{d_q}=d_q(\partial^2_z u_\phi,1)) \left(\sum_{|q'-q|\leq 4 } d_{q'}(\partial_z u_\phi,\theta') 2^{(q-q')(s-\frac{1}{2})} \right). 
\end{align*}
\noindent
$\bullet$ Estimate of $\mathcal{B}_{q} = \int_0^T |<e^{Rt'}\Delta^h_q(T^h_{\partial^2_z  u} v)_\phi,e^{Rt'}\Delta_q^h \partial_z u_\phi>|\,dt'$.
\begin{align*}\mathcal{B}_{q}
&\lesssim \sum_{|q'-q|\leq 4 } \int_0^T e^{2Rt'} 
  \|S_{q'-1}^h \partial^2_z  u^+_\phi\|_{L^2_z(L_x^{\infty})}\,
  \|\Delta^h_{q'} v_\phi\|_{L_z^{\infty}(L_x^2)} 
  \|\Delta^h_q \partial_z u_\phi \|_{L^2}\, dt' 
\\
&\lesssim \sum_{|q'-q|\leq 4 } \int_0^T e^{2Rt'} \,
  \| \partial^2_z u_\phi\|_{B^{\frac{1}{2}}}\,\left(2^{q'} 
  \|\Delta^h_{q'} \partial_z u_\phi\|_{L^2} \right) \|\Delta^h_q \partial_z u_\phi \|_{L^2}\, dt' 
\\
&\lesssim \sum_{|q'-q|\leq 4 } 2^{q'} \left(\int_0^T e^{2Rt'}\, 
  \|\partial^2_z u_\phi\|_{B^{\frac{1}{2}}}\,\|\Delta^h_{q'} \partial_z u_\phi\|^2_{L^2}\right)^{\frac{1}{2}}
 \times \left(\int_0^T e^{2Rt'} \,\| \partial^2_z u_\phi\|_{B^{\frac{1}{2}}}\, \|\Delta^h_q \partial_z u_\phi \|^2_{L^2}\right)^{\frac{1}{2}}
\\
&\lesssim 2^{-2qs} \tilde{d_q} \, \|e^{Rt'} \partial_z u_\phi\|^2_{\widetilde{L}^2_{T,\theta'(t)}(B^{s+\frac{1}{2}})} ,
\end{align*}
where $\tilde{d_q}$ is given by $\eqref{dq tilde2}$. 

\noindent
$\bullet$ Estimate of $\mathcal{R}_{q} = \int_0^T |<e^{Rt'}\Delta^h_q(R^h(v,\partial^2_z  u))_\phi,e^{Rt'}\Delta_q^h \partial_z u_\phi>|\,dt'$. 
From corollary \ref{cor Dq Sq}, we have
\begin{align*}
\mathcal{R}_{q} 
 & \lesssim  \sum_{q'\geq q- 3 }  \int_0^T e^{2Rt'}\, \|\Delta^h_{q'} v_\phi\|_{L_z^{\infty}(L_x^2)}\|\widetilde{\Delta}^h_{q'} \partial^2_z  u_\phi\|_{L^2}\,\|\Delta^h_q \partial_z u_\phi \|_{L^2_z(L^\infty_x)}\, dt' . 
\end{align*}
From \eqref{estim v}, we have
\begin{align*}
\|\Delta^h_{q'} v_\phi\|_{L_z^{\infty}(L_x^2)}
\lesssim 2^{q'}  \|\Delta_{q'}^h \partial_z u_\phi\|_{L^{2}} . 
\end{align*}
Then, using remark \ref{rem dq} and lemma \ref{lem bernstein}, remark \ref {rem dqf}, we have for any $s>0$
\begin{align*}
\mathcal{R}_{q} 
&\lesssim  \sum_{q'\geq q- 3 }  \int_0^T e^{2Rt'}\, 2^{q'} 
  \|\Delta_{q'}^h \partial_z  u_\phi\|_{L^{2}}\,
  \left(2^{-\frac{q'}{2}}\,d_{q'}(\partial_z^2 u_\phi)\| \partial^2_z u_\phi\|_{B^{\frac{1}{2}}} \right)\, 
  2^{\frac{q}{2}} \|\Delta^h_q\partial_z  u_\phi \|_{L^2}\, dt' 
\\
& \lesssim 2^{\frac{q}{2}} \sum_{q'\geq q- 3 }\,
  2^{\frac{q'}{2}}\,  \int_0^T e^{2Rt'}\, 
  \|\Delta_{q'}^h \partial_z  u_\phi\|_{L^{2}}\,
  \| \partial^2_z u_\phi\|_{B^{\frac{1}{2}}}\,
  \|\Delta^h_q\partial_z  u_\phi \|_{L^2}\, dt' 
\\
&\lesssim 2^{\frac{q}{2}} \sum_{q'\geq q- 3 }
  2^{\frac{q'}{2}}\left(\int_0^T e^{2Rt'}\,
    \|\partial^2_z u_\phi\|_{B^{\frac{1}{2}}}\,
    \|\Delta^h_{q'} \partial_z u_\phi\|^2_{L^2}\right)^{\frac{1}{2}}
 \times 
 \left(\int_0^T e^{2Rt'} \,
   \| \partial^2_z u_\phi\|_{B^{\frac{1}{2}}}\, 
   \|\Delta^h_q \partial_z u_\phi \|^2_{L^2}
 \right)^{\frac{1}{2}}
 \\
&\lesssim 2^{-2qs} \check{d}_q\,
  \|e^{Rt'} \partial_z u_\phi\|^2_{\widetilde{L}^2_{T,\theta'(t)}(B^{s+\frac{1}{2}})} , 
\end{align*}
where $\check{d}$ is given by \eqref{dq check 2}.

\noindent
$\bullet$ Summing up the above three estimates, we obtain ($\ref{estimate_6}$).
\end{proof}


\section{Appendix - Proof of the lemma \ref{law product} }\label{app nonlin}

{\color{\LEO}

\begin{proof}[Proof of Estimate $(\ref{Es1})$]
We first recall the Bony's decomposition \eqref{Bony decomp}
\begin{equation*}
  u \partial_x u=T^h_u \partial_x u+ T^h_{\partial_x u} u+R^h(u,\partial_x u).
\end{equation*}
%
%
Then
\begin{align*}
&\left(\int_0^T \| e^{Rt}\Delta^h_q(u\,\partial_{x} u)_{\phi}\|^2_{L^2}\, dt \right)^{\frac{1}{2}}
\\
&\lesssim  \left(\int_0^T \| e^{Rt}\Delta^h_q(T^h_u \partial_x u)_{\phi}\|^2_{L^2} dt \right)^{\frac{1}{2}} 
  + \left(\int_0^T \| e^{Rt}\Delta^h_q(T^h_{\partial_x u} u)_{\phi}\|^2_{L^2} dt \right)^{\frac{1}{2}}
  + \left( \int_0^T \| e^{Rt}\Delta^h_q(R^h(u,\partial_x u))_{\phi}\|^2_{L^2} dt \right)^{\frac{1}{2}}.
\end{align*}
\noindent
$\bullet$ Estimate of $\| e^{Rt}\Delta^h_q(T^h_u \partial_x u)_{\phi}\|_{L^2}$.
From the definition of $T_h$ we have
\begin{align*}
\| e^{Rt}\Delta^h_q(T^h_u \partial_x u)_{\phi}\|_{L^2}
&= \|e^{Rt} \Delta^h_q\Big(\sum_{q' \in \mathbb{Z}} 
  S_{q'-1}^h u\,\Delta^h_{q'} \partial_x u\Big)_\phi\|_{L^2}
&\lesssim \sum_{|q-q'|\leq 4} 
  \left \|e^{Rt} \Big( S_{q'-1}^h u\,\Delta^h_{q'} \partial_x u\Big)_\phi \right\|_{L^2} . 
\end{align*}
From corollary \ref{cor Dq Sq}, Poincarre inequality and lemma \ref{lem bernstein} we get 
\begin{align*} 
 \int_0^T\| e^{Rt}\Delta^h_q(T^h_u \partial_x u)_{\phi}\|^2_{L^2} dt
&  \lesssim \sum_{|q-q'|\leq 4} \int_0^T\|S_{q'-1}^h u^+_\phi\|^2_{L^{\infty}}\,\|e^{Rt}\Delta^h_{q'} \partial_x u_\phi\|^2_{L^2}\, dt 
\\
& \lesssim  \sum_{|q-q'|\leq 4} \int_0^T  \|\Delta^h_{q'}\partial_z u_\phi\|^2_{B^\frac{1}{2}} \, 
    \left(2^{2q'} \|e^{Rt}\Delta^h_{q'} \partial_z u_\phi\|^2_{L^2}  \right) dt
    \\
& \lesssim  \sum_{|q-q'|\leq 4}  \|\Delta^h_{q'}\partial_z u_\phi\|^2_{L_T^{\infty}({B^\frac{1}{2}})} \, 
    \left(2^{2q'} \int_0^T \|e^{Rt}\Delta^h_{q'} \partial_z u_\phi\|^2_{L^2}\, dt \right) . 
\end{align*}
Then from the remark \ref{rem dqf}
\begin{align} \label{Es1a}
 \left(\int_0^T\| e^{Rt}\Delta^h_q(T^h_u \partial_x u)_{\phi}\|^2_{L^2} dt \right)^{\frac{1}{2}}
& \lesssim  \sum_{|q-q'|\leq 4} \|\Delta^h_{q'}\partial_z u_\phi\|_{L_T^{\infty}({B^\frac{1}{2}})} \, 
    2^{q'} \left(\int_0^T \|e^{Rt}\Delta^h_{q'} \partial_z u_\phi\|^2_{L^2} dt \right)^{\frac{1}{2}} \notag
\\
& \lesssim \sum_{|q-q'|\leq 4}  \|\partial_z u_\phi\|_{\widetilde{L}_T^\infty(B^\frac{1}{2})} 
  2^{q'} \, \left(  2^{-\frac{5q'}{2}} d_{q'}(\partial_z u_\phi,1) 
      \|e^{Rt}\Delta^h_{q'} \partial_z u_\phi\|_{\widetilde{L}_T^2(B^\frac{5}{2})}  \right)
      \\
& \lesssim \tilde{d}_{q} \, 2^{-\frac{3q}{2}}\|\partial_z u_\phi\|_{\widetilde{L}_T^\infty(B^\frac{1}{2})}   
    \|e^{Rt} \partial_z u_\phi\|_{\widetilde{L}_T^2(B^\frac{5}{2})} , \notag
\end{align}
where
\begin{align}\label{dq tilde3}
\tilde{d}_q = \sum_{|q-q'|\leq 4} d_{q'}(\partial_z u_\phi,1)  2^{\frac{3}{2}(q-q')}. 
\end{align}
\noindent
$\bullet$ Estimate of $\| e^{Rt}\Delta^h_q(T^h_{\partial_x u} u)_{\phi}\|_{L^2}$.
Following the previous estimates we get
\begin{align}\label{Es1b}
 \left(\int_0^T \| e^{Rt}\Delta^h_q(T^h_{\partial_x u} u)_{\phi}\|_{L^2}\, dt \right)^{\frac{1}{2}}
& \lesssim \tilde{d}_{q} \, 2^{-\frac{3q}{2}}\|\partial_z u_\phi\|_{\widetilde{L}_T^\infty(B^\frac{1}{2})}   
    \|e^{Rt} \partial_z u_\phi\|_{\widetilde{L}_T^2(B^\frac{5}{2})}.
\end{align}
\noindent
$\bullet$ Estimate of $\| e^{Rt}\Delta^h_q(R^h(u,\partial_x u))_{\phi}\|_{L^2}$.
From corollary \ref{cor Dq Sq}, we have 
\begin{align*} 
\left(\int_0^T \| e^{Rt}\Delta^h_q(R^h(u,\partial_x u))_{\phi}\|^2_{L^2}\, dt \right)^\frac12
& \lesssim  \sum_{q' \geq q-3} 
  \left(\int_0^T
   \|\widetilde{\Delta}^h_{q'} u_\phi \|^2_{L^2_z (L^\infty_x)}\, 
   \|e^{Rt} \Delta^h_{q'} \partial_x u_\phi \|^2_{L^\infty_z(L^2_x)} \, dt \right)^\frac12
      \\
& \lesssim  \sum_{|q-q'|\leq 4} 
    \left(\int_0^T  \|\Delta^h_{q'}\partial_z u_\phi\|^2_{B^\frac{1}{2}} \, 
    \left(2^{2q'} \|e^{Rt}\Delta^h_{q'} \partial_z u_\phi\|^2_{L^2}  \right) dt \right)^\frac12
      \\
& \lesssim  \sum_{|q-q'|\leq 4}  \|\Delta^h_{q'}\partial_z u_\phi\|_{L_T^{\infty}({B^\frac{1}{2}})} \, 
    2^{q'} \left( \int_0^T \|e^{Rt}\Delta^h_{q'} \partial_z u_\phi\|^2_{L^2}\, dt \right)^\frac12 . 
\end{align*}
Then from remark \ref{rem dqf}
\begin{align} \label{Es1c} 
 \left(\int_0^T \| e^{Rt}\Delta^h_q(R^h(u,\partial_x u))_{\phi}\|^2_{L^2}\, dt\ \right)^{\frac{1}{2}}
& \lesssim \check{d}_{q} 2^{-\frac{3q}{2}} \|\partial_z u_\phi\|_{\widetilde{L}_T^\infty(B^\frac{1}{2})}   
    \|e^{Rt} \partial_z u_\phi\|_{\widetilde{L}_T^2(B^\frac{5}{2})} , 
\end{align}
where
\begin{align}\label{dq check3}
\check{d}_q = \sum_{q' \geq q-3} d_{q'}(\partial_z u_\phi,1)  2^{\frac{3}{2}(q-q')}. 
\end{align}
\noindent
$\bullet$ By summing estimates \eqref{Es1a}, \eqref{Es1b} and \eqref{Es1c} we get  
\begin{align*}
 \left(\int_0^T \| e^{Rt}\Delta^h_q(u\,\partial_{x} u)_{\phi}\|_{L^2} \, dt \right)^{\frac{1}{2}} 
& \lesssim  (2\tilde{d}_{q}+\check{q}_q) 2^{-\frac{3q}{2}} \|\partial_z u_\phi\|_{\widetilde{L}_T^\infty(B^\frac{1}{2})}   
    \|e^{Rt} \partial_z u_\phi\|_{\widetilde{L}_T^2(B^\frac{5}{2})}.
\end{align*}
Multiplying  the previous inequality by $2^{\frac{3q}{2}}$ and taking the sum over $\mathbb{Z}$
we obtain \eqref{Es1}.
\end{proof}
}
\begin{proof}[Proof of Estimate $(\ref{Es2})$]
We have
\begin{multline*}
\left(\int_0^T \|e^{Rt}  \Delta^h_q(u\,\partial_{x} \partial^2_z u)_{\phi}\|^2_{L^2}\, dt \right)^{\frac{1}{2}} 
 \lesssim  
 \left(\int_0^T \|e^{Rt}\Delta^h_q(T^h_u \partial_x \partial^2_z u)_{\phi}\|^2_{L^2} dt \right)^{\frac{1}{2}} 
 \\
 +  \left( \int_0^T \|e^{Rt}\Delta^h_q(T^h_{\partial_x \partial^2_z u} u)_{\phi}\|^2_{L^2} dt \right)^{\frac{1}{2}}
 + \left( \int_0^T \|e^{Rt}\Delta^h_q(R^h(u,\partial_x\partial^2_z u))_{\phi}\|^2_{L^2} dt \right)^{\frac{1}{2}}.
\end{multline*}
By adapting the proof of Estimate \eqref{Es1a} we get
\begin{align*}
\left(\int_0^T \|e^{Rt}\Delta^h_q(T^h_u \partial_x \partial^2_z u)_{\phi}\|^2_{L^2} dt \right)^{\frac{1}{2}}
& \lesssim \tilde{d}_{q} 2^{-\frac{3q}{2}} \|\partial_z u_\phi\|_{\widetilde{L}_T^\infty(B^\frac{1}{2})}  
	\|e^{Rt} \partial^2_z u_\phi\|_{\widetilde{L}_T^2(B^\frac{5}{2})} , 
\\
\left( \int_0^T \|e^{Rt}\Delta^h_q(T^h_{\partial_x \partial^2_z u} u)_{\phi}\|^2_{L^2} dt \right)^{\frac{1}{2}}
& \lesssim \tilde{d}_{q} 2^{-\frac{3q}{2}}\|e^{Rt} \partial^2_z u_\phi\|_{\widetilde{L}_T^2(B^\frac{1}{2})} 
 \|\partial_z u_\phi\|_{\widetilde{L}_T^\infty(B^\frac{5}{2})} , 
\end{align*}
where $\tilde{d}_q$ is given by \eqref{dq tilde3}. 
By adapting the proof of Estimate \eqref{Es1c} we get
\begin{align*}
 \left( \int_0^T \|e^{Rt}\Delta^h_q(R^h(u,\partial_x\partial^2_z u))_{\phi}\|^2_{L^2} dt \right)^{\frac{1}{2}}
& \lesssim \check{d}_{q} 2^{-\frac{3q}{2}} \|\partial_z u_\phi\|_{\widetilde{L}_T^\infty(B^\frac{1}{2})}  
	\|e^{Rt} \partial^2_z u_\phi\|_{\widetilde{L}_T^2(B^\frac{5}{2})} , 
\end{align*}
where $\check{d}_{q}$ is given by \eqref{dq check3}.
Thus by summing the previous estimates we derive
\begin{align*}
& \left(\int_0^T \|e^{Rt}\Delta^h_q(u\,\partial_{x} \partial^2_z u)_{\phi}\|^2_{L^2}\, dt \right)^{\frac{1}{2}}
\\
& \lesssim 2^{-\frac{3q}{2}} \left((\tilde{d}_{q}+\check{d}_{q})  \|\partial_z u_\phi\|_{\widetilde{L}_T^\infty(B^\frac{1}{2})}  
	\|e^{Rt} \partial^2_z u_\phi\|_{\widetilde{L}_T^2(B^\frac{5}{2})}+\tilde{d}_{q}\|e^{Rt} \partial^2_z u_\phi\|_{\widetilde{L}_T^2(B^\frac{1}{2})}   \| \partial_z u_\phi\|_{\widetilde{L}_T^\infty(B^\frac{5}{2})}\right) . 
\end{align*}
Multiplying  the previous inequality by $2^{\frac{3q}{2}}$ and  taking the sum over $\mathbb{Z}$, 
we obtain \eqref{Es2}. 
\end{proof}
\begin{proof}[Proof of Estimate $(\ref{Es3})$]
We have by Bony's decomposition in $x$ variable
\begin{multline*}
\left(\int_0^T \|e^{Rt}\Delta^h_q(v\,\partial_{z} u)_{\phi}\|^2_{L^2}\, dt \right)^{\frac{1}{2}}
 \lesssim  \left(\int_0^T \| e^{Rt}\Delta^h_q(T^h_v \partial_z u)_{\phi}\|^2_{L^2} dt \right)^{\frac{1}{2}} 
 \\
 + \left( \int_0^T \|e^{Rt}\Delta^h_q(T^h_{\partial_z u} v)_{\phi}\|^2_{L^2} dt \right)^{\frac{1}{2}}
 + \left( \int_0^T \|e^{Rt}\Delta^h_q(R^h(v,\partial_z u))_{\phi}\|^2_{L^2} dt \right)^{\frac{1}{2}}.
\end{multline*}
\noindent
$\bullet$ Estimate of $\left(\int_0^T \| e^{Rt}\Delta^h_q(T^h_v \partial_z u)_{\phi}\|^2_{L^2} dt \right)^{\frac{1}{2}} $.
\begin{align*} 
\int_0^T \| e^{Rt}\Delta^h_q(T^h_v \partial_z u)_{\phi}\|^2_{L^2} dt
& \leq \sum_{|q-q'|\leq 4} \int_0^T \|S_{q'-1}^h v^+_\phi\|^2_{L^{\infty}}\,\|e^{Rt}\Delta^h_{q'} \partial_z u_\phi\|^2_{L^2}\,dt
\\
& \lesssim \sum_{|q-q'|\leq 4} \int_0^T \|\partial_z u_\phi\|^2_{B^\frac{1}{2}} 2^{2q'} \|e^{Rt}\Delta^h_{q'} \partial_z u_\phi\|^2_{L^2} \notag
\\
 & \lesssim  \sum_{|q-q'|\leq 4}  \|\partial_z u_\phi\|^2_{L_T^{\infty}({B^\frac{1}{2}})} \, 
    \left(2^{2q'} \int_0^T \|e^{Rt}\Delta^h_{q'} \partial_z u_\phi\|^2_{L^2}\, dt \right) . \notag
\end{align*}
Similar to \eqref{Es1a} we get
\begin{align}\label{Es3a}
\left(\int_0^T \| e^{Rt}\Delta^h_q(T^h_v \partial_z u)_{\phi}\|^2_{L^2} dt \right)^{\frac{1}{2}}
\lesssim \tilde{d}_{q} \, 2^{-\frac{3q}{2}}\|\partial_z u_\phi\|_{\widetilde{L}_T^\infty(B^\frac{1}{2})}   
    \|e^{Rt} \partial_z u_\phi\|_{\widetilde{L}_T^2(B^\frac{5}{2})} , 
\end{align}
where $\tilde{d}_{q} $ is given by \eqref{dq tilde3}.

\noindent
$\bullet$ Estimate of $\left(\int_0^T \| e^{Rt}\Delta^h_q(T^h_{\partial_z u} v)_{\phi}\|^2_{L^2}\, dt \right)^{\frac{1}{2}}$. 
\begin{align*} 
\int_0^T \| e^{Rt}\Delta^h_q(T^h_{\partial_z u} v)_{\phi}\|^2_{L^2}\, dt
& \lesssim \sum_{|q-q'|\leq 4} \int_0^T \|S_{q'-1}^h 
  \partial_z  u^+_\phi\|^2_{L^2_z(L_x^{\infty})}\,\|e^{Rt}\Delta^h_{q'} v_\phi\|^2_{L_z^{\infty}(L_x^2)}\, dt 
 \\
& \lesssim \sum_{|q-q'|\leq 4} \int_0^T \|\partial_z u_\phi\|^2_{B^\frac{1}{2}} 2^{2q'} \|e^{Rt}\Delta^h_{q'} \partial_z u_\phi\|^2_{L^2}\, dt \notag
\\
 & \lesssim  \sum_{|q-q'|\leq 4}  \|\partial_z u_\phi\|^2_{L_T^{\infty}({B^\frac{1}{2}})} \, 
    \left(2^{2q'} \int_0^T \|e^{Rt}\Delta^h_{q'} \partial_z u_\phi\|^2_{L^2}\, dt \right). \notag
\end{align*}
Similar to \eqref{Es1a} we get
\begin{align}\label{Es3b}
\left(\int_0^T \| e^{Rt}\Delta^h_q(T^h_{\partial_z u} v)_{\phi}\|^2_{L^2}\, dt \right)^{\frac{1}{2}}
\lesssim \tilde{d}_{q} \, 2^{-\frac{3q}{2}}\|\partial_z u_\phi\|_{\widetilde{L}_T^\infty(B^\frac{1}{2})}   
    \|e^{Rt} \partial_z u_\phi\|_{\widetilde{L}_T^2(B^\frac{5}{2})} , 
\end{align}
where $\tilde{d}_{q} $ is given by \eqref{dq tilde3}.

\noindent
$\bullet$ Estimate of $ \| e^{Rt}\Delta^h_q(R^h(v,\partial_z u))_{\phi}\|_{L^2}$. 
\begin{align*}
 \int_0^T \| e^{Rt}\Delta^h_q(R^h(v,\partial_z u))_{\phi}\|^2_{L^2}\, dt
& \lesssim \sum_{q' \geq q-3} 
  \int_0^T \|\Delta^h_{q'} v_\phi\|^2_{L_z^{\infty}(L_x^2)} \, 
  \|e^{Rt}\widetilde{\Delta}^h_{q'} \partial_z  u_\phi\|^2_{L_z^2(L_x^{\infty})} \,dt
\\
& \lesssim \sum_{q' \geq q-3} \int_0^T
  \left( 2^{2q'} \|\Delta^h_{q'} \partial_z u_\phi\|^2_{L^2} \right) \, 
  \|\partial_z u_\phi\|^2_{B^\frac{1}{2}} \, dt \notag
\\
& \lesssim \sum_{q' \geq q-3} 2^{2q'}  \left( \int_0^T
     \|\Delta^h_{q'} \partial_z u_\phi\|^2_{L^2}\,dt \right) \, 
 \|\partial_z u_\phi\|^2_{L_T^{\infty}({B^\frac{1}{2}})} .  \notag  
\end{align*}
Similar to \eqref{Es1a} we get
\begin{align} \label{Es3c}
\left(\int_0^T \|e^{Rt}\Delta^h_q(R^h(v,\partial_z u))_{\phi}\|^2_{L^2}\, dt \right)^{\frac{1}{2}}
\lesssim \check{d}_{q} \, 2^{-\frac{3q}{2}}\|e^{Rt} \partial_z u_\phi\|_{\widetilde{L}_T^2(B^\frac{5}{2})}  
    \|e^{Rt} \partial_z u_\phi\|_{\widetilde{L}_T^\infty(B^\frac{1}{2})} , 
\end{align}
where $\check{d}_{q}$ is given by \eqref{dq check3}.

\noindent
$\bullet$ By summing the previous estimates \eqref{Es3a}, \eqref{Es3b} and \eqref{Es3c} we get 
\begin{align*}
\left(\int_0^T \|e^{Rt}\Delta^h_q(v\,\partial_{z} u)_{\phi}\|^2_{L^2}\, dt \right)^{\frac{1}{2}}
& \lesssim (2\tilde{d}_{q}+\check{d}_q) 
  2^{-\frac{3q}{2}} 
  \left(\|\partial_z u_\phi\|_{\widetilde{L}_T^\infty(B^\frac{1}{2})}   
    \|e^{Rt} \partial_z u_\phi\|_{\widetilde{L}_T^2(B^\frac{5}{2})} \right) . 
\end{align*}
Multiplying the previous inequality by $2^{\frac{3q}{2}}$ and taking the sum over $\mathbb{Z}$, 
we obtain \eqref{Es3}.
\end{proof}
\begin{proof}[Proof of Estimate $(\ref{Es4})$]
We have by Bony's decomposition in $x$ variable
\begin{multline*}
\left(\int_0^T \|e^{Rt}\Delta^h_q(v\, \partial^2_z u)_{\phi}\|^2_{L^2}\, dt \right)^{\frac{1}{2}}
 \lesssim  \left(\int_0^T \|e^{Rt}\Delta^h_q(T^h_v  \partial^2_z u)_{\phi}\|^2_{L^2} dt \right)^{\frac{1}{2}} 
 \\ 
 +  \left( \int_0^T \|e^{Rt}\Delta^h_q(T^h_{ \partial^2_z u} v)_{\phi}\|^2_{L^2} dt \right)^{\frac{1}{2}}
 + \left( \int_0^T \| e^{Rt}\Delta^h_q(R^h(v,\partial^2_z u))_{\phi}\|^2_{L^2} dt \right)^{\frac{1}{2}}.
\end{multline*}
By adapting the proof of Estimate \eqref{Es3a} we obtain
\begin{align*}
\left( \int_0^T \|e^{Rt}\Delta^h_q(T^h_v  \partial^2_z u)_{\phi}\|^2_{L^2}\, dt \right)^{\frac{1}{2}}
& \lesssim \sum_{|q'-q|\leq 4 } \left( \int_0^T \|S_{q'-1}^h v^+_\phi\|^2_{L^{\infty}}\,
\|\Delta^h_{q'} \partial^2_z u_\phi\|^2_{L^2} \,dt \right)^{\frac{1}{2}} 
 \\
& \lesssim \tilde{d}_{q} 2^{-\frac{3q}{2}}
\|\partial_z u_\phi\|_{\widetilde{L}_T^\infty(B^\frac{1}{2})} 
\|e^{Rt} \partial^2_z u_\phi\|_{\widetilde{L}_T^2(B^\frac{5}{2})} . 
\end{align*}
By adapting the proof of Estimate \eqref{Es3b} we obtain
\begin{align*}
 \left( \int_0^T \|e^{Rt}\Delta^h_q(T^h_{ \partial^2_z u} v)_{\phi}\|^2_{L^2} dt \right)^{\frac{1}{2}}
 & \lesssim \sum_{|q-q'|\leq 4} \left(\|S_{q'-1}^h e^{Rt}\partial^2_z  u^+_\phi\|^2_{L^2_z(L_x^{\infty})}\,
 \|Delta^h_{q'} v_\phi\|^2_{L_z^{\infty}(L_x^2)} \right)^{\frac{1}{2} }
 \\
& \lesssim \tilde{d}_{q}\,2^{-\frac{3q}{2}}
\|e^{Rt} \partial^2_z u_\phi\|_{\widetilde{L}_T^2(B^\frac{1}{2})}
\|e^{Rt} \partial_z u_\phi\|_{\widetilde{L}_T^\infty(B^\frac{1}{2})} . 
\end{align*}
By adapting the proof of Estimate \eqref{Es3c} we obtain
\begin{align*}
\left( \int_0^T \| e^{Rt}\Delta^h_q(R^h(v,\partial^2_z u))_{\phi}\|^2_{L^2} dt \right)^{\frac{1}{2}} 
& \lesssim  \sum_{q' \geq q-3} \left( \int_0^T
  \|\Delta^h_{q'} v_\phi\|^2_{L_z^{\infty}(L_x^2)} 
  2^{\frac{q'}{2}}\|e^{Rt}\widetilde{\Delta}^h_{q'} \partial^2_z  u_\phi\|^2_{L^2} \, dt \right)^{\frac{1}{2}}
\\
& \lesssim \check{d}_{q} 2^{-\frac{3q}{2}} 
  \|\partial_z u_\phi\|_{\widetilde{L}_T^\infty(B^\frac{5}{2})} 
  \|e^{Rt} \partial^2_z u_\phi\|_{\widetilde{L}_T^2(B^\frac{1}{2})}  . 
\end{align*}
Then by summing the previous estimates we get 
\begin{multline*}
\left(\int_0^T \|e^{Rt}\Delta^h_q(v\, \partial^2_z u)_{\phi}\|^2_{L^2}\, dt \right)^{\frac{1}{2}} 
\\ \lesssim 
 2^{-\frac{3q}{2}} 
 \left( (\tilde{d}_{q}+\check{d}_q) 
  \left(\|\partial_z u_\phi\|_{\widetilde{L}_T^\infty(B^\frac{5}{2})} 
  \|e^{Rt} \partial^2_z u_\phi\|_{\widetilde{L}_T^2(B^\frac{1}{2})} \right)
       + \tilde{d}_{q} \|\partial_z u_\phi\|_{\widetilde{L}_T^\infty(B^\frac{1}{2})} 
       \|e^{Rt} \partial^2_z u_\phi\|_{\widetilde{L}_T^2(B^\frac{5}{2})} \right). 
\end{multline*}
Multiplying  the previous inequality by $2^{\frac{3q}{2}}$ and taking the sum over $\mathbb{Z}$, 
we obtain \eqref{Es4} . 
\end{proof}
\begin{proof}[Proof of Estimate  $(\ref{Es5})$]
We have by Bony's decomposition in $x$ variable
\begin{multline*}
\| e^{Rt}\Delta^h_q(\partial_z u\,\partial_{x} \partial_z u)_{\phi}\|_{L^2} 
\\
\leq \| e^{Rt}\Delta^h_q(T^h_{\partial_z u} \partial_x \partial_z  u)_{\phi}\|_{L^2} +
\| e^{Rt}\Delta^h_q(T^h_{\partial_x \partial_z u} \partial_z u)_{\phi}\|_{L^2}+ \| e^{Rt}\Delta^h_q(R^h(\partial_z u,\partial_x \partial_z u))_{\phi}\|_{L^2}.
\end{multline*}
By adapting the proof of Estimate \eqref{Es1a} we get
\begin{align*}
\left(\int_0^T \| e^{Rt}\Delta^h_q(T^h_{\partial_z u} \partial_x \partial_z  u)_{\phi} \|^2_{L^2} dt \right)^{\frac{1}{2}}
& \lesssim \tilde{d}_{q} 2^{-\frac{3q}{2}} \|e^{Rt} \partial^2_z u_\phi\|_{\widetilde{L}_T^2(B^\frac{1}{2})}   \| \partial_z u_\phi\|_{\widetilde{L}_T^\infty(B^\frac{5}{2})}
\end{align*}
and
\begin{align*}
 \left( \int_0^T \| e^{Rt}\Delta^h_q(T^h_{\partial_x \partial_z u} \partial_z u)_{\phi}\|^2_{L^2} dt \right)^{\frac{1}{2}}
& \lesssim \tilde{d}_{q} 2^{-\frac{3q}{2}}\|e^{Rt} \partial^2_z u_\phi\|_{\widetilde{L}_T^2(B^\frac{1}{2})}   \| \partial_z u_\phi\|_{\widetilde{L}_T^\infty(B^\frac{5}{2})}. 
\end{align*}
By adapting the proof of Estimate \eqref{Es1c} we get
\begin{align*}
\left( \int_0^T \| e^{Rt}\Delta^h_q(R^h(\partial_z u,\partial_x \partial_z u))_{\phi}\|^2_{L^2} dt \right)^{\frac{1}{2}}  
 \lesssim \check{d}_{q} 2^{-\frac{3q}{2}}\|\partial_z u_\phi\|_{\widetilde{L}_T^\infty(B^\frac{1}{2})}  \|e^{Rt} \partial^2_z u_\phi\|_{\widetilde{L}_T^2(B^\frac{5}{2})}.
\end{align*}
Then by summing the previous estimates we get 
\begin{multline*}
 \left(\int_0^T \|e^{Rt}\Delta^h_q(\partial_z u\,\partial_{x} \partial_z u)_{\phi}\|^2_{L^2}\, dt \right)^{\frac{1}{2}} 
\\
 \lesssim 2^{-\frac{3q}{2}} \left(
 2\tilde{d}_q  \|e^{Rt} \partial^2_z u_\phi\|_{\widetilde{L}_T^2(B^\frac{1}{2})}   
 \| \partial_z u_\phi\|_{\widetilde{L}_T^\infty(B^\frac{5}{2})}+ \check{d}_{q} 
 \|\partial_z u_\phi\|_{\widetilde{L}_T^\infty(B^\frac{1}{2})}  
 \|e^{Rt} \partial^2_z u_\phi\|_{\widetilde{L}_T^2(B^\frac{5}{2})} \right) . 
\end{multline*}
Multiplying  the previous inequality by $2^{\frac{3q}{2}}$ and taking the sum over $\mathbb{Z}$, 
we obtain \eqref{Es5}. 
\end{proof}
\begin{proof}[Proof of Estimate $(\ref{Es6})$]
We have by Bony's decomposition in $x$ variable
\begin{align*}
\| e^{Rt}\Delta^h_q(\partial_x u\, \partial^2_z u)_{\phi}\|_{L^2} 
\leq 
  \| e^{Rt}\Delta^h_q(T^h_{\partial_x u} \partial^2_z u)_{\phi}\|_{L^2} 
+ \| e^{Rt}\Delta^h_q(T^h_{ \partial^2_z u} \partial_x u)_{\phi}\|_{L^2} 
+ \| e^{Rt}\Delta^h_q(R^h(\partial_x u,\partial^2_z u))_{\phi}\|_{L^2} .
\end{align*}
By adapting the proof of Estimate \eqref{Es1a} we get
\begin{align*}
\left( \int_0^T \| e^{Rt}\Delta^h_q(T^h_{\partial_x u}  \partial^2_z u)_{\phi}\|^2_{L^2} \,dt \right)^{\frac{1}{2}}
& \lesssim \sum_{|q'-q|\leq 4 } \left( \int_0^T 
\|S_{q'-1}^h \partial_xu^+_\phi\|^2_{L^{\infty}}\,
\|\Delta^h_{q'} \partial^2_z u_\phi\|^2_{L^2} \, dt \right)^{\frac{1}{2}} 
\\
& \lesssim \tilde{d}_{q,2}  2^{-\frac{3q}{2}}
\|\partial_z u_\phi\|_{\widetilde{L}_T^\infty(B^\frac{1}{2})}  
\|e^{Rt} \partial^2_z u_\phi\|_{\widetilde{L}_T^2(B^\frac{5}{2})} , 
\end{align*}
where $\tilde{d}_{q,2} = \sum_{|q-q'|\leq 4} d_{q'}(\partial_z^2 u_\phi,1)  2^{\frac{3}{2}(q-q')}$. 

\noindent
$\bullet$ Estimate of $\left( \int_0^T \| e^{Rt}\Delta^h_q(T^h_{\partial^2_z u} \partial_x u)_{\phi}\|_{L^2}\, dt \right)^{\frac{1}{2}}$. 
\begin{align*}
\left( \int_0^T \| e^{Rt}\Delta^h_q(T^h_{\partial^2_z u} \partial_x u)_{\phi}\|_{L^2}\, dt \right)^{\frac{1}{2}}
 & \lesssim \sum_{|q-q'|\leq 4} \left(\int_0^T \|S_{q'-1}^h e^{Rt} \partial^2_z  u^+_\phi\|^2_{L^2_z(L_x^{\infty})}\,
   \| \Delta^h_{q'} \partial_x u_\phi\|^2_{L_z^{\infty}(L_x^2)}\, dt \right)^{\frac{1}{2}}
 \\
& \lesssim \sum_{|q-q'|\leq 4} \left( \int_0^T \|e^{Rt} \partial^2_z u_\phi\|^2_{B^\frac{1}{2}} 2^{2q'} 
  \|\Delta^h_{q'} \partial_z u_\phi\|^2_{L^2}\, dt\right)^{\frac{1}{2}}
\\
& \lesssim \tilde{d}_{q} 2^{-\frac{3q}{2}} 
  \| e^{Rt} \partial^2_z u_\phi\|_{\widetilde{L}_T^2(B^\frac{1}{2})}  
  \| \partial_z u_\phi\|_{\widetilde{L}_T^\infty(B^\frac{5}{2})} . 
\end{align*}
By adapting the proof of Estimate \eqref{Es1c} we get
\begin{align*}
 \left( \int_0^T \| e^{Rt}\Delta^h_q(R^h(\partial_x u,\partial^2_z u))_{\phi}\|^2_{L^2}\, dt \right)^{\frac{1}{2}}
& \lesssim  \sum_{q' \geq q-3} \left( \int_0^T \|\Delta^h_{q'} \partial_x u_\phi\|^2_{L_z^{\infty}(L_x^2)} 
\|e^{Rt}\widetilde{\Delta}^h_{q'} \partial^2_z  u_\phi\|^2_{L^2_z (L^\infty_x)} \right)^{\frac{1}{2}} 
\\
& \lesssim \check{d}_{q} 2^{-\frac{3q}{2}} \|\partial_z u_\phi\|_{\widetilde{L}_T^\infty(B^\frac{5}{2})} 
\|e^{Rt} \partial^2_z u_\phi\|_{\widetilde{L}_T^2(B^\frac{1}{2})} . 
\end{align*}
Then by summing up the previous estimates, we get
\begin{multline*}
\left( \int_0^T \| e^{Rt}\Delta^h_q(\partial_x u\, \partial^2_z u)_{\phi}\|^2_{L^2} \right)^{\frac{1}{2}}
 \\
\lesssim 2^{-\frac{3q}{2}} \left(
 \tilde{d}_{q,2} \|\partial_z u_\phi\|_{\widetilde{L}_T^\infty(B^\frac{5}{2})} \|e^{Rt} \partial^2_z u_\phi\|_{\widetilde{L}_T^2(B^\frac{1}{2})} 
 + (\tilde{d}_{q}+\check{d}_{q}) \|\partial_z u_\phi\|_{\widetilde{L}_T^\infty(B^\frac{1}{2})}  \|e^{Rt} \partial^2_z u_\phi\|_{\widetilde{L}_T^2(B^\frac{5}{2})}\right). 
\end{multline*}
Multiplying  the previous inequality by $2^{\frac{3q}{2}}$ and taking the sum over $\mathbb{Z}$, 
we obtain \eqref{Es6}.
\end{proof}


\section{Appendix - Proof of the lemma \ref{lem:estim-Iq} }\label{app estim-Iq}

In this Appendix, we give a brief proof of estimates used to prove the uniqueness of the solution.

\noindent
$\bullet$ Estimate of $I_{1,q}$.
From Lemma  3.1 in \cite{PZZ2020} we obtain \eqref{I1q}.

\noindent
$\bullet$ Estimate of $I_{2,q}$.
From Lemma  3.1 in \cite{PZZ2020} we have
\begin{align*}
\abs{\int_0^t \psca{e^{Rt'}\Delta^h_q (u_1\,\partial_{x}U)_\Phi, e^{Rt'}\Delta^h_q U_\phi} dt'} 
\leq C 2^{-2qs}d^2_q\,\|e^{Rt'}U_\Phi\|^2_{\widetilde{L}^2_{T,\Theta'(t)}(B^{s+\frac{1}{2}})} . 
\end{align*}
By using integration by parts, and from Lemma  3.1 in \cite{PZZ2020} we have
\begin{align*}
& \abs {\int_0^t \psca{e^{Rt'}\Delta^h_q (u_1\,\partial_{x} \partial_z^2 U)_\Phi, e^{Rt'}\Delta^h_q U_\phi} dt'}
\\
&\quad\quad \leq  \abs {\int_0^t \psca{e^{Rt'}\Delta^h_q (\partial_z u_1\,\partial_{x} \partial_z U)_\Phi, e^{Rt'}\Delta^h_q U_\phi} dt'}+\abs {\int_0^t \psca{e^{Rt'}\Delta^h_q (u_1\,\partial_{x} \partial_z U)_\Phi, e^{Rt'}\Delta^h_q \partial_z U_\phi} dt'}
\\
&\quad\quad \leq  C 2^{-2qs}d^2_q\,\|e^{Rt'} \partial_z U_\Phi\|^2_{\widetilde{L}^2_{T,\Theta'(t)}(B^{s+\frac{1}{2}})} .
\end{align*}
By summing up the above estimates we obtain \eqref{I2q}. 

\noindent
$\bullet$ Estimate of $I_{3,q}$.
By integration by parts in $z$ variable and from Lemma  3.1 in \cite{PZZ2020}, we obtain \eqref{I3q}.

\noindent
$\bullet$ Estimate of $I_{4,q}$.
It follows from the proof of Estimate (5.14) in \cite{PZZ2020} that
\begin{align}\label{I_1}
\int_0^t \abs {\psca{e^{Rt'}\Delta^h_q(V\,\partial_z u_2)_{\Phi}, e^{Rt'} \Delta^h_q U_\Phi}} dt' 
\lesssim 2^{-2qs}d^2_q\,\|e^{Rt'} \partial_z U_\Phi\|^2_{\widetilde{L}^2_{T,\Theta'(t)}(B^{s+\frac{1}{2}})} .
\end{align}
By integration by parts
\begin{align*}
\int_0^t & \abs{\psca{e^{Rt'}\Delta^h_q(V  \partial_z^3 u_2)_{\Phi}, e^{Rt'} \,\Delta^h_q U_\Phi}}\, dt'
\\
\leq &\int_0^t \abs{\psca{e^{Rt'}\Delta^h_q(V \partial^2_z u_2)_{\Phi}, e^{Rt'} \,\Delta^h_q \partial_z U_\Phi}}\, dt'
   + \int_0^t \abs{\psca{e^{Rt'}\Delta^h_q(\partial_x U \partial^2_zu_2)_{\Phi}, e^{Rt'} \,\Delta^h_q U_\Phi}}\, dt' \\
\leq  & \int_0^t \abs{\psca{e^{Rt'}\Delta^h_q(V \partial^2_z u_2)_{\Phi}, e^{Rt'} \,\Delta^h_q \partial_z U_\Phi}}\, dt'
  + \int_0^t \abs{\psca{e^{Rt'}\Delta^h_q(\partial_zu_2 \partial_x U)_{\Phi}, e^{Rt'} \,\Delta^h_q\partial_z U_\Phi}}\, dt'\\
&+\int_0^t \abs{\psca{e^{Rt'}\Delta^h_q(\partial_zu_2 \partial_x \partial_z U)_{\Phi}, e^{Rt'} \,\Delta^h_q U_\Phi}}\, dt'.
\end{align*}
Following the prove of Lemma 3.2 in \cite{PZZ2020} we deduce for any $s \in ]0,1[$ that
\begin{align*}
\int_0^t \abs{\psca{e^{Rt'}\Delta^h_q(V \partial^2_z u_2)_{\Phi}, e^{Rt'} \,\Delta^h_q \partial_z U_\Phi}}\, dt'
\lesssim & 2^{-2qs}d^2_q \|e^{Rt'} \partial_z U_\Phi\|^2_{\widetilde{L}^2_{T,\Theta'(t)}(B^{s+\frac{1}{2}})} . 
\end{align*}
From Lemma  3.1 in \cite{PZZ2020} we deduce that
\begin{align*}
\int_0^t &\abs{\psca{e^{Rt'}\Delta^h_q(\partial_zu_2 \partial_x U)_{\Phi}, e^{Rt'} \,\Delta^h_q\partial_z U_\Phi}}\, dt'+\int_0^t \abs{\psca{e^{Rt'}\Delta^h_q(\partial_zu_2 \partial_x \partial_z U)_{\Phi}, e^{Rt'} \,\Delta^h_q U_\Phi}}\, dt'\\
&\lesssim  2^{-2qs}d^2_q \|e^{Rt'} \partial_z U_\Phi\|^2_{\widetilde{L}^2_{T,\Theta'(t)}(B^{s+\frac{1}{2}})} .
\end{align*}
Then
\begin{align} \label{I_2}
\int_0^t& \abs{\psca{e^{Rt'}\Delta^h_q(V  \partial_z^3 u_2)_{\Phi}, e^{Rt'} \,\Delta^h_q U_\Phi}}\, dt'
\lesssim 2^{-2qs}d^2_q \|e^{Rt'} \partial_z U_\Phi\|^2_{\widetilde{L}^2_{T,\Theta'(t)}(B^{s+\frac{1}{2}})} .
\end{align}
By summing the estimates \eqref{I_1} and \eqref{I_2} we obtain \eqref{I4q}.

\noindent
$\bullet$ Estimate of $I_{5,q}$.
Following the proof of Estimate (5.13) in \cite{PZZ2020}, we deduce that
\begin{align*}
\int_0^t& \abs{\psca{e^{Rt'}\Delta^h_q (v_1\partial_z U)_\Phi,e^{Rt'}\Delta^h_q U_\Phi} dt'}
\\
 &\quad\quad\quad \leq C 2^{-2qs}d^2_q\,\| u_{1_{\Phi}} \|_{L^{\infty}(B^{\frac{3}{2}})}^{\frac{1}{2}}  
\,\|e^{Rt'} \partial_z U_\Phi\|_{\widetilde{L}^2_{T}(B^{\frac{1}{2}})}\|e^{Rt'} \partial_z U_\Phi\|^2_{\widetilde{L}^2_{T,\Theta'(t)}(B^{s+\frac{1}{2}})} .
\end{align*}
By integration by parts in $z$ variable, we have
\begin{align*}
&\abs{\int_0^t \psca{e^{Rt'}\Delta^h_q (v_1\partial^3_z U)_\Phi,e^{Rt'}\Delta^h_q U_\Phi} dt'} 
\\
&\quad  \leq  \abs{\int_0^t \psca{e^{Rt'}\Delta^h_q (v_1\partial^2_z U)_\Phi,e^{Rt'}\Delta^h_q \partial_z U_\Phi} dt'}
+\abs{\int_0^t \psca{e^{Rt'}\Delta^h_q (\partial_x u_1\partial^2_z U)_\Phi,e^{Rt'}\Delta^h_q U_\Phi} dt'} 
\\
&\quad  \leq C 2^{-2qs}d^2_q\,\| u_{1_{\Phi}} \|_{L^{\infty}(B^{\frac{3}{2}})}^{\frac{1}{2}}  
\,\|e^{Rt'} \partial^2_z U_\Phi\|_{\widetilde{L}^2_{T}(B^{\frac{1}{2}})}\|e^{Rt'} \partial_z U_\Phi\|^2_{\widetilde{L}^2_{T,\Theta'(t)}(B^{s+\frac{1}{2}})} .
\end{align*}
This conclude the proof of $\eqref{I5q}$. 

\noindent
$\bullet$ Estimate of $I_{6,q}$.
Following the proof of Estimate (5.13) in \cite{PZZ2020}, we get \eqref{I6q}.

\noindent
$\bullet$ Estimate of $I_{7,q}$.
Following the proof of Estimate (5.11) in \cite{PZZ2020}, we obtain \eqref{I7q}.

\noindent
$\bullet$ Estimate of $I_{8,q}$.
Following the proof of Estimate (5.11) in \cite{PZZ2020}, we deduce that
\begin{align*}
\int_0^t& \abs{\psca{e^{Rt'}\Delta^h_q(U\partial_x u_2)_{\Phi}, e^{Rt'} \,\Delta^h_q U_\Phi}}\, dt'
\\
&\quad \lesssim 
  2^{-2qs} \, d^2_q \|e^{Rt'}  U_\Phi\|_{\widetilde{L}^2_{T,\Theta'(t)}(B^{s+\frac{1}{2}})} 
  \left(\|e^{Rt'} U_\Phi\|_{\widetilde{L}^2_{T,\Theta'(t)}(B^{s+\frac{1}{2}})} 
    + \| u_{2_{\Theta}} \|_{L^{\infty}(B^{\frac{3}{2}})}^{\frac{1}{2}} \,
     \|e^{Rt'} \partial_z U_\Phi\|_{\widetilde{L}^2_{T}(B^{s+\frac{1}{2}})} \right). 
\end{align*}
By integration by parts
\begin{align*}
\int_0^t  & \abs{\psca{e^{Rt'}\Delta^h_q(U\partial_x \partial_z^2 u_2)_{\Phi}, e^{Rt'} \,\Delta^h_q U_\Phi}}\, dt'
\\
\leq & \int_0^t \abs{\psca{e^{Rt'}\Delta^h_q(U\partial_x \partial_z u_2)_{\Phi}, e^{Rt'} \,\Delta^h_q \partial_z U_\Phi}}\, dt'+\int_0^t \abs{\psca{e^{Rt'}\Delta^h_q(\partial_z U\partial_x \partial_zu_2)_{\Phi}, e^{Rt'} \,\Delta^h_q U_\Phi}}\, dt' \\
\lesssim 
& 2^{-2qs}d^2_q \|e^{Rt'} \partial_z U_\Phi\|_{\widetilde{L}^2_{T,\Theta'(t)}(B^{s+\frac{1}{2}})} 
\\
&\times\left(\|e^{Rt'} \partial_z U_\Phi\|_{\widetilde{L}^2_{T,\Theta'(t)}(B^{s+\frac{1}{2}})} 
                  + \| \partial_zu_{2_{\Theta}} \|_{L^{\infty}(B^{\frac{3}{2}})}^{\frac{1}{2}}\,
                    \|e^{Rt'} \partial_z U_\Phi\|_{\widetilde{L}^2_{T}(B^{s+\frac{1}{2}})} \right).
\end{align*}
Then by summing up the previous estimates, we get \eqref{I8q}.



\begin{thebibliography}{99}
	
\bibitem{AWXY2015}  R. Alexandre, Y. Wang, C.-J. Xu and T. Yang, Well-posedness of The Prandtl Equation in Sobolev Spaces, J. Amer. Math. Soc. 28 (2015), 745-784.

\bibitem{BCD2011} H. Bahouri, J-Y. Chemin and R. Danchin, 
Fourier  Analysis  and  Nonlinear  Partial  Differential  Equations, 
Grundlehren der Mathematischen Wissenschaften 343 (2011) Springer, Heidelberg. 
%
\bibitem{B94} A. J. Bourgeois and J. T. Beale,  
Validity of the quasigeostrophic model for large-scale flow in the atmosphere and ocean, 
SIAM J. Math. Anal. 25 (1994) 1023–68.
%
%
\bibitem{Bu2002} A. V. Busuioc, 
Sur les \'equations $\alpha$ Navier-Stokes dans un ouvert born\'e, 
C. R. Acad. Sci. Paris, S\'erie I 334 (2002) 823-826.
%
\bibitem{CH93}  R. Camassa and D. D. Holm, 
An integrable shallow water equation with peaked solitons, 
Phys. Rev. Lett. 71:11 (1993) 1661-1664.
%
\bibitem{Chemin2004} J.-Y. Chemin, 
Le syst\`eme de Navier-Stokes incompressible soixante dix ans apr\`es Jean Leray, 
\textit{Actes des Journ\'ees Math\'ematique \`a la M\'emoire de Jean Leray}, 
S\'eminaire \& Congr\`es, 9, Soc. Math. France, Paris (2004) 99-123.
%
\bibitem {CGP2011}
J.-Y. Chemin, I. Gallagher and M. Paicu, 
Global regularity for some classes of large solutions to the Navier-Stokes equations, 
Annals of Mathematics 173 (2011) 983–1012.
%
%
\bibitem{CL1992} J.-Y. Chemin and N. Lerner, 
Flot de champs de vecteurs non Lipschitziens et \'equations de Navier-Stokes, 
Journal of Differential Equations 121 (1992) 314-328. 
%
%
\bibitem {CFHOTW98} S. Chen, C. Foias, D. D. Holm, E. Olson, E. S. Titi and S. Wynne,  
The Camassa-Holm equations as a closure model for turbulent channel and pipe flow,
Phys. Rev. Lett. 81 (1998) 5338–5341.
%
%
%
\bibitem {C2004} A. Cheskidov, 
Boundary layer for the Navier-Stokes-alpha model of fluid turbulence, 
Arch. Ration. Mech. Anal. 172:3 (2004) 333-362. 
%
%
%
\bibitem{FHT2002} C. Foias, D.D. Holm and E.S. Titi, 
The three-dimensional viscous Camassa-Holm equations and their relation to the Navier-Stokes equation and turbulence theory, 
J. Dynamics Differential Equations 14:1 (2002) 1-35. 
%
%
\bibitem {GMV2019} D. G\'erard-Varet, N. Masmoudi and V. Vicol, Well-posedness of the hydrostatic Navier-Stokes equations, Anal. PDE 13 (2020), no. 5, 1417–1455.
%
\bibitem{Gh99} S. Ghosal, 
Mathematical and physical constraints on large-eddy simulation of turbulence, 
AIAA J. 37:4 (1999) 425-433.
%
%
\bibitem{HMR98} D. D. Holm, J. E. Marsden and T. Ratiu, 
Euler-Poincar\'e models of ideal fluids with nonlinear dispersion, 
Phys. Rev. Lett. 80:19 (1998) 4173-4177.
%
\bibitem{K99} S. Kouranbaeva, 
The Camassa–Holm equation as a geodesic flow on the diffeomorphism group, 
J. Math. Phys. 40:2 (1999) 857-868.
%
\bibitem{KTVZ2011} I. Kukavica, R. Temam, V. Vicol and M. Ziane, 
Local existence and uniqueness for the hydrostatic Euler equations on a bounded domain, 
J. Differential Equations 250:3 (2011) 1719-1746.
%
\bibitem{LTW92} J.-L. Lions, R. Temam and S. Wang, On the equations of the large-scale ocean, Nonlinearity 5 (1992) 1007-1053.
%
\bibitem{MS2001} J.E. Marsden and S. Shkoller, 
Global well-posedness for the LANS-$\alpha$ equations on bounded domains, 
Philos. Trans.Roy. Soc. London Ser. A 359 (2001) 1449-1468.

\bibitem{MW2015} N.~Masmoudi and T.~K. Wong, Local-in-time existence and uniqueness of solutions to the Prandtl equations by energy methods, Comm. Pure Appl. Math. 68 (2015) 1683-1741.
%
%
%
\bibitem{M98} G. Misiolek, 
A shallow water equation as a geodesic flow on the Bott–Virasoro group, 
J. Geom. Phys. 24:3 (1998) 203-208.
%
%
\bibitem{PZZ2020} M. Paicu, P. Zhang and Z. Zhang, 
On the hydrostatic approximation  of the Navier-Stokes equations in a thin strip,  Adv. Math. 372 (2020).

\bibitem{Pe1987} J. Pedlosky, 
Geophysical Fluid Dynamics, Springer-Verlag New York (1987).
%
\bibitem{R2009} M. Renardy,  
Ill-posedness of the hydrostatic Euler and Navier-Stokes equations, 
Arch. Ration. Mech. Anal. 194:3 (2009) 877-886.

\bibitem{samm} M.  Sammartino and R. E. Caflisch,  Zero viscosity limit for analytic solutions of the Navier-Stokes equations on a half-space, I. Existence for Euler and Prandtl equations, Comm. Math. Phys. 192 (1998) 433-461; II. Construction of the Navier-Stokes solution, Comm. Math. Phys. 192 (1998) 463-491.


\end{thebibliography}
\end{document}